\documentclass[10pt]{article}
\usepackage[english]{babel}
\usepackage{amsmath}
\usepackage{amssymb}
\usepackage{amsthm}
\newcommand{\acks}{\bigskip\par\noindent{\bf Acknowledgement}\par}
\usepackage{graphicx}
\usepackage{epstopdf}

    \newtheorem{lemma}{Lemma}[section]
    
    \newtheorem{theorem}[lemma]{Theorem}
    \newtheorem{definition}[lemma]{Definition}
     \newtheorem{corollary}[lemma]{Corollary}
    \newtheorem{remark}[lemma]{Remark}

\newcommand{\no}{{\nonumber}}
\newcommand{\ov}{\overline}
\newcommand{\bear}{\begin{array}}
\newcommand{\enar}{\end{array}}
\newcommand{\beq}{\begin{equation}}
\newcommand{\eeq}{\end{equation}}
\newcommand{\beqn}{\begin{eqnarray}}
\newcommand{\eeqn}{\end{eqnarray}}
\newcommand{\beit}{\begin{itemize}}
\newcommand{\eeit}{\end{itemize}}

\renewcommand{\d}{\,{\rm d}}

\newcommand{\ee}{{\rm e}}

\newcommand{\rsp}{{\bf R}}
\newcommand{\csp}{{\bf C}}
\newcommand{\nsp}{{\bf N}}

\renewcommand{\div}{{\rm div}\:}

\newcommand{\ds}{\displaystyle}

\newcommand{\e}{\epsilon}
\newcommand{\ve}{\varepsilon}
\newcommand{\g}{\gamma}
\renewcommand{\l}{\lambda}
\newcommand{\dl}{\delta}
\newcommand{\Dl}{\Delta}
\newcommand{\s}{\sigma}

\newcommand{\G}{\Gamma}
\newcommand{\om}{\omega}
\newcommand{\Om}{\Omega}
\renewcommand{\a}{\alpha}
\renewcommand{\b}{\beta}

\newcommand{\what}{\widehat}

\newcommand{\wtil}[1]{\widetilde{#1}}
\newcommand{\pn}{\par \noindent}
\newcommand{\med}{\medskip}
\newcommand{\qq}{\qquad}
\newcommand{\q}{\quad}
\newcommand{\hookto}{\hookrightarrow}

\title{{$L^2$ estimates for the eigenfunctions}
{corresponding\\ to real eigenvalues of the Tricomi operator}}
\author{{\rm Alberto Favaron}\footnote
{Dipartimento di Matematica ``F. Brioschi'', Politecnico di Milano,
via Bonardi 9, 20133 Milano, Italy. Email: alberto.favaron@polimi.it}}
\date{}
\begin{document}
\maketitle
\pn
{\bf Abstract.}
We introduce a family ${\cal F}$ of normal Tricomi domains 
$\Om_{\a,\b}$, $\a>0>\b$, and we show that its elements
are $D$-star-shaped with respect to the vector field
$D=-3x\partial_x-2y\partial_y$ if and only if $\a\ge1/2$. 
Provided that the underlying domain $\Om$ belongs to ${\cal F}$ 
for some $\a\ge1/2$, in Theorem \ref{thm4.20} we then establish $L^2$ estimates 
for the eigenfunctions corresponding to real eigenvalues 
of the Tricomi operator. In particular, our result highlights
a dependency of these estimates on the values 
of $\a$ and $\b$ and the parabolic diameter of $\Om$.
\med\pn
\med\pn
\section{Introduction}\label{Sec1}
\setcounter{equation}{0}
In this paper we deal with the problem of establishing $L^2$ estimates
for the eigenfunctions corresponding to real eigenvalues of the Tricomi problem,
i.e. the nontrivial solutions to
\beqn\label{1.1}
&&\hskip -1truecm
\left\{\!\!
\begin{array}{lll}
Tu=\l u\q {\rm in}\ \Om,
\\[2mm]
u=0\q{\rm on}\ AC\cup\s,
\end{array}
\right.
\q\l\in\rsp,
\qq
\begin{array}{lll}
\includegraphics[width=1.3in]{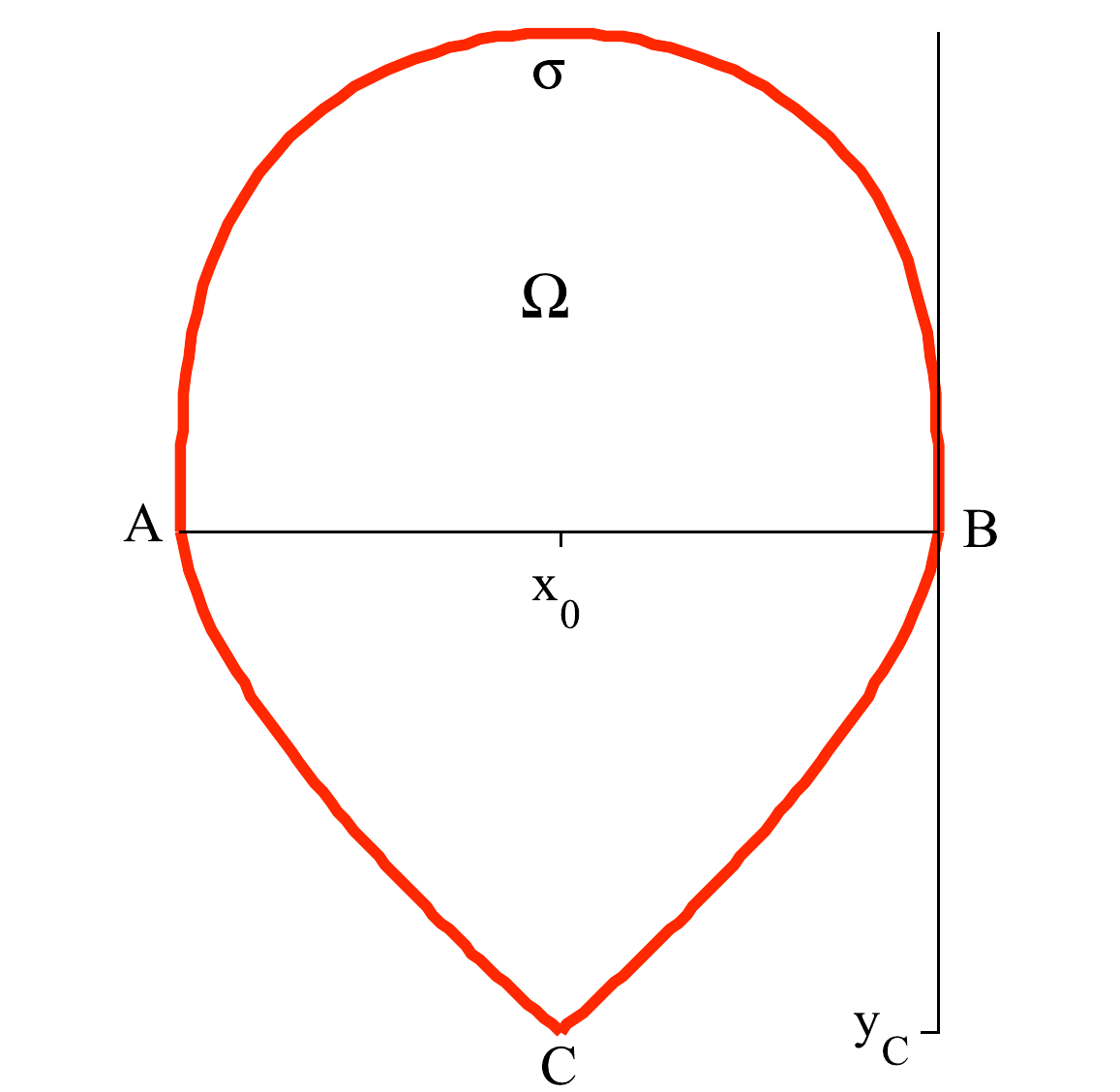}
\end{array}
\eeqn
where $T=-y\partial_x^2-\partial_y^2$ is the Tricomi operator on $\rsp^2$.
Here $\Om$ is a {\it Tricomi domain}; that is,
a simply connected bounded region of the plane
whose boundary $\partial\Om$ consists of
the elliptic arc $\s$ joining $A=(2x_0,0)$, $x_0<0$, to $B=(0,0)$
in the region $y>0$ and the two characteristics 
$AC$ and $BC$ for $T$ which lie in the half-plane $y\le 0$ 
and meet at the point $C=(x_0,y_C)$, $y_C<0$
(see Section \ref{Sec2} for a precise description).

Due to its physical importance, which derives from
its relations with the theory 
of two-dimensional transonic fluid flows first observed in \cite{Fr}, 
the literature concerning the question of the unique solvability
and the research of the Green's function 
for the underlying Tricomi problem (\ref{1.1}), 
with $\l u$ being replaced by $h\in L^2(\Om)$,
is nowadays very wide. See, for instance, the papers
\cite{Ag}, \cite{BN1}--\cite{BN3}, \cite{Di}, \cite{GB}, \cite{LP2}, \cite{Pr}
and the references therein.

On the contrary, only in quite recent times
there has been a growing interest towards a development 
of a clear spectral theory for the Tricomi operator; an interest
mainly motivated by the perspectives of making substantial 
progresses in the study of associated nonlinear problems,
(see \cite{Ga}, \cite{LP3}, \cite{LP4}, \cite{Mo1} and \cite{Mo2}).
The main results in this direction are probably those in
\cite{LP3} and \cite{LP4} where, 
provided that $\Om$ is {\it normal} in the sense that 
the elliptic arc $\s$ is perpendicular to the $x$-axis at the boundary 
points $A$ and $B$, it is shown
that a principal eigenvalue $\l_0>0$ exists such
that all the other real eigenvalues, if any, belong to $(\l_0,+\infty)$.
Employing the linear solvability theory combined with nonlinear analysis tools, 
such a spectral information is then exploited in \cite{LP4}
to derive existence and uniqueness for semilinear Tricomi problems.

Differently from \cite{LP4}, it is our aim, here,
to use the informations on the real spectrum of $T$ to show that,
if $\s=\s_{\a,\b}$ is given explicitly as the graph of the function
\beqn\label{1.2}
&&\hskip -0truecm
g_{\a,\b}(x)=\Big[\frac{(\a+1)|\b|x(2x_0-x)}{2}\Big]^{1/(\a+1)},
\q x\in[2x_0,0],\ \a>0>\b,
\eeqn
and if $u$ is an eigenfunction corresponding to $\l\in[\l_0,+\infty)$ 
enjoying some further regularity on the subset $\g_1\cup\g_2=BC\cup\s$ 
of $\partial\Om$, then the norm $\|u\|_{L^2(\Om)}$ is upper bounded 
by $\l^{-1/2}$ times a quantity depending on $\|u\|_{L^2(BC)}$,  
$\||y|^{1/2}u_x\|_{L^2(\g_j)}$, $\|u_y\|_{L^2(\g_j)}$, $j=1,2$, and 
the triplet $(\a,\b,x_0)$. 
Of course, here the mentioned results about the existence of $\l_0$
applies since, for construction, the curves $\s_{\a,\b}$
are perpendicular to the $x$-axis at the points $A$ and $B$ (see Section \ref{Sec4}).
Our $L^2$ eigenfunction bounds come out from an application 
to problem (\ref{1.1}) of the Poho\v zaev-type identity derived in \cite{LP5} for the
more general semilinear problem 
\beqn\label{1.3}
&&\hskip -1truecm
\left\{\!\!
\begin{array}{lll}
Tu=f(u)\q {\rm in}\ \Om,
\\[2mm]
u=0\q{\rm on}\ AC\cup\s,
\end{array}
\right.
\eeqn
where $f\in C^0(\rsp)$, and then estimating the right-hand side of
such an identity taking advantage from the fact that 
$\s$ belongs to the family of graphs $\s_{\a,\b}$, $\a>0>\b$.
We stress that this choice for $\s$ 
is motivated by two essential reasons. 
At first, if $\a\ge1/2$, then it makes each domain $\Om_{\a,\b}$, 
$\partial\Om_{\a,\b}=AC\cup BC\cup\s_{\a,\b}$,
a concrete example of domain $D$-star-shaped with respect 
to the vector field $D=-3x\partial_x-2y\partial_y$, a notion introduced
in \cite{LP5} only from an abstract point of view. 
In a certain sense, since for the choice $(\a,\b)=(2,-3/2)$ 
we get back the famous normal Tricomi curve, 
this also shows that the initial intuition of Tricomi of considering
the graph of $g_{2,-3/2}$ as the elliptic part of $\partial\Om$ (see \cite{Tr}) 
was correct, even though he was unaware of the notion of $D$-star-shapedness.
Secondly, it allows us to compute exactly the unit outer normal to $\s$ 
entering in the right-hand side of the quoted Poho\v zaev identity 
and hence to derive explicit formulae for the constants depending 
on $(\a,\b,x_0)$ in our estimates.
In particular, our computations exhibit the unexpected fact that
to each fixed pair $(\a,\b)$, $\a\ge 1/2$, $\a\neq1$, 
there corresponds a ``critical" value of $x_0$, 
in the sense that our constants change according 
to the fact that the half length $|x_0|$ of the parabolic segment $AB$
of $\Om$ is greater or not than a precise quantity depending on $\a$ and $\b$. 
As we shall see in Section \ref{Sec4}, 
this quantity takes essentially three different forms according
to the cases $\a=1/2$, $\a\in(1/2,1)$ or $\a>1$.
On the contrary, if $\a=1$, then the role of ``critical" value is played 
by the value $\b=-1$ which corresponds to the case when
the length of the outer normal vector to $\s_{1,\b}$ is constant equal to $|x_0|$.

It is also worth to observe that $L^2$--$L^p$-bounds for spectral projections
onto eigenspaces, as those derived in \cite{KR} and \cite{KT} for the
twisted Laplacian and the Hermite operator, respectively, 
are still lacking for the Tricomi operator. It thus seems to us
that our $L^2$ eigenfunction bounds may represent a first step in this direction
as well as in the research of some asymptotic estimate for the 
real eigenvalues of $T$.

The article is organized as follows.
In Section \ref{Sec2} we define the weighted Sobolev spaces
$\wtil W_{AC\cup\s}^1(\Om)$ and we give an overview 
of the linear solvability theory for the Tricomi problem 
developed in \cite{LP2} when $\Om$ is a normal Tricomi domain. 
This yields also to recall the main results of spectral theory for the 
Tricomi operator established in \cite{LP3} and \cite{LP4}.

Section \ref{Sec3} is devoted to introduce the notion
of $D$-star-shapedness, $D=-3x\partial_x-2y\partial_y$, 
and the Poho\v zaev identity of \cite{LP5} for the semilinear problem (\ref{1.3}). 
Moreover, recalling the basic symmetry groups 
that generate conservation laws for problem (\ref{1.3}) we are naturally led
to introduce the normal Tricomi curve $g_T=g_{2,-3/2}$ which constitutes 
the prototype for the construction of the functions $g_{\a,\b}$ 
defined by (\ref{1.2}).

Section \ref{Sec4} is the core of the paper. At first, we construct
the functions $g_{\a,\b}$ so that the domains $\Om_{\a,\b}$
having boundary $\partial\Om_{\a,\b}=AC\cup BC\cup\s_{\a,\b}$,
represent a family of normal Tricomi domains. Then, in Lemma \ref{lem4.3}
we show that $\Om_{\a,\b}$ is $D$-star-shaped 
in the sense of Section \ref{Sec3} if and only if $\a\ge1/2$.
We then prove the three preliminary Lemmas \ref{lem4.4}, \ref{lem4.5}
and \ref{lem4.18} which supply estimates for the line integrals
on the right-hand side of the Poho\v zaev identity.
Finally, combining the quoted lemmas with the fact that when $f(u)=\l u$, 
$\l\in[\l_0,+\infty)$, the left-hand side of the Poho\v zaev identity 
reduces to $4\l\|u\|_{L^2(\Om)}^2$,
in Theorem \ref{thm4.20} we prove our $L^2$ eigenfuctions bounds.

In Section \ref{Sec5} we give the proofs 
of the technical Lemmas \ref{lem4.8} and \ref{lem4.13} 
and Corollaries \ref{cor4.10} and \ref{cor4.16}, 
which are basic for proving Lemma \ref{lem4.18}.
In particular, Lemma \ref{lem4.8} provides the necessary estimates
on the modulus $h_{\a,\b}(x)$ of the normal vector to $\s_{\a,\b}$ 
at the point $(x,g_{\a,\b}(x))$ (see formulae (\ref{4.11}) and (\ref{4.12}))
and highlights their dependence on the values of $\a$, $\b$ and $x_0$.
Such estimates are then used in Corollary \ref{cor4.10}, 
Lemma \ref{lem4.13} and Corollary \ref{cor4.16} 
to deduce upper and lower bounds of two functions $\Theta_{\a,\b}(x)$
and $\Psi_{\a,\b}(x)$ entering the proof of Lemma \ref{lem4.18}.
Notice that, although $\Theta_{\a,\b}$ and $\Psi_{\a,\b}$ 
depend on one single real variable,
they both elude the standard methods of calculus for
finding greatest and least values, due to the computational
difficulty of locating their stationary points 
(see Remarks \ref{rem5.2} and \ref{rem5.3}).

We conclude the paper in Section \ref{Sec6} with some remarks
on the regularity assumptions of Theorem \ref{thm4.20}
and the possible links between our estimate 
and the open problem of finding, if any, eigenvalue asymptotics
for the Tricomi operator. 
\section{The Tricomi problem}\label{Sec2}
\setcounter{equation}{0}
The Tricomi operator $T$ in two independent variable $x$ and $y$ 
is the second order linear partial differential operator
\beqn\label{2.1}
&&\hskip -1truecm
T=-y\partial_x^2-\partial_y^2,
\eeqn
which is {\it elliptic} in the half-plane $y>0$, {\it parabolic} along the $x$ axis,
and {\it hyperbolic} in half-plane $y<0$. 
A subset $\Om\subset\rsp^2$  is said a {\it Tricomi domain} for  $T$ 
if $\Om$ is an open, bounded, simply connected set of $\rsp^2$ with
$C^1$ piecewise boundary $\partial\Om=AC\cup BC\cup\s$,
where $AC$ and $BC$ are the characteristic of negative and positive slopes
respectively issuing from the points $A=(2x_0,0)$ and $B=(0,0)$, $x_0<0$,
and meeting at the point $C=(x_C,y_C)$ in the hyperbolic region $y<0$.
The curve $\s$ is instead piecewise $C^1$ simple, joining $A$ to $B$ in
the elliptic region $y>0$. Of course, one has the explicit representation
\beqn\label{2.2}
&&\hskip -2truecm
AC=\{(x,y)\in\rsp^2:\ y_C\le y\le 0,\ 3(x-2x_0)=2(-y)^{3/2}\},
\\[2mm]
\label{2.3}
&&\hskip -2truecm
BC=\{(x,y)\in\rsp^2:\ y_C\le y\le 0,\ 3x=-2(-y)^{3/2}\},
\eeqn
so that $C=(x_0, -(3|x_0|/2)^{2/3})$. Due to the parabolic character of $T$ along 
the $x$ axis, the segment $AB=\{(x,y)\in\rsp^2:2x_0< x< 0, y=0\}$ is called the
parabolic segment of $\Om$, and its length $|AB|=2|x_0|$ is called the parabolic 
diameter of $\Om$.

For a connected subset $\G$ of $\partial\Om$
consider the following spaces of smooth real valued functions
\beqn\no
C_{0,\G}^\infty(\ov\Om;\rsp)
=\{\psi\in C^\infty(\ov\Om;\rsp): \psi\equiv 0\ {\rm on}\ N_\e\G\ {\rm for\ some}\ \e>0\},
\eeqn
where $N_\e\G=
\{(x,y)\in\Om:{\rm dist}((x,y);\G)=\inf_{(\wtil x,\wtil y)\in\G}|(x,y)-(\wtil x,\wtil y)|<\e\}$
and $C^\infty(\ov\Om;\rsp)$ denotes the set of
all functions from $\Om$ to $\rsp$ whose derivatives of any order
are continuos in $\Om$ and admit continuos extension up to  the boundary
$\partial\Om$. To simplify notations, from now on, we shall always write 
$C_{0,\G}^\infty(\ov\Om)$ in place of $C_{0,\G}^\infty(\ov\Om;\rsp)$.
Then, denote by
$\wtil W_\G^1(\Om)$ the weighted Sobolev space obtained as closure of 
$C_{0,\G}^\infty(\ov\Om)$ with respect to the norm
\beqn\no
&&\hskip -2truecm
\|\psi\|_{\wtil W_\G^1(\Om)}^2=\|\psi\|_{\wtil W^{1,2}(\Om)}^2
=\int_\Om(|y|\psi_x^2+\psi_y^2+\psi^2)\d x\d y.
\eeqn
Finally, the dual space $\wtil W_\G^{-1}(\Om)$ of $\wtil W_\G^1(\Om)$
is chararacterized as the norm closure of $L^2(\Om)$ with respect to the norm 
$\|w\|_{\wtil W_\G^{-1}(\Om)}=
\sup_{\|\psi\|_{\wtil W_\G^1(\Om)}=1}|(w,\psi)_2|$,
where $(\cdot,\cdot)_2$ is the standard inner (real) product of $L^2(\Om)$.
Obviously, $\wtil W_\G^1(\Om)\subset L^2(\Om)\subset \wtil W_\G^{-1}(\Om)$.
Moreover (see \cite[p. 538]{LP2}), 
using the definition of the $\wtil W_\G^{-1}(\Om)$-norm 
it is easy to show that there exist positive constants $c_j$, $j=1,2$, such that
\beqn\label{2.4}
&&\hskip -1truecm
\|Tu\|_{\wtil W_{BC\cup\s}^{-1}(\Om)}\le c_1\|u\|_{\wtil W_{AC\cup\s}^1(\Om)},
\q \forall\, u\in C_{0,AC\cup\s}^\infty(\ov\Om),
\\[2mm]
\label{2.5}
&&\hskip -1truecm
\|Tv\|_{\wtil W_{AC\cup\s}^{-1}(\Om)}\le c_2\|v\|_{\wtil W_{BC\cup\s}^1(\Om)},
\q \forall\, v\in C_{0,BC\cup\s}^\infty(\ov\Om).
\eeqn
The continuity estimates (\ref{2.4}) and (\ref{2.5}) give rise to the 
continuous extensions
\beqn\label{2.6}
&&\hskip -1truecm
\wtil T_{AC\cup\s}:\wtil W_{AC\cup\s}^1(\Om)
\to \wtil W_{BC\cup\s}^{-1}(\Om)
\q\ {\rm and}\q\ 
\wtil T_{BC\cup\s}:\wtil W_{BC\cup\s}^1(\Om)
\to\wtil W_{AC\cup\s}^{-1}(\Om)
\eeqn
of the Tricomi operator $T$ defined on the dense subspaces 
$C_{0,AC\cup\s}^\infty(\ov\Om)$ and $C_{0,BC\cup\s}^\infty(\ov\Om)$. 
Notice that, by denoting with $(\wtil T_{AC\cup\s})^*$ and $(\wtil T_{BC\cup\s})^*$ 
the adjoint operators
of $\wtil T_{AC\cup\s}$ and $\wtil T_{BC\cup\s}$, respectively, 
from (\ref{2.6}) we deduce $(\wtil T_{AC\cup\s})^*=\wtil T_{BC\cup\s}$ and 
$(\wtil T_{BC\cup\s})^*=\wtil T_{AC\cup\s}$. This implies that the problems
\beqn\no
&&\hskip -1truecm
({\rm LT}):\
\left\{
\begin{array}{lll}
Tu=h\q{\rm in}\ \Om,
\\[2mm]
u=0\q{\rm on}\ AC\cup\s,
\end{array}
\right.
\q{\rm and}\q\ 
({\rm LT})^*:\
\left\{
\begin{array}{lll}
Tv=h\q{\rm in}\ \Om,
\\[2mm]
v=0\q{\rm on}\ BC\cup\s,
\end{array}
\right.
\eeqn
where $h\in L^2(\Om)$, are adjoint one to each other, but they are {\it not}
self-adjoint. Then, from now on, to simplify notation, 
we shall consider only the problem ({\rm LT}).
In fact, due to the adjoint character of ({\rm LT}) and $({\rm LT})^*$, 
in what follows it will suffice to replace 
the pair $(AC\cup\s,BC\cup\s)$ with $(BC\cup\s,AC\cup\s)$
in all the statements concerning problem $({\rm LT})$ for having analogous 
statements for problem $({\rm LT})^*$. 

\begin{definition}\label{def2.1}
\emph{A function $u\in\wtil W_{AC\cup\s}^1(\Om)$ 
is called a {\it generalized solution} to problem ({\rm LT})
if there exists a sequence $\{u_n\}_{n\in\nsp}\subset C_{0,AC\cup\s}^\infty(\ov\Om)$ 
such that $\|u_n-u\|_{\wtil W_{AC\cup\s}^1(\Om)}\to 0$
and $\|Tu_n-h\|_{\wtil W_{BC\cup\s}^{-1}(\Om)}\to 0$ as $n\to\infty$.}
\end{definition}
As shown first in \cite{Di}, a necessary and sufficient condition
in order to have generalized solvability of $(\rm{LT})$ for every $h\in L^2(\Om)$
is to have the continuity estimates (\ref{2.4}) and (\ref{2.5}) as well as
both the following {\it a priori} estimates, for some positive constants $c_j$, $j=3,4$:
\beqn\label{2.7}
&&\hskip -1truecm
\|u\|_{L^2(\Om)}\le c_3\|Tu\|_{\wtil W_{BC\cup\s}^{-1}(\Om)},
\q \forall\, u\in C_{0,AC\cup\s}^\infty(\ov\Om),
\\[2mm]
\label{2.8}
&&\hskip -1truecm
\|v\|_{L^2(\Om)}\le c_4\|Tv\|_{\wtil W_{AC\cup\s}^{-1}(\Om)},
\q \forall\, v\in C_{0,BC\cup\s}^\infty(\ov\Om).
\eeqn
Precisely, (\ref{2.8}) provides the existence of a generalized 
solution to problem $({\rm LT})$
whereas (\ref{2.7}) guarantees that the solution is unique. 
For this reason, we say that a Tricomi domain $\Om$ is {\it admissible}
if (\ref{2.7}) and (\ref{2.8}) hold. 
Observe also that (\ref{2.7}) and (\ref{2.8}) are in accordance
with the result in \cite{Na} (see also \cite[p. 11]{Di}) 
concerning the validity of a priori estimates for operators of mixed type. 
That is, if an inequality with a step in smoothness of two units such as  
$\|\psi\|_{\wtil W_{AC\cup\s}^1(\Om)}\le c_5\|T\psi\|_{\wtil W_{AC\cup\s}^{-1}(\Om)}$
would hold for every $\psi\in C_{0,AC\cup\s}^\infty(\ov\Om)$,
then $T$ would be elliptic in $\Om$.

The class of admissible domains includes
{\it normal} Tricomi domains whose elliptic boundary arc $\s$ is given as a graph 
$\{(x,y)\in\rsp^2:y=g(x), x\in[2x_0,0]\}$ satisfying the following hypotheses,
where $K_0$ is a positive constant:
\beqn\no
&&\hskip -1truecm
(g1):\q 
g(2x_0)=g(0)=0\q{\rm and}\q g(x)>0,\ \forall\,x\in(2x_0,0),
\\[1mm]
&&\hskip -1truecm
(g2):\q 
g'_{d}(2x_0)=\lim_{t\to 0^+}\frac{g(2x_0+t)}{t}=+\infty
\q{\rm and}\q
g'_{s}(0)=\lim_{t\to 0^-}\frac{g(t)}{t}=-\infty,\no
\\[1mm]
&&\hskip -1truecm
(g3):\q g\in C^2((2x_0,0)),\no
\\[2mm]
&&\hskip -1truecm
(g4):\q g''(x)\le -K_0,\q \forall\,x\in(2x_0,0).\no
\eeqn
We remark that condition $(g2)$ implies that $\s$ is perpendicular 
to the $x$-axis at the boundary points $A$ and $B$. 
That normal Tricomi domains are admissible 
is a consequence of the mentioned necessary condition 
proved in \cite{Di} for the existence of generalized solution 
and of the following result (see \cite[Theorem 2.3]{LP2}).
\begin{theorem}\label{thm2.2}
Let $\Om$ be a normal Tricomi domain. Then, for every $h\in L^2(\Om)$, 
there exists a unique generalized solution $u\in\wtil W_{AC\cup\s}^1(\Om)$   
to problem $({\rm LT})$.
\end{theorem}
The admissibility of normal Tricomi domains allows to enlarge the class of
admissible domains and lead to the following theorem (see \cite[Theorem 2.4]{LP2}).
\begin{theorem}\label{thm2.3}
Let $\Om$ be a Tricomi domain such that:
i) $\Om$ contains a normal subdomain $\Om_0$ having boundary
$\partial\Om_0=AC\cup BC\cup\s_0$; 
ii) there exists an $\e>0$ such that the elliptic boundaries $\s$ and $\s_0$
of $\Om$ and $\Om_0$ coincide in a strip $\{(x,y)\in\rsp^2:0\le y\le \e\}$.
Then $\Om$ is admissible.
\end{theorem}
We stress that (see \cite{Di} and \cite{LP1}), 
for Tricomi domains in which $\s$ forms acute angles
with the parabolic segment $AB$, the previous solvability theory 
can be developed with the pair
$(\wtil W_{AC\cup\s}^1(\Om),\wtil W_{BC\cup\s}^1(\Om))$
being replaced by $(W_{AC\cup\s}^1(\Om),W_{BC\cup\s}^1(\Om))$, 
where $W_\G^1(\Om)$, $\G\in\{AC\cup\s,BC\cup\s\}$,
is defined as the closure with respect 
to the usual $W^{1,2}(\Om)$-norm of the space 
$C^{\infty}_\G(\ov\Om)=\{\psi\in C^\infty(\ov\Om):\psi\equiv 0\ {\rm on}\ \G\}$.
On the contrary (see \cite[p. 445]{LP2}),
when dealing with normal Tricomi domain
the weight $|y|$ in the $\wtil W^{1,2}(\Om)$-norm appears naturally
and describes the possible lack of square integrability of the partial
derivative with respect to $x$ of the solutions in a neighborhood of $A$ and $B$.

Theorem \ref{2.2} implies the existence 
of a continuous right inverse $\wtil T_{AC\cup\s}^{-1}$ 
from all of $L^2(\Om)$ onto a dense proper subspace 
of $\wtil W_{AC\cup\s}^1(\Om)$, and such that 
the generalized solution is exactly $u=\wtil T_{AC\cup\s}^{-1}h$. 
Then, using Rellich's lemma, this continuous
right inverse give rise to an injective, non surjective and compact operator
from $L^2(\Om)$ to $L^2(\Om)$ which we denote again by
$\wtil T_{AC\cup\s}^{-1}$.
It is just such a compactness of the inverse operator that permits the 
possibility of studying the generalized solvability of the spectral problem
 \beqn\no
&&\hskip -2truecm
({\rm LTE}):\
\left\{
\begin{array}{lll}
Tu=\l u\q{\rm in}\ \Om,
\\[2mm]
u=0\q{\rm on}\ AC\cup\s,
\end{array}
\right.
\eeqn
where $\l\in\csp$. Indeed, the compactness of $\wtil T_{AC\cup\s}^{-1}$
combined with a maximum principle for the Tricomi problem 
established in \cite{LP2} exploiting a slight variant of that in \cite{ANP},
yields the following Theorem \ref{thm2.4} which is proved in \cite{LP3}.
We mention that Theorem \ref{thm2.4} was already announced in \cite{Ga}, 
but (see \cite[p. 536]{LP2}) that paper presented two major problems 
to which the proof in \cite{LP3} supplies a remedy. 
\begin{theorem}\label{thm2.4}
Let $\Om$ be normal Tricomi domain. Then there exists an 
eigenvalue-eigenfunction pair $(\l_0,u_0)$ such that $0<\l_0\le|\l|$ 
for every $\l$ in the spectrum $\s(\wtil T_{AC\cup\s})$ of $\wtil T_{AC\cup\s}$
and $u_0\in\wtil W_{AC\cup\s}^1(\Om)$ satisfies $u_0\ge 0$ 
almost everywhere in $\Om$.
\end{theorem}
Note that, since the eigenvalues of $\wtil T_{AC\cup\s}$ are the inverse
of those of $\wtil T_{AC\cup\s}^{-1}$, 
the compactness of $\wtil T_{AC\cup\s}^{-1}$ implies that $\wtil T_{AC\cup\s}$ 
has a discrete spectrum composed entirely of eigenvalues of finite multiplicity 
with a unique accumulation point at infinity. 
The eigenvalue $\l_0$ of Theorem \ref{thm2.4} is called
a {\it principal eigenvalue} due to the positivity of 
the associated eigenfunction $u_0$ and its being of minimum modulus. 
However, at present, it is not known neither whether if
the associated eigenspace is simple, nor whether if other eigenspaces do not contain
eigenfunctions that are nonnegative almost everywhere, as it happens 
in the purely elliptic case. Nevertheless, what is known is that all real eigenvalues 
of $\wtil T_{AC\cup\s}$ must be positive. This spectral information is the
content of \cite[Theorem 2.5(a)]{LP4}, and, according to Theorem \ref{thm2.4},
may be summarized as
\beqn\label{2.9}
&&\hskip -1truecm
\s(\wtil T_{AC\cup\s})\cap(-\infty,\l_0)=\emptyset.
\eeqn
To the author's knowledge, (\ref{2.9}) is the best information on the spectrum
of $\wtil T_{AC\cup\s}$ compatible with the solvability theory in the
space $\wtil W_{AC\cup\s}^1(\Om)$. 
Indeed, the results in \cite{Mo1} and \cite{Mo2}, which establish that 
$\s(\wtil T_{AC\cup\s})\cap\{\l\in\csp: 2\pi/3\le|\arg\l|\le4\pi/3\}=\emptyset$,
require that the eigenfunctions should be at least of class 
$C(\ov\Om)\cap C^1(\Om)\cap C^2(\Om_+)\cap C^2(\Om_-)$, 
where $\Om_\pm=\{(x,y)\in\Om:\pm y>0\}$.
Unfortunately, the question of regularity of the eigenfunctions is still an open question, but, anyhow, one can show the existence of a continuous eigenfunction.
More precisely, using the solvability result in \cite[p. 64]{Ag} for normal domains,
in \cite{LP4} it is shown the following theorem.
\begin{theorem}\label{thm2.5}
Let $\Om$ be a normal Tricomi domain and let $\l_0$ be the positive
eigenvalue of {\rm Theorem \ref{thm2.4}}. 
Then there exists an eigenvalue-eigenfunction pair $(\wtil\l_0,\wtil u_0)$ 
such that $\wtil\l_0\ge\l_0$ and 
$\wtil u_0\in\wtil W_{AC\cup\s}^1(\Om)\cap C(\ov\Om)$
satisfies $\wtil u_0\ge 0$ in $\ov\Om$.
\end{theorem}
\section{$D$-star-shaped domains and Poho\v zaev identity}\label{Sec3}
\setcounter{equation}{0}
In Section \ref{Sec2} we have defined the Tricomi domains so that the boundary 
points $A$ and $B$ coincide, respectively, with $(2x_0,0)$ and $(0,0)$, 
where $x_0<0$. Such a choice is made only in order to uniform our notation 
with that of \cite{LP5}, whose results we shall need later. 
Indeed, due to the invariance  of the Tricomi operator (\ref{2.1}) with respect to translations along the $x$ axis, any other choices for $A$ and $B$ 
could be possible. To this purpose, it suffices to observe that 
if $u\in C^2(\Om)$ solves one of the problems ({\rm LT}) and ({\rm LTE}) 
in $\Om$,  then, by setting $x^*=x-l$, $y^*=y$, $l\in\rsp$, 
the function $\wtil u(x^*,y^*)=u(x^*+l,y^*)$ solves the corresponding problem
in the relevant translate $\wtil\Om$ of $\Om$. 

As noticed in \cite{LP6} (take there $m=N=1$ in the equation
$y|y|^{m-1}\sum_{i=1}^Nu_{x_ix_i}+u_{yy}+f(u)=0$), 
translations in the $x$ variables are the easiest
symmetries that generate conservation laws associated to the semilinear problem
\beqn\label{3.1}
&&\hskip -2truecm
\left\{
\begin{array}{lll}
Tu=f(u)\q {\rm in}\ \Om,
\\[2mm]
u=0\q{\rm on}\ AC\cup\s,
\end{array}
\right.
\eeqn
where $f\in C(\rsp)$. Recall that a conservation law associated to (\ref{3.1})
is a first-order equation in divergence form $\div(U)=0$ which must be
satisfied by every sufficiently regular solution of the given problem, where
$U=U(x,y,u,\nabla u,f)$ is some vector field whose dependence on $u$ is, 
in general, highly nonlinear. 
Apart from translations, other two symmetry groups 
that generate conservation laws for problem (\ref{3.1}) are exhibited in \cite{LP6}, 
i. e. those coming from certain anisotropic dilations and 
from inversion with respect to the curve 
\beqn\label{3.2}
&&\hskip -1truecm
9(x-x_0)^2+4y^3=9x_0^2,\q y\ge 0.
\eeqn
According to \cite[Chapter IV]{Tr},  
the curve in (\ref{3.2}) which joins the boundary points 
$A$ and $B$ in the elliptic region,
is called the {\it normal curve} for the Tricomi operator.
In particular, from (\ref{3.2}) we get 
\beqn\label{3.3}
&&\hskip 0,5truecm
y=\big[9x(2x_0-x)/4\big]^{1/3}=:g_{T}(x),\q x\in[2x_0,0].
\eeqn
Hence, a standard exercise of calculus shows that 
the function $g_T$ in (\ref{3.3}) satisfies all the conditions 
$(g1)$--$(g4)$ of Section \ref{Sec2} with 
\beqn\label{3.4}
&&\hskip -1truecm
K_0=-g_T''(x_0)=\big(3x_0^4/2\big)^{-1/3}.
\eeqn
Since we do not need inversions in this paper, we only refer to
\cite{GB} for their construction and their application to (\ref{3.1})
with $f=0$, and to \cite{LP6} for how to use inversions to derive
conservation laws for (\ref{3.1}) in both the cases $f=0$ 
and $f(u)=u^9$, the exponent $\ov\a=9$ 
corresponding to the critical exponent obtained in \cite{LP5}.
Here, instead, we focus our attention to the second group of symmetries,
which leads to the concept of $D$-{\it star-shaped domain}
and is strongly related to the Poho\v zaev identity that we shall recall later.

Let $\g>0$ and consider the change of variable
\beqn\no
&&\hskip -0truecm
(x,y)\in\Om\to(\g^3x,\g^2y)=:(x^*,y^*)\in\Om^*.
\eeqn
It is easy to verify that if $u\in C^2(\Om)$ is a solution 
of problem (\ref{3.1}) with $f=0$,  then, for every
fixed $\dl\ge 0$, the scaled function (see \cite[p. 256]{LP6})
\beqn\label{3.5}
&&\hskip -1truecm
u_{\g}(x^*,y^*)=\g^{-\dl}u(\g^{-3}x^*,\g^{-2}y^*),
\eeqn
solves the same problem in the scaled domain $\Om^*$ of $\Om$. 
Thus, we have a multiplicative group $\rsp_+$ of anisotropic dilations 
as a symmetry group for the linear homogeneous problem (\ref{3.1}). 
For instance, such a dilation invariance has been applied in 
\cite{BN1}--\cite{BN3} to the search of fundamental solutions 
for the Tricomi operator.
In the general case, the semilinear problem (\ref{3.1}) does not have 
this symmetry group of dilations, but a straightforward computation 
shows that this is true for power type nonlinearities, 
provided $\dl$ is suitably chosen in (\ref{3.3}).
That is, if $f(u)=Cu^\a$ with $C\in\rsp$ and $\a>1$,
then problem (\ref{3.1}) has the property of dilations invariance
for $\dl=4(\a-1)^{-1}$.
However, it is worth to remark that in the case $f(u)=\l u$, 
corresponding to problem ({\rm LTE}), there is no way to choose $\dl\ge 0$ in (\ref{3.3}) such that the dilation invariance is satisfied.

The first variation of the one-parameter family of scaled functions (\ref{3.5})
under the action of the one-parameter group of dilation is
\beqn\no
&&\hskip -1truecm
\Big[\frac{\d}{\d\g}u_\g(x^*,y^*)\Big]\Big|_{\g=1}=Du(x,y)-\dl u(x,y),
\eeqn
where $D$ is the vector field
\beqn\label{3.6}
&&\hskip -1truecm
D=-3x\partial_x-2y\partial_y.
\eeqn
This vector field determines a flow ${\cal F}_t:\rsp^2\to\rsp^2$, $t\in\rsp$,
such that ${\cal F}_t(\ov x,\ov y)=\phi_{(\ov x,\ov y)}(t)$, where,
denoting by $B^T$ the transpose of a $p\times q$ matrix $B$, 
$\phi_{(\ov x,\ov y)}^T(t)$ is the unique integral curve of the linear system
\beqn\label{3.7}
&&\hskip 0truecm
\left\{\!
\begin{array}{lll}
V'(t)= AV(t),
\\[2mm]
V(0)=(\ov x,\ov y)^T,
\end{array}
\right.
\q
V(t)=\bigg(
\begin{matrix}
      x(t)\\
      y(t)\\
\end{matrix}
\bigg),
\q
A=\bigg(
\begin{matrix}
      -3 & 0 \\
       0 & -2 \\
\end{matrix}
\bigg).
\eeqn
Therefore, for every $(t,\ov x,\ov y)\in\rsp^3$,
we have ${\cal F}_t(\ov x,\ov y)=(\ov x\ee^{-3t},\ov y\ee^{-2t})$.
\begin{definition}\label{def3.1}
\emph{Let $D$ be defined by (\ref{3.6}). 
An open set $G\subset\rsp^2$ is said to be 
{\it $D$--star-shaped} if for each $(\ov x,\ov y)\in\ov G$
one has ${\cal F}_t(\ov x,\ov y)\subset\ov G$ for every $t\in[0,+\infty]$,
where ${\cal F}_{+\infty}(\ov x,\ov y)=\lim_{t\to+\infty}{\cal F}_t(\ov x,\ov y)=(0,0)$.}
\end{definition}
To make clear the importance of this definition, we recall that if 
$\Om$ is a normal Tricomi domain which is also $D$-star-shaped  
then the continuous and compact embedding 
$\wtil W_{AC\cup\s}^1(\Om)\hookto L^p(\Om)$ holds for every $p\in[1,p^*)$,
where $p^*=2N(N-2)^{-1}=10$, $N=5/2$. Here, $N=5/2$ is the so-called 
homogeneous dimension of $\rsp^2$ when equipped with a 
non-Euclidian metric $d$ which is natural for the Tricomi operator 
as the Euclidian metric is natural for the Laplace operator
(see \cite{FL1}, \cite{FL2}, \cite{LP5} and \cite{LP6}).

Bounded $D$-star-shaped domains have $D$-starlike boundaries 
as established by the following lemma (see \cite[Lemma 2.2]{LP5}). 
From now on, $\langle\cdot,\cdot\rangle$ will always denote
the canonical inner product of $\rsp^2$.
\begin{lemma}\label{lem3.2}
Let $G\subset\rsp^2$ be an open set with piecewise $C^1$ boundary 
$\partial G$. If $G$ is $D$-star-shaped, then $\partial G$ is $D$-starlike 
in the sense that $\langle(-3x,-2y),\vec{\bf n}(x,y)\rangle\le 0$
at each regular point $(x,y)\in\partial G$ where 
$\vec{\bf n}(x,y)$ is the unit outer normal to $\partial G$ at the point $(x,y)$.
\end{lemma} 
The notion of $D$-star-shaped domains has been used in \cite{LP5} to 
prove the nonexistence of nontrivial regular solutions 
to problem (\ref{3.1}) in the case $f(u)=|u|^\a$ with $\a>p^*-1$,
thus showing that the homogeneous dimension of $\rsp^2$ is
responsible for a critical-exponent phenomenon in the nonlinearity.
In the quoted paper, the key tool is to combine the $D$-star-shapedness of $\Om$ 
with the following Poho\v zaev-type identity that we recall 
for the reader's convenience, by referring to \cite{LP5} for its proof.
\begin{theorem}\label{thm3.3}
Let $\Om$ be a Tricomi domain for $T$
and let $D$ be the vector field defined by $(\ref{3.6})$.
Let $u$ be a solution of problem $(\ref{3.1})$
such that $u_y$, $xu_x$, $yu_x\in C^1(\ov\Om)$
and $xu\in C^2(\Om)$. Then the following identity holds true
\beqn\label{3.8}
&&
\hskip -1,5truecm
\int_{\Om}\big[10F(u)-uf(u)\big]\d x\d y
=\int_{BC}(\om_1+\om_2)\d s+\int_{\s}\om_1\d s.
\eeqn
Here $F$ is a primitive of $f\in C^0(\rsp)$ such that $F(0)=0$, 
whereas $\om_1$ and $\om_2$ are defined by
\beqn\label{3.9}
&&\hskip -1,5truecm
\om_1=\langle2Du(-yu_x,-u_y)+(yu_x^2+u_y^2)(-3x,-2y),\vec{\bf n}\rangle,
\\[2mm]
\label{3.10}
&&\hskip -1,5truecm
\om_2=\langle-2F(u)(-3x,-2y)-u(-yu_x,-u_y), \vec{\bf n}\rangle,
\eeqn
$\vec {\bf n}$ being the unit outer normal field to $\Om$.
\end{theorem}
\begin{remark}\label{rem3.4}
\emph{Observe that, according to \cite[p. 420]{LP5},
we have formulated Theorem \ref{thm3.3} under weaker assumptions for $u$.
In fact, the requirements $u_y$, $xu_x$, $yu_x\in C^1(\ov\Om)$  and
$xu\in C^2(\Om)$ suffice for applying the classical divergence theorem
for $C^1(\ov\Om)$ vector fields and for exchanging the order of certain partial
derivatives in the proof of (\ref{3.8}), and allow to weakening the
original stronger condition $u\in C^2(\ov\Om)$ (see \cite[Theorem 3.1]{LP5}).}
\end{remark}
Since the starting point for obtaining our estimates on the eigenfunctions of the
Tricomi operator is the identity (\ref{3.8}), we conclude the section
spending some words on it. 
In the theory of semilinear elliptic equations the first appearance of 
an identity between volume and surface integrals of kind (\ref{3.8}) goes back
to \cite{Po}. There, such an identity resulted from an energy 
integral method consisting in multiplying the differential equation 
by a suitable vector field and then applying the divergence theorem.
Since \cite{Po}, this method for obtaining identities of type (\ref{3.8})
has become a standard tool in the theory of semilinear elliptic equations.
On the contrary, the situation is quite different
for semilinear equations of mixed elliptic-hyperbolic and degenerate types
where, to our knowledge, the only remarkable results in the derivation
of such identities are those in \cite{LP5}.
Indeed, using an argument that reproduces the original idea of \cite{Po},
in \cite{LP5} identities of type (\ref{3.8}) are derived for 
the semilinear problem (\ref{3.1}), with the Tricomi operator $T$ 
being replaced by the more general Gellerstedt operator 
$L=-y^{2k+1}\partial_x^2-\partial_y^2$, $k\in\nsp\cup\{0\}$.
In particular, the above Theorem \ref{thm3.3} is obtained
by taking $k=0$ in \cite[Theorem 3.1]{LP5}.

Usually, Poho\v zaev identities are applied for the proof of nonexistence results.
In doing so, one has only to show that the signs of the volume 
and surface integrals are incompatible with the existence of nontrivial solutions.
This is, for instance, the scheme followed in the quoted papers 
\cite{Po} and \cite{LP5}.
Our approach will be different. For the problem ({\rm LTE}) 
(corresponding to $f(u)=\l u$ in (\ref{3.1})) $F(u)$ turns out to be $\l u^2/2$,
so that the left-hand side of (\ref{3.8}) reduces to $4\l\|u\|_{L^2(\Om)}^2$.
Then, we shall get our estimates on the eigenfunctions of the Tricomi problem
simply by showing that the right-hand side of (\ref{3.8}) is nonnegative and
upper bounded by a suitable quantity.
\begin{remark}\label{rem3.5}
\emph{Of course, a remark is in order about the approach 
summarized in the last paragraph. 
Indeed, the eigenfunctions of the Tricomi problem are, in general, complex valued, 
and we are not in position to apply Theorem \ref{thm3.3}, 
which, due to the assumption $f\in C^0(R)$ and the presence
of the canonical inner product of $\rsp^2$, 
requires a real context for its application.
However, if we restrict our interest to the eigenfunctions $u$ corresponding to 
real positive eigenvalues $\l\in[\l_0,+\infty)$ (see Theorem \ref{thm2.4}),
then we can apply separately our approach to their real and imaginary parts,
$\Re u$ and $\Im u$. 
For, $T$ being a linear operator, we have} 
\beqn\label{3.11}
&&\hskip -0,5truecm
Tu=\l u,\q \l\in\rsp,
\q\iff\q
\left\{\!\!
\begin{array}{lll}
T(\Re u)=\l\Re u,
\\[2mm]
T(\Im u)=\l\Im u,
\end{array}
\right.
\q\l\in\rsp.
\eeqn
\emph{That is, $u$ is an eigenfunction corresponding to a real eigenvalue $\l$
if and only if its real and imaginary parts $\Re u$ and $\Im u$ 
are {\it real valued} eigenfunctions corresponding to $\l$.
Thus, once we have estimated $\|\Re u\|_{L^2(\Om)}^2$
and $\|\Im u\|_{L^2(\Om)}^2$, our estimate on the $L^2$-norm 
of the (possibly complex valued) eigenfunctions $u$ corresponding to positive 
eigenvalues will follow from $\|u\|_{L^2(\Om)}^2=(u,u)_{2,\sim}=
\|\Re u\|_{L^2(\Om)}^2+\|\Im u\|_{L^2(\Om)}^2$, 
where $(v_1,v_2)_{2,\sim}=\int_{\Om}v_1\ov{v_2}\d x\d y$.}
\end{remark}
\section{Main results}\label{Sec4}
\setcounter{equation}{0}
From the definition of normal Tricomi domains given in Section \ref{Sec1}, 
and in particular from conditions $(g1)$--$(g3)$, it follows that
the product $g^{\a}g'$, $\a>0$, is a continuous function in the
interval $(2x_0,0)$, but the limits $\lim_{x\to 0^-}[g(x)]^{\a}g'(x)$
and  $\lim_{x\to (2x_0)^+}[g(x)]^{\a}g'(x)$ lead, in general, to the
indeterminate form  $0\cdot\infty$. 
On the other side, if we look at the normal curve (\ref{3.2}), we see that
derivative $g_T'$ of the function $g_T$ defined by (\ref{3.3}) satisfies the
identity $g_T'(x)=-(3/2)(x-x_0)[g_T(x)]^{-2}$ for every $x\in(2x_0,0)$.
Hence, in this case, the product $g_T^\a g_T'$, $\a=2$, is the
linear function $-(3/2)(x-x_0)$ and both
the mentioned limits exist and are equal to $(3/2)x_0$ and $-(3/2)x_0$, respectively. 
This observation leads us to consider those normal Tricomi domains
having boundary $\partial\Om=AC\cup BC\cup\s$, where the characteristics
$AC$ and $BC$ are as in (\ref{2.2}) and (\ref{2.3}), 
and the elliptic boundary arc $\s$ is the graph of a function 
$g=g(x)$, $x\in[2x_0,0]$, which satisfies conditions $(g1)$--$(g4)$ 
and the additional requirement that $g^\a g'$, $\a>0$, 
is a linear function with negative slope, i.e.
\beqn\label{4.1}
&&\hskip -1truecm
[g(x)]^\a g'(x)=\b x+\g,\q x\in(2x_0,0),\q \a>0>\b,\q \g\in\rsp.
\eeqn
Integrating (\ref{4.1}) from $2x_0$ to $x$, $x\in[2x_0,0]$, 
and using $g(2x_0)=0$, we thus find
\beqn\label{4.2}
&&\hskip -1truecm
[g(x)]^{\a+1}=(\a+1)[(\b/2)x^2+\g x-2x_0(\b x_0+\g)],\q x\in[2x_0,0].
\eeqn
The additional requirement $g(0)=0$ yields $\g=-\b x_0$, 
so from (\ref{4.2}) we finally obtain the family of functions
\beqn\label{4.3}
&&\hskip -0,5truecm
g_{\a,\b}(x)=\Big[\frac{(\a+1)|\b| x(2x_0-x)}{2}\Big]^{1/(\a+1)},\q x\in[2x_0,0],\q\a>0>\b.
\eeqn
For construction, the functions $g_{\a,\b}$ defined by (\ref{4.3}) 
satisfy conditions $(g1)$-$(g4)$. Of course, $(gj)$, $j=1,2$, are obvious,
since (see (\ref{4.1}) with $\g=-\b x_0$)  
$g_{\a,\b}'(x)=\b(x-x_0)[g_{\a,\b}(x)]^{-\a}$, $x\in(2x_0,0)$. 
Observe also that the functions $g_{\a,\b}$
are even with respect to the line $\{x=x_0\}$ and increasing
in the interval $[2x_0,x_0]$. Hence they attain their maximum
value $g_{\a,\b}(x_0)=[(\a+1)|\b| x_0^2/2]^{1/(\a+1)}$
at the unique critical point $x=x_0$. 
Moreover, differentiating $(\ref{4.1})$ with respect to $x$ we get
\beqn\no
&&\hskip-0,5truecm
g_{\a,\b}''(x)=\{\b-\a[g_{\a,\b}(x)]^{\a-1}[g_{\a,\b}'(x)]^2\}[g_{\a,\b}(x)]^{-\a},\q
x\in(2x_0,0),
\eeqn
so $g_{\a,\b}''\in C((2x_0,0))$ and
$g_{\a,\b}''(x)\le -K_0(\a,\b)$ for every $x\in(2x_0,0)$, where 
\beqn\no
&&\hskip-1truecm
K_0(\a,\b)=-g_{\a,\b}''(x_0)=-\b[g_{\a,\b}(x_0)]^{-\a}
=|\b|[(\a+1)|\b|x_0^2/2]^{-\a/(\a+1)}>0.
\eeqn 
Therefore, conditions $(gj)$, $j=3,4$, are satisfied, too.
\begin{remark}\label{rem4.1}
\emph{We stress that the function $g_T$, defined by (\ref{3.3})
and corresponding to the choice of the normal curve for the elliptic boundary
arc $\s$, is obtained for the values $\a=2$ and $\b=-3/2$ in (\ref{4.3}), 
that is $g_T=g_{2,-3/2}$. It is a simple computation to verify
that in this case the value $K_0(2,-3/2)$ coincides with the value 
$K_0$ in (\ref{3.4})}
\end{remark}
From now on $\Om_{\a,\b}$, $\a>0>\b$, will denote a normal Tricomi domain
whose boundary $\partial\Om_{\a,\b}$ consists of the characteristics $AC$ and 
$BC$ given in (\ref{2.2}) and (\ref{2.3}) and of the elliptic arc 
$\s_{\a,\b}=\{(x,y)\in\rsp^2:y=g_{\a,\b}(x),\ x\in[2x_0,0]\}$,
where $g_{\a,\b}$ is defined by (\ref{4.3}).
Let us parametrize the curves $AC$, $BC$ and $\s_{\a,\b}$ in order to give to
$\partial\Om_{\a,\b}$ the positive orientation of leaving the interior of 
$\Om_{\a,\b}$ on the left, i.e. the counterclockwise orientation.
To this purpose, denoting by $r_\G:I\subset\rsp\to\rsp^2$, $I$ an interval, 
the parametric curve representing a subset $\G$ of $\partial\Om_{\a,\b}$, 
we have:
\beqn\label{4.4}
&&\hskip -2truecm
r_{AC}(-y)=(2x_0+(2/3)(-y)^{3/2},y),\q y\in[y_C,0],
\\[2mm]
\label{4.5}
&&\hskip -2truecm
r_{BC}(y)=(-(2/3)(-y)^{3/2},y),\q y\in[y_C,0],
\\[2mm]
\label{4.6}
&&\hskip -2truecm
r_{\s_{\a,\b}}(-x)=(x,g_{\a,\b}(x)),\q x\in[2x_0,0],
\eeqn
where $y_C=-(3|x_0|/2)^{2/3}$.
Consequently, the unit outer normals on the characteristics and on $\s_{\a,\b}$ 
are given by
\beqn
\label{4.7}
&&\hskip -1truecm
\vec{\bf n}_{AC}=(1-y)^{-1/2}(-1,-(-y)^{1/2}),\q y\in[y_C,0],
\\[2mm]
\label{4.8}
&&\hskip -1truecm
\vec{\bf n}_{BC}=(1-y)^{-1/2}(1,-(-y)^{1/2}),\q y\in[y_C,0],
\\[2mm]
\label{4.9}
&&\hskip -1truecm
\vec{\bf n}_{\s_{\a,\b}}=
\left\{
\begin{array}{lll}
(-1,0)\q x=2x_0,
\\[2mm]
\{[g_{\a,\b}'(x)]^2+1\}^{-1/2}(-g_{\a,\b}'(x),1),\q x\in(2x_0,0),
\\[2mm]
(1,0),\q x=0.
\end{array}
\right.
\eeqn
Observe that, using $g_{\a,\b}'(x)=\b(x-x_0)[g_{\a,\b}(x)]^{-\a}$, 
easy computations yield to:
\beqn\label{4.10}
&&\hskip -1truecm
\frac{(-g_{\a,\b}'(x),1)}{\{[g_{\a,\b}'(x)]^2+1\}^{1/2}}=
\frac{(x-x_0,|\b|^{-1}[g_{\a,\b}(x)]^\a)}
{\{(x-x_0)^2+\b^{-2}[g_{\a,\b}(x)]^{2\a}\}^{1/2}},\q x\in(2x_0,0).
\eeqn
Then, the vector on the right-hand side of (\ref{4.10}) being defined also 
for $x=2x_0$ and $x=0$ where it is equal to $(-1,0)$ and $(1,0)$, respectively,
we can replace (\ref{4.9}) with the more compact formula:
\beqn\label{4.11}
&&\hskip -1truecm
\vec{\bf n}_{\s_{\a,\b}}=[h_{\a,\b}(x)]^{-1}(x-x_0,|\b|^{-1}[g_{\a,\b}(x)]^\a),
\q x\in[2x_0,0],
\eeqn
where $h_{\a,\b}$ is the positive continuous function, 
even with respect to the line $\{x=x_0\}$,
\beqn\label{4.12}
&&\hskip -0truecm
h_{\a,\b}(x)=\{(x-x_0)^2+\b^{-2}[g_{\a,\b}(x)]^{2\a}\}^{1/2},\q x\in[2x_0,0].
\eeqn
Notice that, if $\a=1$, then (\ref{4.3}) and (\ref{4.12}) yield
$h_{1,\b}(x)=\{(|\b|^{-1}-1)x(2x_0-x)+x_0^2\}^{1/2}$ and hence, if $\b=-1$,
$h_{1,\b}$ reduces to the constant function $h_{1,-1}(x)=|x_0|$, $x\in[2x_0,0]$.
\begin{lemma}\label{lem4.2}
The boundary $\partial\Om_{\a,\b}$ of $\Om_{\a,\b}$, $\a>0>\b$, 
is $D$-starlike with respect to $D=-3x\partial_x-2y\partial_y$ 
if and only if $\a\ge1/2$.
\end{lemma}
\begin{proof}
We have to prove that $\langle(-3x,-2y),\vec{\bf n}(x,y)\rangle\le 0$ at each point 
$(x,y)\in\partial\Om_{\a,\b}$ if and only if $\a\ge 1/2$, where $\vec{\bf n}(x,y)$ 
is the unit outer normal to $\partial\Om_{\a,\b}$ at the point $(x,y)$. 
First, from (\ref{4.4}), (\ref{4.5}), (\ref{4.7}) and (\ref{4.8}) we get:
\beqn\label{4.13}
&&\hskip -1truecm
\left\{\!\!
\begin{array}{lll}
\langle(-3x,-2y),\vec{\bf n}_{AC}\rangle=6x_0(1-y)^{-1/2}<0,\q\forall\,(x,y)\in AC,
\\[2mm]
\langle(-3x,-2y),\vec{\bf n}_{BC}\rangle=0,\q\forall\,(x,y)\in BC.
\end{array}
\right.
\eeqn
On the contrary, from (\ref{4.6}) and (\ref{4.11}) it follows
\beqn\label{4.14}
\hskip -0,5truecm
\langle(-3x,-2y),\vec{\bf n}_{\s_{\a,\b}}\rangle
&\!\!\!=\!\!\!&
[h_{\a,\b}(x)]^{-1}\{-3x(x-x_0)-2|\b|^{-1}[g_{\a,\b}(x)]^{\a+1}\}\no
\\[2mm]
\hskip -0,5truecm
&\!\!\!=\!\!\!&
[h_{\a,\b}(x)]^{-1}[-3x(x-x_0)-(\a+1)x(2x_0-x)]\no
\\[2mm]
\hskip -0,5truecm
&\!\!\!=\!\!\!&
[h_{\a,\b}(x)]^{-1}x[(\a-2)x+(1-2\a)x_0],
\q\forall\,(x,y)\in\s_{\a,\b}.
\eeqn
Therefore, since $x\in[2x_0,0]$ in $\s_{\a,\b}$, the inequality 
$\langle(-3x,-2y),\vec{\bf n}_{\s_{\a,\b}}\rangle\le 0$ will be satisfied
for every $(x,y)\in\s_{\a,\b}$ if and only if $(\a-2)x+(1-2\a)x_0\ge 0$
for every $x\in[2x_0,0]$. 
It thus suffices to analyze the behavior of
the straight line $l_\a(x)=(x,(\a-2)x+(1-2\a)x_0)$, $x\in\rsp$, 
in dependence of the parameter $\a>0$. To this purpose, 
we denote with $m_\a$ and $y_\a$ the real numbers $\a-2$ and
$(1-2\a)x_0$, respectively.
Of course, if $\a<1/2$, then $y_\a<0$ and 
$l_\a$ is a straight line passing for the point $(0,y_\a)$ with 
negative slope $m_\a<-3/2$. 
So, in this case, the inequality $m_\a x+y_\a\ge 0$ is satisfied 
only for the values of $x$ less or equal than the negative number 
$-m_\a^{-1}y_\a$ and the right-hand side of (\ref{4.14}) is positive 
for $x\in (-m_\a^{-1}y_\a,0]$, contradicting the
$D$-starlikeness of $\partial\Om_{\a,\b}$.
On the contrary, if $\a\ge1/2$, then $y_\a\ge 0$ 
and $l_\a$ is a straight line passing through $(0,y_\a)$ with slope $m_\a\ge-3/2$.
In particular, if $\a\in[1/2,2)$, then $m_\a\in[-3/2,0)$ and the
inequality $m_\a x+y_\a\ge 0$ is satisfied for every $x$ less or equal than
the nonnegative number $-m_\a^{-1}y_\a$ and a fortiori for every $x\in[2x_0,0]$.
If $\a=2$, then $(m_\a,y_\a)=(0,-3x_0)$ and $m_\a x+y_\a=-3x_0>0$
for every $x\in\rsp$. Finally, if $\a> 2$, then
$l_\a$ passes through the point $(0,y_\a)$ with positive slope $m_\a$,
so that $m_\a x+y_\a\ge 0$ for every $x$ greater or equal than the 
negative number $-m_\a^{-1}y_\a$. 
However, when $\a>2$, we have $2x_0\ge -m_\a^{-1}y_\a=-(\a-2)^{-1}(1-2\a)x_0$ 
and hence the inequality $m_\a x+y_\a\ge 0$ holds a fortiori for every 
$x\in[2x_0,0]$. This completes the proof.
\end{proof}
There is more. That is, $\Om_{\a,\b}$ is just $D$-star-shaped 
in the sense of Definition \ref{def3.1} if and only if $\a\ge 1/2$. 
We shall not need this fact later 
(all that we shall need is the $D$-starlikeness of $\Om_{\a,\b}$, already proved), 
but we prove it for completeness since the proof is very easy and since 
it gives a concrete character to the abstract notion of $D$-star-shaped domain.
\begin{lemma}\label{lem4.3}
$\Om_{\a,\b}$, $\a>0>\b$, is $D$-star-shaped with respect to 
$D=-3x\partial_x-2y\partial_y$ if and only if $\a\ge 1/2$.
\end{lemma}
\begin{proof}
Clearly, if $\Om_{\a,\b}$ is $D$-star shaped, then (see Lemma \ref{lem3.2}) 
its boundary $\partial\Om_{\a,\b}$ is $D$-starlike and hence $\a\ge1/2$
by virtue of the previous Lemma \ref{lem4.2}.
Let us assume now $\a\ge 1/2$ and prove that $\Om_{\a,\b}$ is $D$-star-shaped.
As it is well-known (see \cite[Chapter 15]{CL}), for the linear system (\ref{3.7})
the origin is an improper node asymptotically stable and
every orbit, except the two corresponding to the positive and negative $x$-axis,
tends to the origin tangentially to the $y$-axis. 
It thus suffices to show 
${\cal F}_t(\ov x,\ov y)=(\ov x\ee^{-3t},\ov y\ee^{-2t})\subset\ov{\Om_{\a,\b}}$
for every $t\in[0,+\infty]$ only for the points $(\ov x,\ov y)\in\partial\Om_{\a,\b}$.
Indeed, for every $(\ov x,\ov y)\in\Om_{\a,\b}$ there corresponds 
a unique $(\wtil x,\wtil y)\in\partial\Om_{\a,\b}$ such that 
$(\wtil x,\wtil y)={\cal F}_{t_0}(\ov x,\ov y)$ for 
$t_0=3^{-1}\ln(\ov x/\wtil x)=2^{-1}\ln(\ov y/\wtil y)<0$,
and, cosequently, ${\cal F}_t(\ov x,\ov y)={\cal F}_{t-t_0}(\wtil x,\wtil y)$, $t\in\rsp$.
Let first $(\ov x,\ov y)\in\s_{\a,\b}$. 
To prove ${\cal F}_t(\ov x,\ov y)\in\ov{\Om_{\a,\b}}$, $t\in[0,+\infty]$, 
where ${\cal F}_{+\infty}(\ov x,\ov y)=(0,0)=B$, we have to show that
$\ov y\ee^{-2t}\le g_{\a,\b}(\ov x\ee^{-3t})$ for every $t\ge 0$, 
or, equivalently,
\beqn\label{4.15}
&&\hskip -1truecm
(\a+1)|\b|{\ov x}^2\ee^{-6t}-2(\a+1)|\b|x_0\ov x\ee^{-3t}
+2{\ov y}^{(\a+1)}\ee^{-2(\a+1)t}\le 0,\q t\ge 0.
\eeqn
But, if $(\ov x,\ov y)\in\s_{\a,\b}$, then $\ov y=g_{\a,\b}(\ov x)$ and
$2{\ov y}^{(\a+1)}=(\a+1)|\b|\ov x(2x_0-\ov x)$.
Replacing this latter identity in (\ref{4.15}) we find that we have to prove 
\beqn\label{4.16}
&&\hskip -1truecm
(\a+1)|\b|\ov x\big\{\ov x[\ee^{-6t}-\ee^{-2(\a+1)t}]
-2x_0[\ee^{-3t}-\ee^{-2(\a+1)t}]\big\}\le 0,\q t\ge 0.
\eeqn
Since $\a\ge 1/2$ and $\ov x\in[2x_0,0]$, (\ref{4.16}) is satisfied if and only if 
\beqn\label{4.17}
&&\hskip -1truecm
\ov x c_{\a,t}\ge
2x_0 d_{\a,t},\q t\ge 0,
\eeqn
where for $\a\ge 1/2$ and $t\ge 0$ we have set 
$c_{\a,t}=[\ee^{-6t}-\ee^{-2(\a+1)t}]$ and $d_{\a,t}=[\ee^{-3t}-\ee^{-2(\a+1)t}]$. 
Of course, $c_{\a,t}<d_{\a,t}$ and $d_{\a,t}\ge 0$. 
Moreover, $c_{\a,t}>0$, $c_{\a,t}=0$ or $c_{\a,t}<0$ according that 
$\a>2$, $\a=2$ or $\a\in[1/2,2)$. So, if $\a>2$, then (\ref{4.17}) holds for
$\ov x\ge 2x_0d_{\a,t}(c_{\a,t})^{-1}$. But $d_{\a,t}(c_{\a,t})^{-1}>1$, 
and (\ref{4.17}) holds a fortiori for every $\ov x\in[2x_0,0]$. 
If $\a=2$, then $0=c_{\a,t}<d_{\a,t}=[\ee^{-3t}-\ee^{-6t}]$ and $(\ref{4.17})$ 
reduces to $0\ge 2x_0d_{\a,t}$, which is true for every $t\ge 0$.
Finally, if $\a\in[1/2,2)$, then $c_{\a,t}<0\le d_{\a,t}$ and (\ref{4.17}) is satisfied
for every $\ov x$ less or equal than the nonnegative real number
$2x_0d_{\a,t}(c_{\a,t})^{-1}$. Hence, it is a fortiori satisfied for every $x\in[2x_0,0]$.
It remains to analyze the orbits ${\cal F}_{t}(\ov x,\ov y)$ of the
points $(\ov x,\ov y)\in AC\cup BC$.
If $(\ov x,\ov y)\in BC$, then $3\ov x\ee^{-3t}+2(-\ov y\ee^{-2t})^{3/2}
=[3\ov x+2(-\ov y)^{3/2}]\ee^{-3t}=0$, meaning that 
${\cal F}_{t}(\ov x,\ov y)\in BC$ for every $t\in[0,+\infty]$.
Let $(\ov x,\ov y)\in AC$. Due to what already proved
and since orbits do not intersect each other,
we have that ${\cal F}_{t}(\ov x,\ov y)$ remains between 
the orbit ${\cal F}_{t}(2x_0,0)$ and the characteristic $BC$, 
that is ${\cal F}_{t}(\ov x,\ov y)\in\ov{{(\Om_{\a,\b})}_-}$ 
for every $t\in[0,+\infty]$. This completes the proof.
\end{proof}
We now start to estimate the right-hand side of (\ref{3.8}).
For the sake of simplicity, in the sequel, for any 
$v:\ov{\Om_{\a,\b}}\to\rsp$, $\a>0>\b$, $\wtil v$ and $\what v$ 
denote the restrictions of $v$ to $BC$ and $\s_{\a,\b}$, respectively, i.e.
\beqn\label{4.18}
&&\hskip -2truecm
\wtil v:=v_{|BC},\qq \what v:=v_{|\s_{\a,\b}}.
\eeqn
As usual, for any pair ${\bf w}=(w_1,w_2)$, $|{\bf w}|$ stands for
its Euclidian norm $(w_1^2+w_2^2)^{1/2}$. Then, recalling that 
a curve $\G\subset\rsp^2$ is said {\it regular} if it admits a parameterization  
$r_\G:I\subset\rsp\to\rsp^2$, $I$ an interval, such that $r_\G\in C^1(I)$
and $r_\G'(t)\neq 0$ for every $t\in I$, 
we denote by $L^2(\G)$ the set of all real (respectively, complex)
valued functions $\psi$ such that 
$\|\psi\|_{L^2(\G)}^2=\int_\G \psi^2\d s<+\infty$
(respectively, $\|\psi\|_{L^2(\G)}^2=\int_\G |\psi|^2\d s<+\infty$), 
where $\d s=|r_\G'(t)|\d t$, $t\in I$. In particular, $\G$ is rectifiable if
and only if $\psi\equiv 1\in L^2(\G)$.
\begin{lemma}\label{lem4.4}
Let $\om_1$ be defined by formula $(\ref{3.9})$, where $u$
is a real valued function such that $|y|^{1/2}u_x$, $u_y\in L^2(BC)$, 
$BC$ being defined by $(\ref{2.3})$.
Then, for every $\ve>0$, the following estimate holds:
\beqn\label{4.19}
&&\hskip -1truecm
0\le \int_{BC}\om_1\d s
\le C_1(x_0,\ve)\||y|^{1/2}u_x\|_{L^2(BC)}^2+C_2(x_0,\ve)\|u_y\|_{L^2(BC)}^2,
\eeqn
where
\beqn\label{4.20}
&&\hskip -1truecm
C_j(x_0,\ve):=
\frac{6|x_0|\big(1+\ve^{(-1)^{j+1}}\big)}{[1+(3|x_0|/2)^{2/3}]^{1/2}},\q j=1,2.
\eeqn
\end{lemma}
\begin{proof}
First, replacing $\vec{\bf n}$ with the unit vector $\vec{\bf n}_{BC}$ 
defined by (\ref{4.8}) and using the second equality of (\ref{4.13}), 
formula (\ref{3.9}) simplifies to give
\beqn\label{4.21}
\hskip -1truecm
\wtil{\om_1}
&\!\!\!=\!\!\!&
\langle 2Du(-yu_x,-u_y),\vec{\bf n}_{BC}\rangle\no
\\[2mm]
&\!\!\!=\!\!\!&
2(1-y)^{-1/2}
\big\{3xy\wtil{u_x}^2+[2y^2-3x(-y)^{1/2}]\wtil{u_x}\wtil{u_y}
+2(-y)^{3/2}\wtil{u_y}^2\big\}.
\eeqn
According to (\ref{4.5}) we now replace $x$ with $-(2/3)(-y)^{3/2}$,
where $y\in[y_C,0]$. With a such substitution, from $(\ref{4.21})$ we easily find
\beqn\label{4.22}
\hskip -1truecm
\wtil{\om_1}
&\!\!\!=\!\!\!&
4(1-y)^{-1/2}\big[(-y)^{5/2}\wtil{u_x}^2
+2(-y)^2\wtil{u_x}\wtil{u_y}+(-y)^{3/2}\wtil{u_y}^2\big]\no
\\[2mm]
&\!\!\!=\!\!\!&
4(-y)^{3/2}(1-y)^{-1/2}\big[(-y)^{1/2}\wtil{u_x}+\wtil{u_y}\big]^2\ge 0.
\eeqn
Then, using the well-know inequality
$(a+b)^2\le (1+\ve)a^2+(1+\ve^{-1})b^2$, $a$, $b\in\rsp$, $\ve>0$, 
and observing that the function $p(y)=(-y)^{3/2}(1-y)^{-1/2}$ 
is decreasing for $y\le 0$, from (\ref{4.22}) we obtain
\beqn\label{4.23}
\hskip -1truecm
0&\!\!\!\le\!\!\!&\int_{BC}\om_1\d s\le
4\int_{y_C}^0(-y)^{3/2}(1-y)^{-1/2}
\big[(-y)^{1/2}\wtil{u_x}+\wtil{u_y}\big]^2|r_{BC}'(y)|\d y\no
\\[2mm]
\hskip -1truecm
&\!\!\!\le\!\!\!&
4(-y_C)^{3/2}(1-y_C)^{-1/2}
\big[(1+\ve)\||y|^{1/2}u_x\|_{L^2(BC)}^2+(1+\ve^{-1})\|u_y\|_{L^2(BC)}^2\big].
\eeqn
Replacing $y_C$ with $-(3|x_0|/2)^{2/3}$ in (\ref{4.23}),
the proof of (\ref{4.19}) is complete.
\end{proof}
\begin{lemma}\label{lem4.5}
Let us replace $\om_1$ and formula $(\ref{3.9})$ with $\om_2$
and formula $(\ref{3.10})$ in the hypothesis of {\rm Lemma \ref{lem4.4}}
and assume further that $u\in L^2(BC)$. Then, the following estimate holds:
\beqn\label{4.24}
&&\hskip -1truecm
\int_{BC}\om_2\d s\le 
C_3(x_0)\|u\|_{L^2(BC)}\big[\||y|^{1/2}u_x\|_{L^2(BC)}+\|u_y\|_{L^2(BC)}\big],
\eeqn
where
\beqn\label{4.25}
&&\hskip -2truecm
C_3(x_0):=\frac{(3|x_0|/2)^{1/3}}{[1+(3|x_0|/2)^{2/3}]^{1/2}}.
\eeqn
\end{lemma}
\begin{proof}
As in the proof of Lemma \ref{lem4.4},
replacing $\vec{\bf n}$ with the explicit form of $\vec{\bf n}_{BC}$ given by (\ref{4.8}) 
and using the second equation in (\ref{4.13}), we simplify (\ref{3.10}) to
\beqn\label{4.26}
&&\hskip -1truecm
\wtil\om_2=\langle-\wtil u(-y\wtil{u_x},-\wtil{u_y}),\vec{\bf n}_{BC}\rangle
=-(-y)^{1/2}(1-y)^{-1/2}\wtil u\big[(-y)^{1/2}\wtil{u_x}+\wtil{u_y}\big].
\eeqn
Therefore, applying H\"older's inequality and observing that
the function $q(y)=-y(1-y)^{-1}$ is decreasing for $y\le 0$, 
from (\ref{4.26}) it follows
\beqn\no
\hskip 0truecm
\int_{BC}\om_2\d s
&\!\!\!=\!\!\!&
\int_{y_C}^0\big[-(-y)^{1/2}(1-y)^{-1/2}\wtil u\,|r_{BC}'(y)|^{1/2}\big]
\big[(-y)^{1/2}\wtil{u_x}+\wtil{u_y}\big]|r'_{BC}(y)|^{1/2}\d y\no
\\[2mm]
&\!\!\!\le\!\!\!&
\Big(\int_{y_C}^0(-y)(1-y)^{-1}\wtil u^{\,2}|r_{BC}'(y)|\d y\Big)^{1/2}
\big[\||y|^{1/2}u_x\|_{L^2(BC)}+\|u_y\|_{L^2(BC)}\big]\no
\\[2mm]
&\!\!\!\le\!\!\!&
(-y_C)^{1/2}(1-y_C)^{-1/2}\|u\|_{L^2(BC)}
\big[\||y|^{1/2}u_x\|_{L^2(BC)}+\|u_y\|_{L^2(BC)}\big].\no
\eeqn
This completes the proof.
\end{proof}
\begin{remark}\label{rem4.6}
\emph{Observe that, contrarily to Lemma \ref{lem4.4} where (\ref{4.22})
implies $\int_{BC}\om_1\d s\ge 0$, 
in Lemma \ref{lem4.5} we cannot ensure $\int_{BC}\om_2\d s\ge 0$, since
(\ref{4.26}) may change sign. However, as we shall see later, 
when $f(u)=\l u$ in problem (\ref{3.1}) 
with $\l\in\s(\wtil T_{AC\cup\s})\cap[\l_0,+\infty)$ (see (\ref{2.9})),
the nonnegativity of the sum of integrals on the right-hand side of (\ref{3.8})
will be a consequence of that of the left-hand side. 
Of course, this agrees with the obvious fact
that the sum $\int_{BC}(\om_1+\om_2)\d s+\int_{\s}\om_1\d s$ 
may be nonnegative even though some of its terms are nonpositive. 
Notice also that we cannot use the result in \cite[pp. 416, 417]{LP5} 
which establishes $\int_{BC}(\om_1+\om_2)\d s\ge 0$,
since there it is assumed that the function $\varphi(y)=u(r_{BC}(y))$
belongs to $C^2((y_C,0))\cap C^1([y_C,0])$, which is not our case.}
\end{remark}
We now turn our attention to the last term on the right-hand side of (\ref{3.4}), 
i.e. the integral of $\om_1$ along the elliptic normal arc $\s$,
where $\s=\s_{\a,\b}$ for some $\a>0>\b$.
For our purposes, we shall need to prove some preliminary technical results. 
To motivate them, we first observe that recalling formula (\ref{4.11}) 
and notation (\ref{4.18}), from definition (\ref{3.9}) of $\om_1$ we get:
\beqn\label{4.27}
\hskip 0,2truecm
\widehat{\om_1}=[h_{\a,\b}(x)]^{-1}(I_1+I_2),
\eeqn
where
\beqn
&&\hskip -1truecm
I_1:=(6x\what{u_x}+4y\what{u_y})
\big[(x-x_0)y\what{u_x}+|\b|^{-1}[g_{\a,\b}(x)]^\a\what{u_y}\big],\no
\\[2mm]
&&\hskip -1truecm
I_2:=-(y\what{u_x}^2+\what{u_y}^2\big)
\big[3x(x-x_0)+2|\b|^{-1}[g_{\a,\b}(x)]^\a y\big].\no
\eeqn
Expanding $I_1$ and $I_2$, from (\ref{4.27}) we thus find
\beqn
\hskip -0truecm
\widehat{\om_1}
&\!\!\!=\!\!\!&
[h_{\a,\b}(x)]^{-1}
\big\{3x(x-x_0)-2|\b|^{-1}[g_{\a,\b}(x)]^\a y\big\}(y\what{u_x}^2-\what{u_y}^2)\no
\\[1mm]
\hskip -0truecm
&&
+[h_{\a,\b}(x)]^{-1}\big\{6|\b|^{-1}x[g_{\a,\b}(x)]^\a y^{-1/2}+4(x-x_0)y^{3/2}\big\}
y^{1/2}\what{u_x}\,\what{u_y}.\no
\eeqn
Therefore, using $y=g_{\a,\b}(x)$ on $\s_{\a,\b}$ and 
$2|\b|^{-1}[g_{\a,\b}(x)]^{\a+1}=(\a+1)x(2x_0-x)$, 
the latter equality reduces to
\beqn\label{4.28}
\widehat{\om_1}
&\!\!\!=\!\!\!&
[h_{\a,\b}(x)]^{-1}x[(4+\a)x-(5+2\a)x_0](y\what{u_x}^2-\what{u_y}^2)\no
\\[1mm]
\hskip -0truecm
&&
+[h_{\a,\b}(x)]^{-1}\big\{6|\b|^{-1}x[g_{\a,\b}(x)]^{(2\a-1)/2}+4(x-x_0)[g_{\a,\b}(x)]^{3/2}\big\}y^{1/2}\what{u_x}\,\what{u_y}.
\eeqn
Let us rewrite (\ref{4.28}) in a more compact form. 
From now on, for every $\a>0>\b$ and $x\in[2x_0,0]$, 
we denote with $d_\a$ and $\theta_\a$, $\Theta_{\a,\b}$,
$\varphi_{\a,\b}$, $\phi_{\a,\b}$, $\psi_{\a,\b}$ and $\Psi_{\a,\b}$ 
the positive number and the functions defined, respectively, by
\beqn\label{4.29}
&&\hskip -1truecm
d_\a:=\frac{5+2\a}{4+\a},
\q\ 
\theta_{\a}(x):=(4+\a)x(x-d_\a x_0),\q\ 
\Theta_{\a,\b}(x):=[h_{\a,\b}(x)]^{-1}\theta_\a(x),
\\[1mm]
\label{4.30}
&&\hskip -1truecm
\varphi_{\a,\b}(x):=6|\b|^{-1}x[g_{\a,\b}(x)]^{(2\a-1)/2},\q
\phi_{\a,\b}(x):=4(x-x_0)[g_{\a,\b}(x)]^{3/2},
\\[2mm]
\label{4.31}
&&\hskip -1truecm
\psi_{\a,\b}(x):=\varphi_{\a,\b}(x)+\phi_{\a,\b}(x),\q\ 
\Psi_{\a,\b}(x):=[h_{\a,\b}(x)]^{-1}\psi_{\a,\b}(x).
\eeqn
With this notation, we can thus rewrite (\ref{4.28}) as
\beqn\label{4.32}
&&\hskip -1truecm
\widehat{\om_1}=
\Theta_{\a,\b}(x)y\what{u_x}^2+\Psi_{\a,\b}(x)y^{1/2}\what{u_x}\,\what{u_y}
-\Theta_{\a,\b}(x)\what{u_y}^2.
\eeqn
This latter equality suggests us that, in order to estimate 
the integral of $\om_1$ along $\s=\s_{\a,\b}$ in terms
of $\|y^{1/2}u_x\|_{L^2(\s)}$ and $\|u_y\|_{L^2(\s)}$,
we first need to find upper and lower bounds of the functions
$\Theta_{\a,\b}$ and $\Psi_{\a,\b}$ defined in (\ref{4.29}) and (\ref{4.31}).
Since $\Theta_{\a,\b}$ and $\Psi_{\a,\b}$ both depend on the reciprocal of the
positive continuous function $h_{\a,\b}$ defined by (\ref{4.12}), these bounds 
will be easily obtained once we shall be able to determine how 
the least and greatest values of $h_{\a,\b}$ vary with $\a$, $\b$ and $x_0$.  
Notice that, if $\a\in(0,1/2)$, then $\varphi_{\a,\b}$, 
and consequently $\Psi_{\a,\b}$,  blows-up to $-\infty$ at $x=2x_0$. 
However, since in the following we shall need to deal with the case in which
$\partial\Om=\partial\Om_{\a,\b}$ is $D$-starlike, then,
according to Lemma \ref{lem4.2}, we shall be interested in estimating
the function $\Psi_{\a,\b}$ only for $\a\ge1/2$.

To proceed in our analysis, we introduce some further notation.
Throughout the rest of the paper, $c_{\a,\b}$, $M_{\a,\b}$ and $D_{\a,\b}$, 
$\a>0>\b$, $\a\neq 1$, denote the positive numbers
\beqn\label{4.33}
&&\hskip -1truecm
c_{\a,\b}:=2(\a+1)^{-1}\a^{-\g_\a}|\b|^{\g_\a-1},\q \g_\a=(\a+1)/(\a-1),
\\[2mm]
\label{4.34}
&&\hskip -1truecm
M_{\a,\b}:=(c_{\a,\b})^{1/2},\q
D_{\a,\b}:=[2^\a(\a+1)^{-\a}|\b|]^{1/(\a-1)}.
\eeqn
In particular, $M_{\a,\b}$ is smaller than $D_{\a,\b}$. 
Indeed, due to (\ref{4.33}), $M_{\a,\b}<D_{\a,\b}$ is equivalent to
\beqn\label{4.35}
&&\hskip -1truecm
2^{1/2}(\a+1)^{-{1/2}}\a^{-(\a+1)/[2(\a-1)]}|\b|^{1/(\a-1)}
<[2^\a(\a+1)^{-\a}|\b|]^{1/(\a-1)}.
\eeqn
Dividing both sides of (\ref{4.35}) by $|\b|^{1/(\a-1)}$ 
and passing to the logarithm we are led to 
\beqn\no
&&\hskip -1truecm
\frac{\a+1}{2(\a-1)}\ln\Big(\frac{1}{\a}\Big)<
\Big[\frac{\a}{\a-1}-\frac{1}{2}\Big]
\ln\Big(\frac{2}{\a+1}\Big)=\frac{\a+1}{2(\a-1)}\ln\Big(\frac{2}{\a+1}\Big),
\eeqn
and this latter inequality is satisfied for every $\a\in(0,1)\cup(1,+\infty)$. 
\begin{remark}\label{rem4.7}
\emph{For instance, in the case $(\a,\b)=(2,-3/2)$ of the normal Tricomi
curve $g_{\a,\b}=g_T$, from (\ref{4.33}) and (\ref{4.34}) 
we derive $c_{\a,\b}=3/16$, $M_{\a,\b}={\sqrt 3}/4$ and $D_{\a,\b}=2/3$.}
\end{remark}
For $x_0<-M_{\a,\b}$ we then define the points 
$x_{\a,\b,\pm}\in(2x_0,0)$, symmetric with respect to $x=x_0$, by
\beqn\label{4.36}
&&\hskip -0,5truecm
x_{\a,\b;\pm}:=x_0\pm{(x_0^2-c_{\a,\b})}^{1/2},\qq \a>0>\b,\ \a\neq 1.
\eeqn
Observe that, for every fixed $\b<0$, the points $x_{\a,\b,-}$ and $x_{\a,\b,+}$
approach $0$ and $2x_0$, respectively, when $\a\to+\infty$. 
For, from (\ref{4.33}) it follows $c_{\a,\b}\to 0$ as $\a\to+\infty$.
Hence, from the definitions (\ref{4.3}) and (\ref{4.12}) of the functions
$g_{\a,\b}$ and $h_{\a,\b}$ it follows that
\beqn\no
\hskip -1truecm
g_{\a,\b}(x_{\a,\b;\pm})
&\!\!\!=\!\!\!&
[2^{-1}(\a+1)|\b|c_{\a,\b}]^{1/(\a+1)}=(\a^{-1}|\b|)^{1/(\a-1)},
\\[2mm]
\label{4.37}
\hskip -1truecm
h_{\a,\b}(x_{\a,\b;\pm})
&\!\!\!=\!\!\!&
\{x_0^2-c_{\a,\b}+\a^{-2\a/(\a-1)}|\b|^{2/(\a-1)}\}^{1/2}=:C_4(\a,\b,x_0).
\eeqn
Finally, when $x_0<M_{\a,\b}$, we set
\beqn\label{4.38}
C_5(\a,\b,x_0):=\left\{\!
\begin{array}{lll}
\max\{h_{\a,\b}(x_0),h_{\a,\b}(0)\},\q {\rm if}\ \a>1,\ \b<0,
\\[2mm]
\min\{h_{\a,\b}(x_0),h_{\a,\b}(0)\},\q {\rm if}\ \a\in(0,1),\ \b<0.
\end{array}
\right.
\eeqn
Precisely, using $g_{\a,\b}(x_0)=[(\a+1)|\b|x_0^2/2]^{1/(\a+1)}$ and 
comparing the values $h_{\a,\b}(x_0)$ and $h_{\a,\b}(0)$, we find
\beqn\label{4.39}
&&\hskip -1truecm
C_5(\a,\b,x_0)=
\left\{\!
\begin{array}{lll}
h_{\a,\b}(0)=|x_0|,\q{\rm if}\ x_0\in(-D_{\a,\b},-M_{\a,\b}),
\\[2mm]
h_{\a,\b}(x_0)=h_{\a,\b}(0)=|x_0|,\q{\rm if}\ x_0=-D_{\a,\b},
\\[2mm]
h_{\a,\b}(x_0)= |\b|^{-1}[g_{\a,\b}(x_0)]^\a,\q{\rm if}\ x_0<-D_{\a,\b}.
\end{array}
\right.
\eeqn 
\begin{figure}
\begin{center}
\includegraphics[width=1.43in]{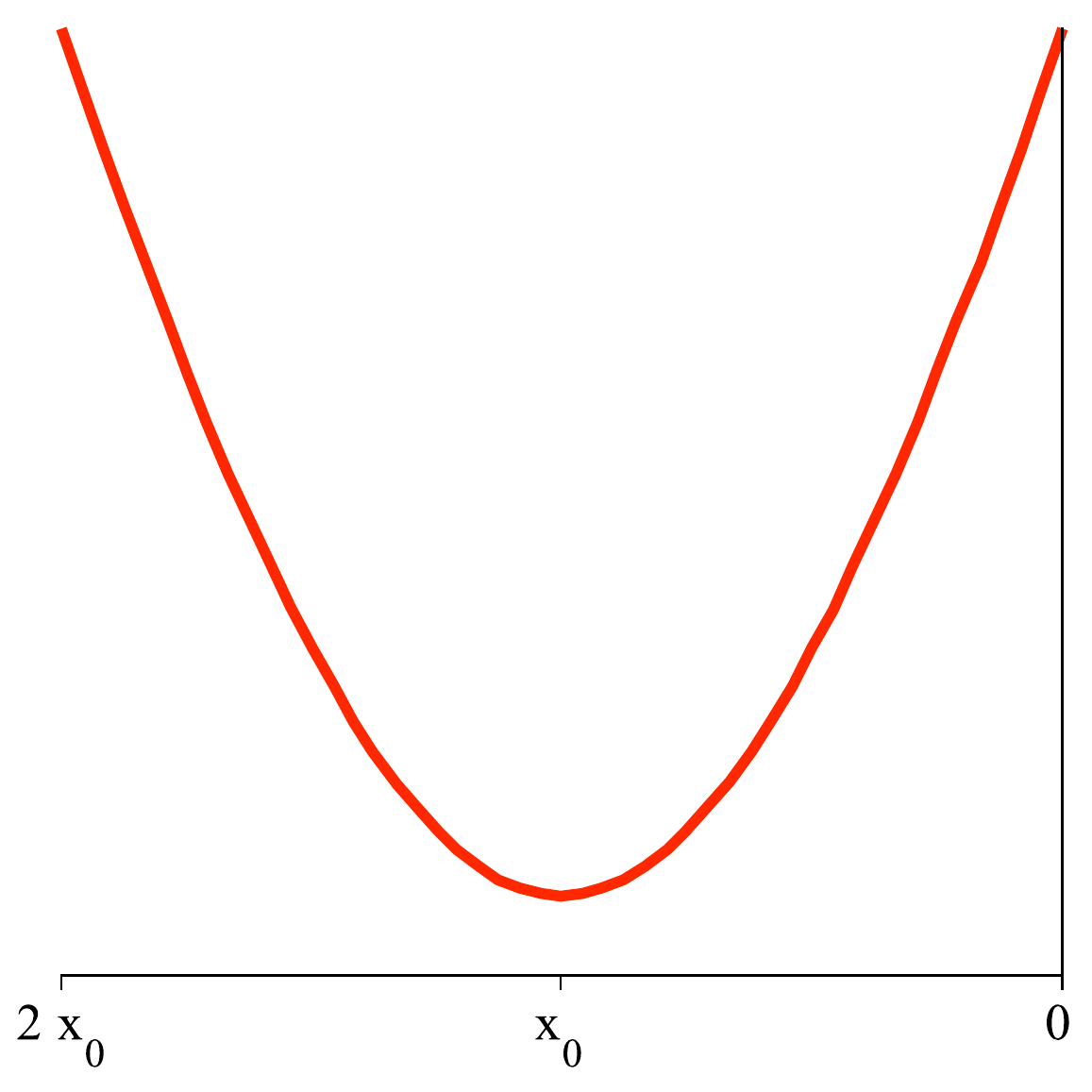}
\q
\includegraphics[width=1.43in]{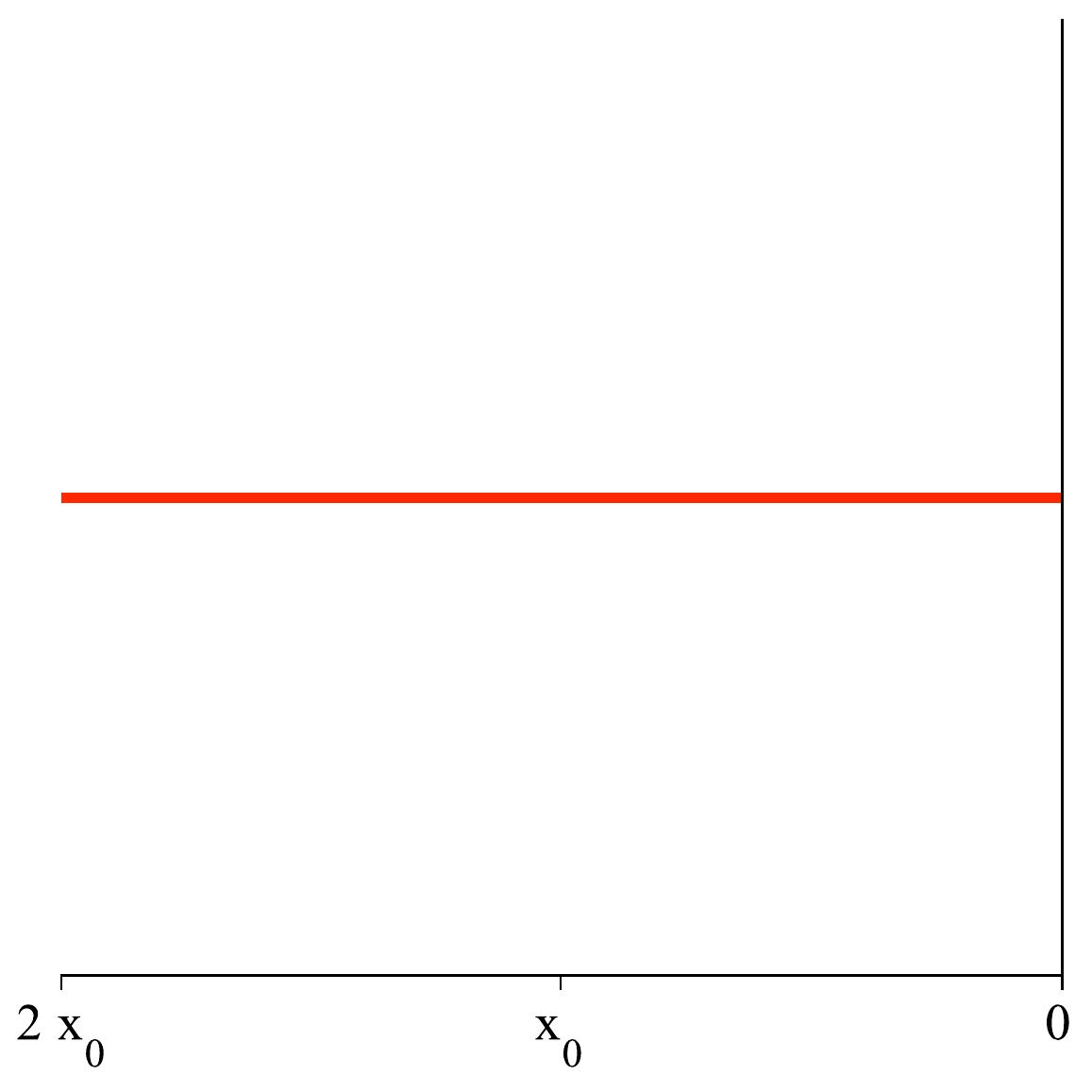}
\q
\includegraphics[width=1.43in]{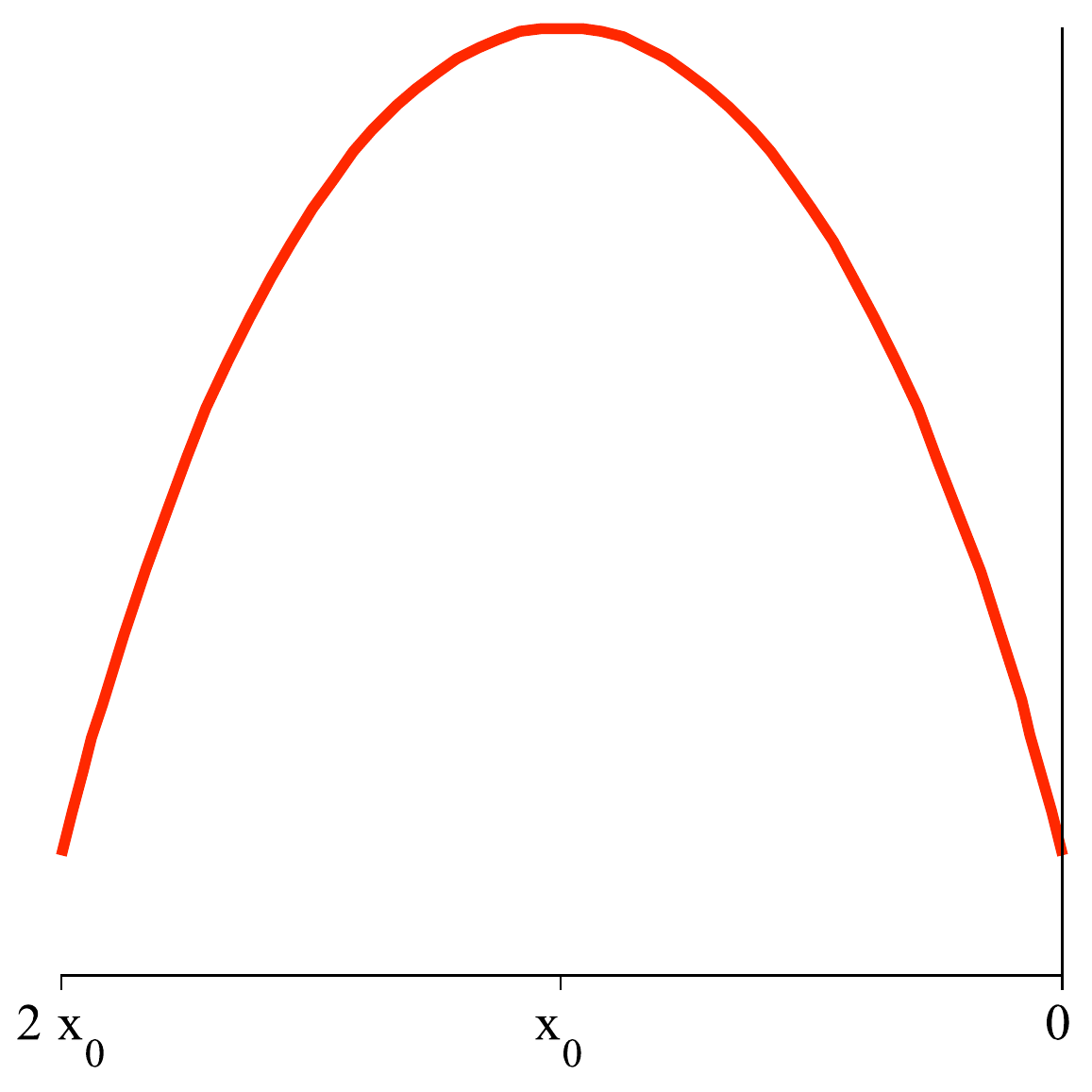}
\caption{{From the left, the function $h_{1,\b}$ 
for $\b<-1$, $\b=-1$ and $\b\in(-1,0)$.}}
\end{center}
\begin{center}
\includegraphics[width=1.43in]{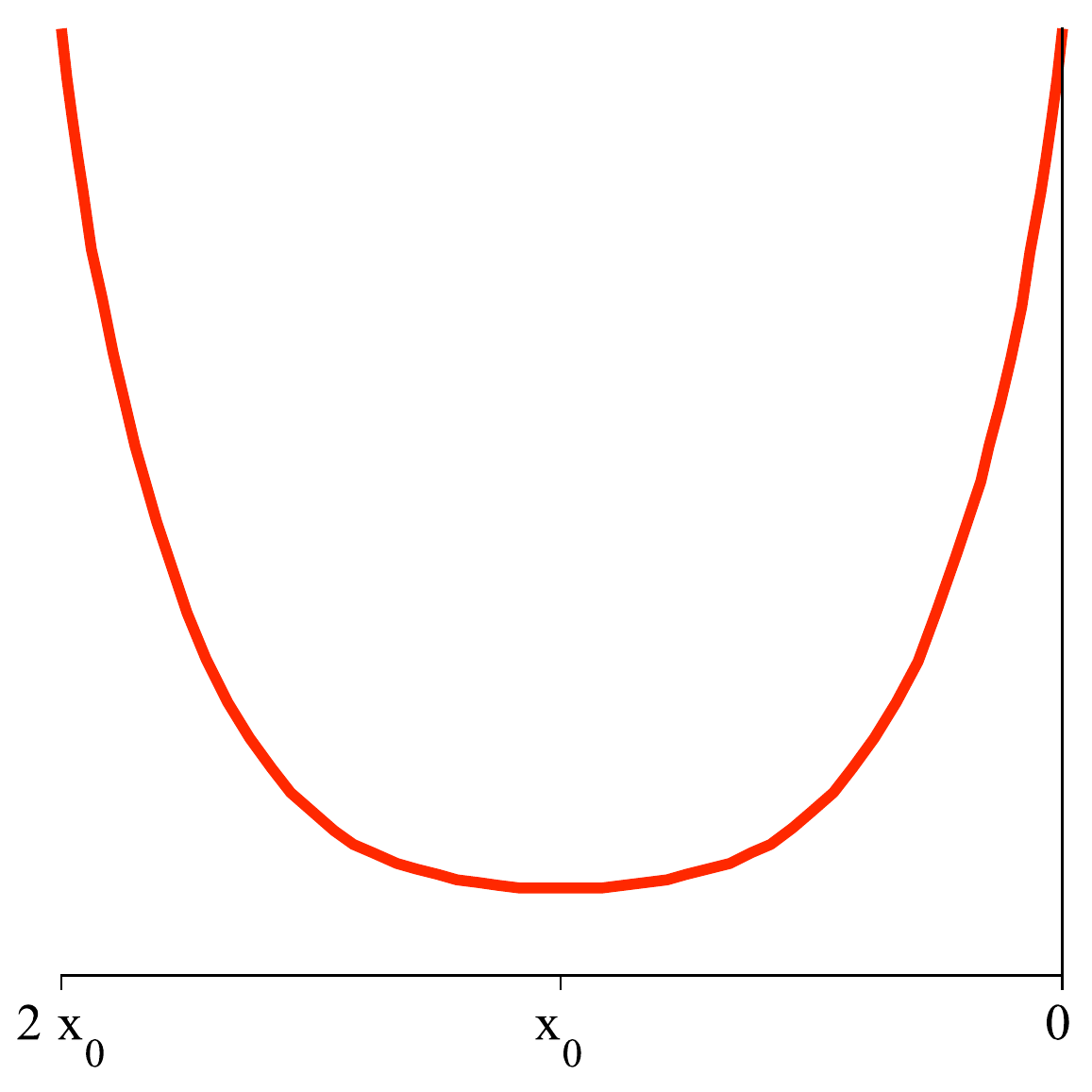}
\q
\includegraphics[width=1.43in]{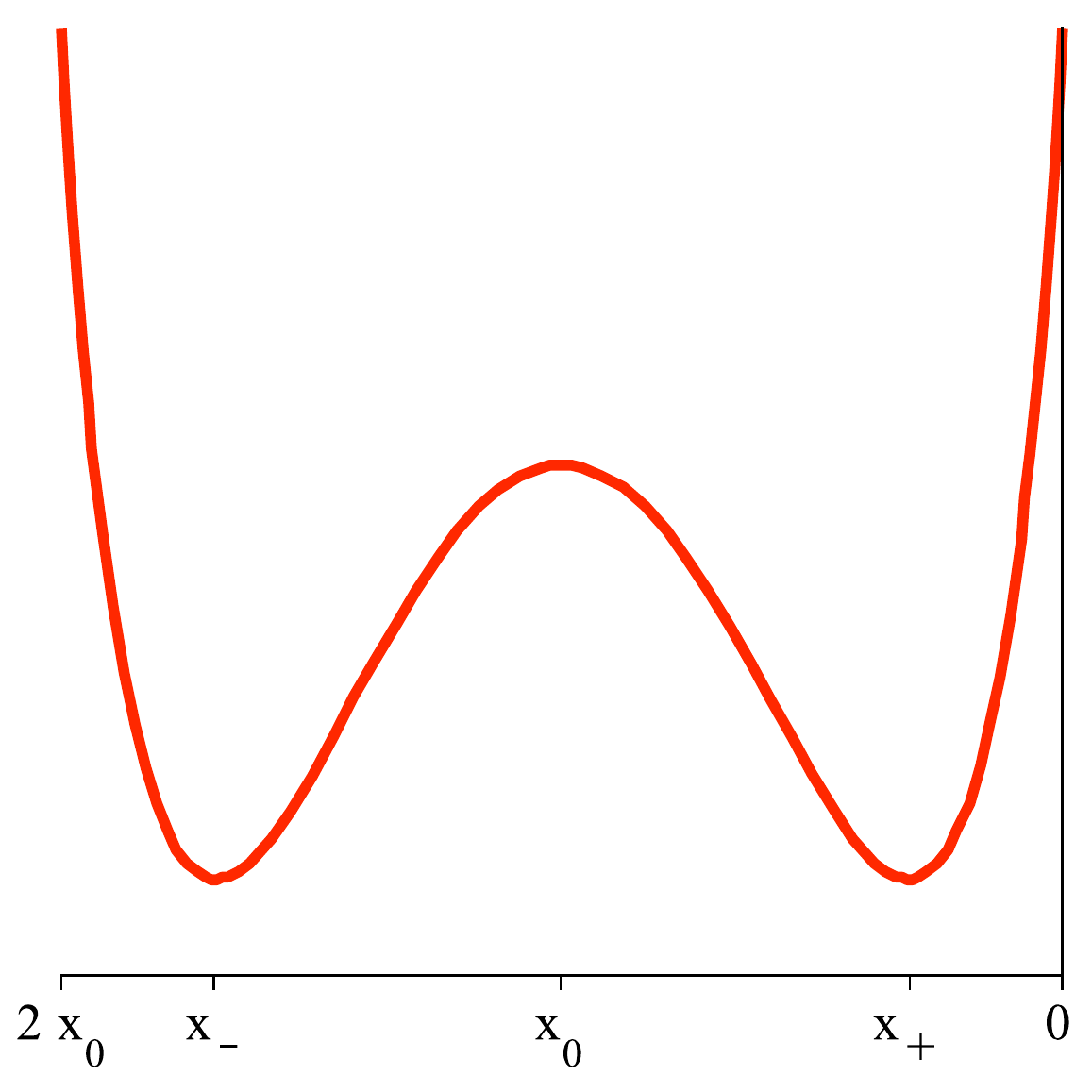}
\q
\includegraphics[width=1.43in]{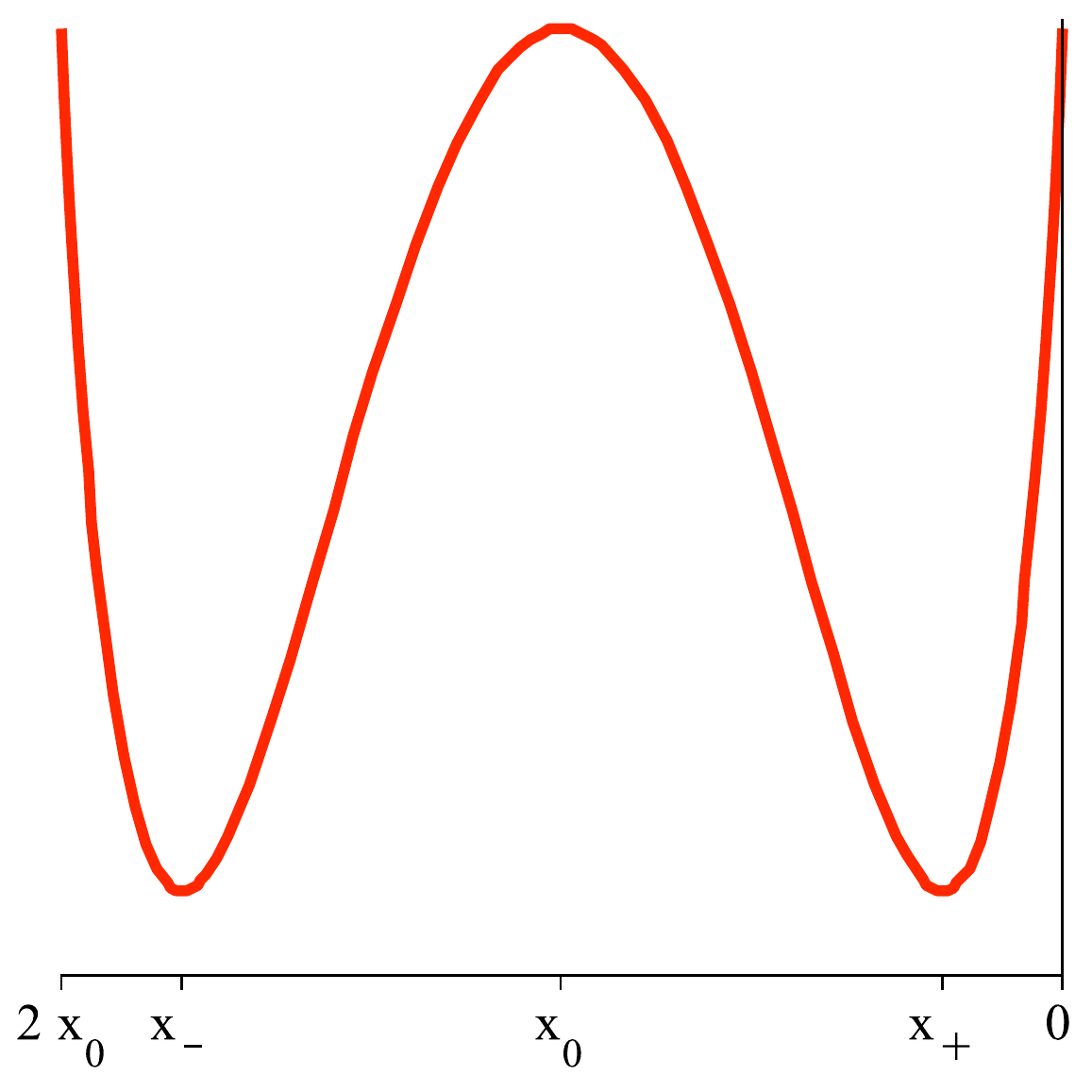}
\q
\includegraphics[width=1.43in]{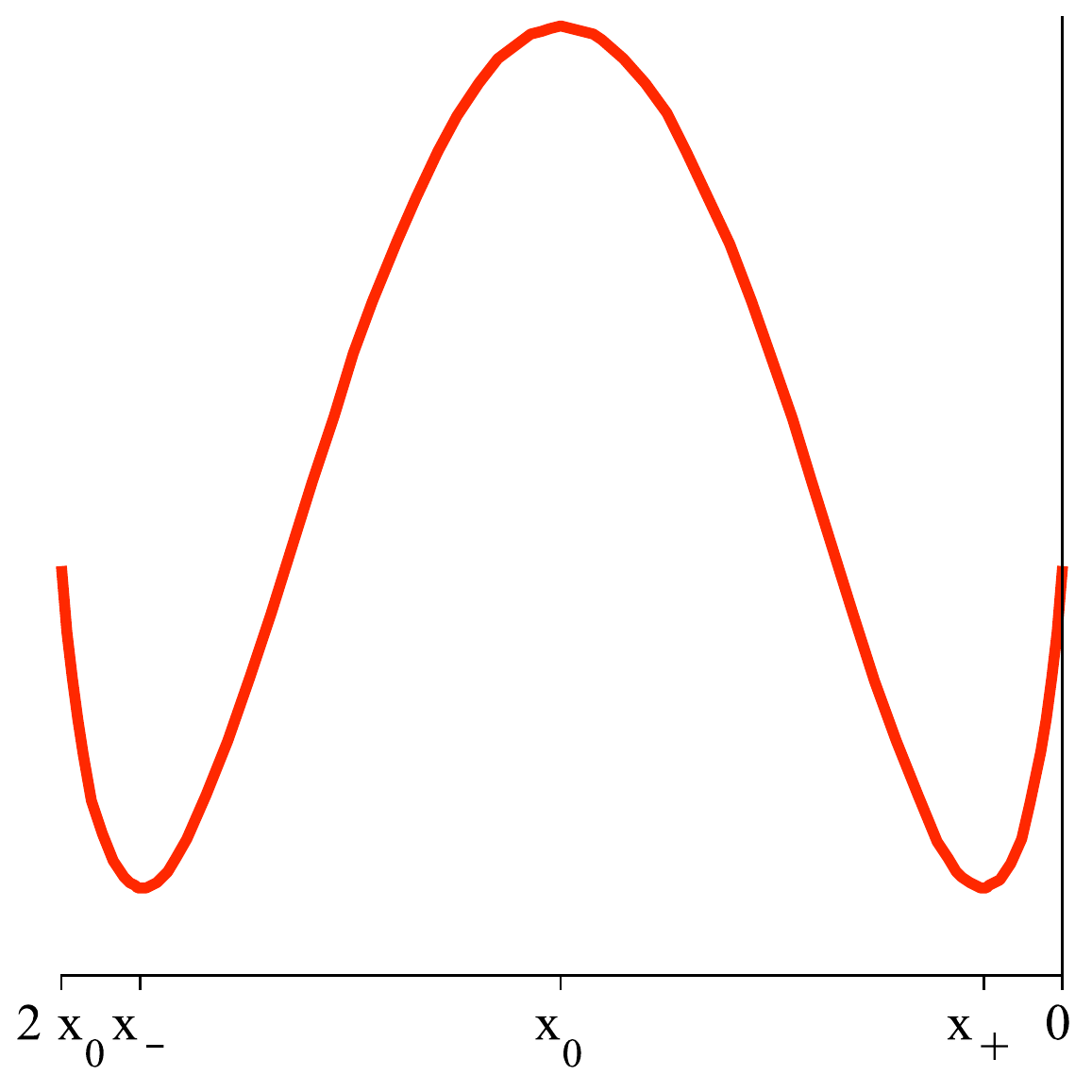}
\caption{From the left, the function $h_{\a,\b}$, $\a>1$, $\b<0$, 
for $x_0\in[-M_{\a,\b},0)$, $x_0\in(-D_{\a,\b},-M_{\a,\b})$, $x_0=-D_{\a,\b}$
and $x_0<-D_{\a,\b}$.}
\end{center}
\begin{center}
\includegraphics[width=1.43in]{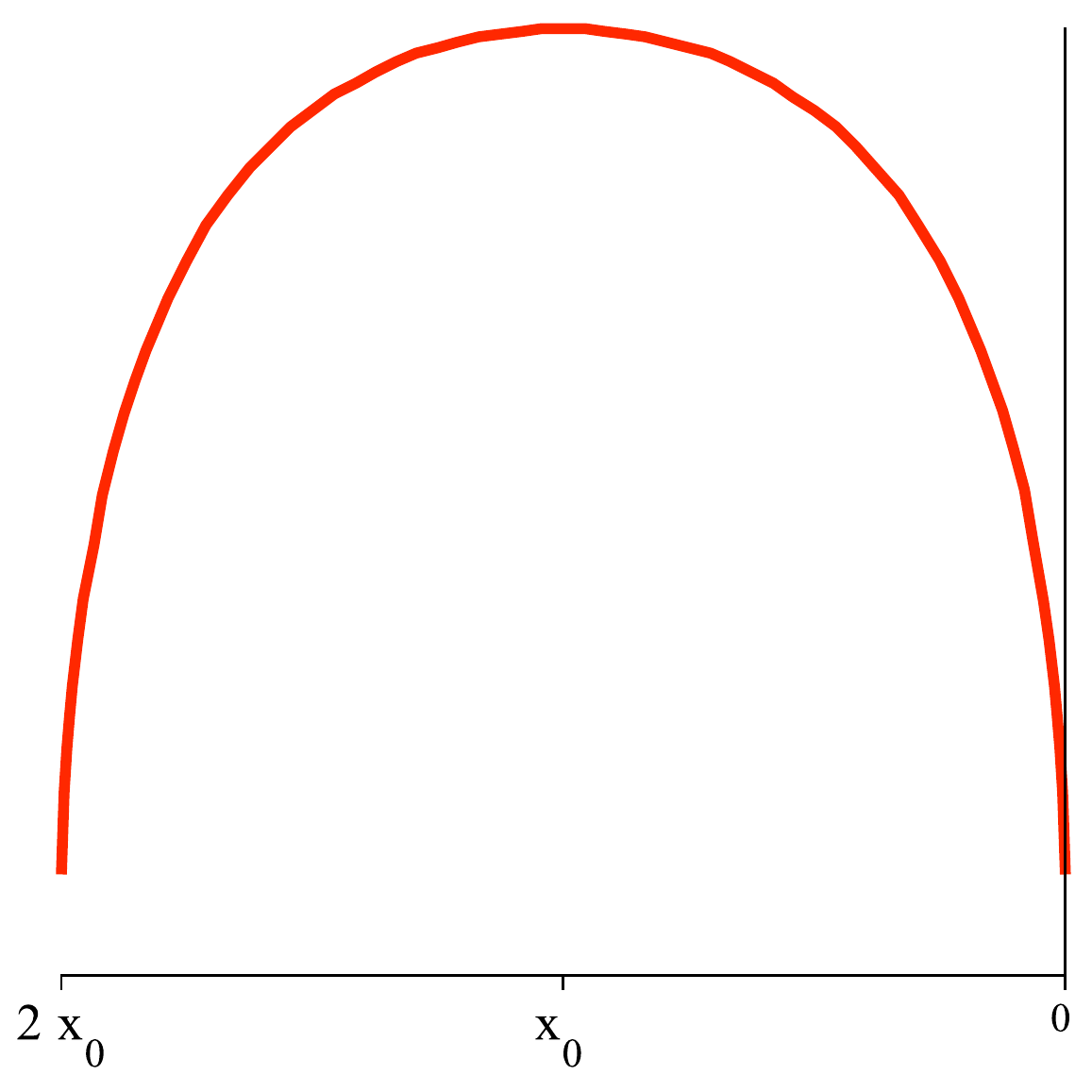}
\q
\includegraphics[width=1.43in]{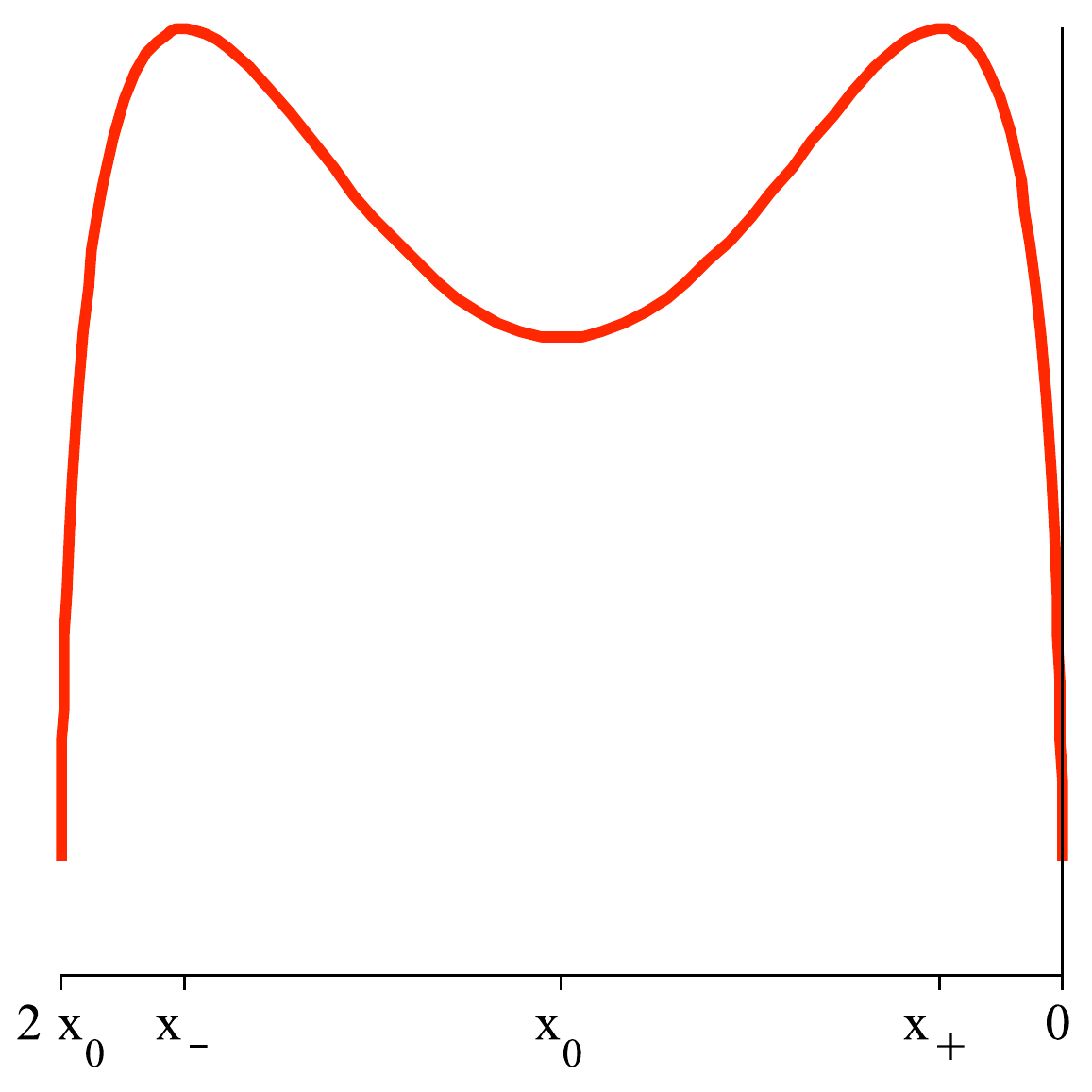}
\q
\includegraphics[width=1.43in]{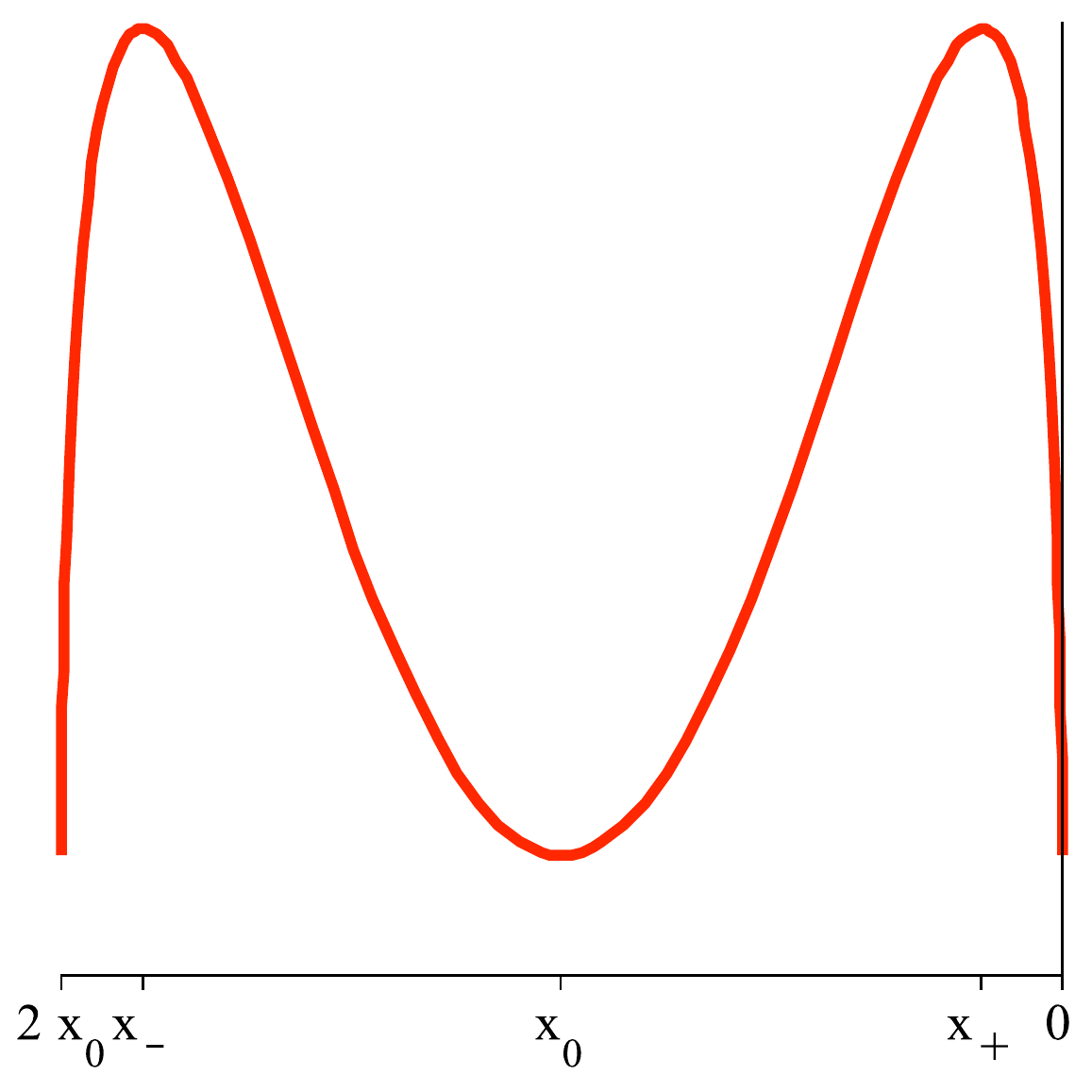}
\q
\includegraphics[width=1.43in]{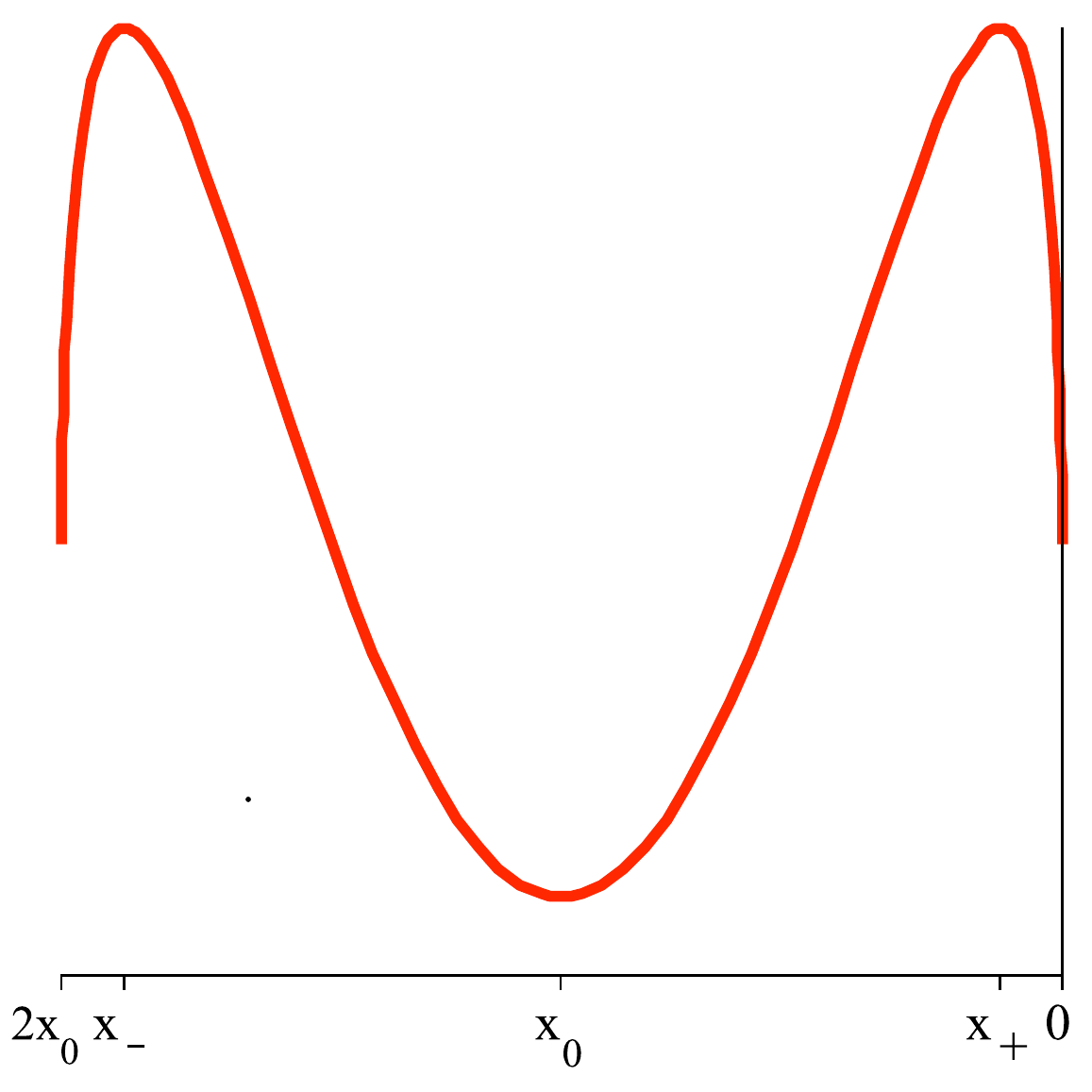}
\caption{From the left, the function $h_{\a,\b}$, $\a\in(0,1)$, $\b<0$, 
for $x_0\in[-M_{\a,\b},0)$, $x_0\in(-D_{\a,\b},-M_{\a,\b})$, $x_0=-D_{\a,\b}$
and $x_0<-D_{\a,\b}$.}
\end{center}
\end{figure}
A detailed analysis of the function $h_{\a,\b}$ 
then yields the next lemma which we shall prove in Section \ref{Sec5}.
\begin{lemma}\label{lem4.8}
Let $h_{\a,\b}$, $\a>0>\b$, be the positive continuous function 
defined by $(\ref{4.12})$ and which is even with respect to the line $\{x=x_0\}$. 
i) Let $\a=1$. If $\b\le-1$ (respectively, $\b\in(-1,0)$), 
then $h_{1,\b}$ is a convex (respectively, concave) function, 
non decreasing (respectively, decreasing)
in the interval $[x_0,0]$. Therefore, for every $x\in[2x_0,0]$ we have
\beqn\label{4.40}
&&\hskip -1,5truecm
\left\{\!\!
\begin{array}{lll}
|\b|^{-1/2}|x_0|=h_{1,\b}(x_0)
\le h_{1,\b}(x)\le h_{1,\b}(0)=|x_0|,\q {\rm if}\ \b\le-1
\\[2mm]
|x_0|=h_{1,\b}(0)
\le h_{1,\b}(x)\le h_{1,\b}(x_0)=|\b|^{-1/2}|x_0|,\q{\rm if}\ \b\in(-1,0).
\end{array}
\right.
\eeqn
ii) Let $\a>1$. 
ii-a) If $x_0\in[-M_{\a,\b},0)$, then $h_{\a,\b}$ is a convex function, 
non decreasing in the interval $[x_0,0]$. Therefore
\beqn\label{4.41}
&&\hskip -1truecm
0< |\b|^{-1}[g_{\a,\b}(x_0)]^\a=h_{\a,\b}(x_0)
\le h_{\a,\b}(x)\le h_{\a,\b}(0)=|x_0|,\q\forall\,x\in[2x_0,0].
\eeqn
ii-b) If $x_0<-M_{\a,\b}$, then 
\beqn\label{4.42}
&&\hskip -1truecm
0<C_4(\a,\b,x_0)\le h_{\a,\b}(x)\le C_5(\a,\b,x_0),\q\forall\,x\in[2x_0,0],
\eeqn
$C_j(\a,\b,x_0)$, $j=4,5$, being defined by $(\ref{4.37})$ and
the first line in $(\ref{4.38})$, respectively. Moreover, 
there exists a unique inflection point $\wtil x_{\a,\b}\in (x_0,x_{\a,\b;+})$
such that $h_{\a,\b}$ is convex in $[2x_0,2x_0-\wtil x_{\a,\b}]\cup[\wtil x_{\a,\b},0]$
and concave in $[2x_0-\wtil x_{\a,\b},\wtil x_{\a,\b}]$.

\noindent
iii) Let $\a\in(0,1)$. 
iii-a) If $x_0\in[-M_{\a,\b},0)$, then $h_{\a,\b}$ is a 
concave function, non increasing in the interval $[x_0,0]$. 
Therefore
\beqn\label{4.43}
&&\hskip -1truecm
0< |x_0|=h_{\a,\b}(0)\le h_{\a,\b}(x)\le 
h_{\a,\b}(x_0)=|\b|^{-1}[g_{\a,\b}(x_0)]^\a,\q\forall\,x\in[2x_0,0].
\eeqn
iii-b) If $x_0<-M_{\a,\b}$, then 
\beqn\label{4.44}
&&\hskip -1truecm
0<C_5(\a,\b,x_0)\le h_{\a,\b}(x)\le C_4(\a,\b,x_0),\q\forall\,x\in[2x_0,0],
\eeqn
$C_j(\a,\b,x_0)$, $j=4,5$, being defined by $(\ref{4.37})$ and
the second line in $(\ref{4.38})$, respectively.
Moreover, there exists a unique inflection point 
$\ov x_{\a,\b}\in (x_0,x_{\a,\b;+})$
such that $h_{\a,\b}$ is concave in 
$[2x_0,2x_0-\ov x_{\a,\b}]\cup[\ov x_{\a,\b},0]$
and convex in $[2x_0-\ov x_{\a,\b},\ov x_{\a,\b}]$.
\end{lemma}
With the help of Lemma \ref{lem4.8} we can now find upper and lower bound
for the functions $\Theta_{\a,\b}$ and $\Psi_{\a,\b}$ in (\ref{4.32}).
As for Lemma \ref{lem4.8}, the proofs of the forthcoming 
Corollary \ref{cor4.10}, Lemma \ref{lem4.13} and Corollary \ref{cor4.16} 
will be given in Section \ref{Sec5}.
In the sequel, $d_\a$ being defined in (\ref{4.29}),
$E_{\a,\b}$ and $L_{\a,\b}$, $\a>0>\b$, $\a\neq 1$, 
denote the positive numbers
\beqn\no
&&\hskip -1truecm
E_{\a,\b}:=\frac{M_{\a,\b}}{[d_\a(2-d_\a)]^{1/2}},\qq
L_{\a,\b}:=\frac{D_{\a,\b}}{[d_\a(2-d_\a)]^{1/2}}.
\eeqn
Of course, since $d_\a\in(5/4,2)$, the numbers $E_{\a,\b}$ and $L_{\a,\b}$
are greater than $M_{\a,\b}$ and $D_{\a,\b}$, respectively.
Also, in the case $\a\in(0,1)$, it is easy to see that $E_{\a,\b}$ 
is smaller than $D_{\a,\b}$. In fact,  due to (\ref{4.33}) and (\ref{4.34}), 
the inequality $D_{\a,\b}\le E_{\a,\b}$, $\a\in(0,1)$, is equivalent to
\beqn\no
&&\hskip -1truecm
[2^\a(\a+1)^{-\a}]^{1/(\a-1)}\le
2^{1/2}(\a+1)^{-{1/2}}\a^{-(\a+1)/[2(\a-1)]}[d_\a(2-d_\a)]^{-1/2},\q \a\in(0,1).
\eeqn
Using $d_\a(2-d_a)=3(5+2\a)/(4+\a)^2$ and passing to the logarithm,
the latter inequality reduces to 
\beqn\label{4.45}
\mu(\a):=\frac{\a+1}{1-\a}\ln\Big(\frac{\a+1}{2\a}\Big)\le
\ln\Big[\frac{(4+\a)^2}{3(5+2\a)}\Big]=:\nu(\a),\q\a\in(0,1).
\eeqn
Now, standard computations show that the function $\mu(\a)$
on the left-hand side of (\ref{4.45}) is decreasing in $(0,1)$
and satisfies $\lim_{\a\to 0^+}\mu(a)=+\infty$ and $\lim_{\a\to 1^-}\mu(\a)=1$,
so that $1<\mu(\a)$ for every $\a\in(0,1)$. On the contrary, the function
$\nu(\a)$ on the right-hand side of (\ref{4.45}) is increasing in $(0,1)$,
so that $\ln(16/15)=\nu(0)<\nu(\a)<\nu(1)=\ln(25/21)$, $\a\in(0,1)$.
Since $\ln(25/21)<1$, this yields $\nu(\a)<\mu(\a)$, 
which contradicts (\ref{4.45}) and proves $E_{\a,\b}<D_{\a,\b}$, $\a\in(0,1)$.
\begin{remark}\label{rem4.9}
\emph{In the case $(\a,\b)=(2,-3/2)$ of the normal Tricomi
curve we have $d_\a=3/2$ and $E_{\a,\b}=1/2$. 
This suggests that, since (see Remark \ref{rem4.7}) $D_{2,-3/2}=2/3$, 
the inequality $E_{\a,\b}<D_{\a,\b}$, $\a\neq 1$, 
is satisfied in a real range larger than $\a\in(0,1)$.
However, in the following we shall not need to know until when the inequality
$E_{\a,\b}<D_{\a,\b}$ holds true in the case $\a>1$.} 
\end{remark}
We now specified two points of the interval $(2x_0,0)$ 
that appear in the upper and lower bounds of the function $\Theta_{\a,\b}$
given in the next Corollary \ref{cor4.10}.
For every $\a>0$ we set
\beqn\label{4.46}
&&\hskip -1truecm
x_\a:=d_\a x_0\in(2x_0,x_0),\qq \wtil x_\a:=x_\a/2\in(x_0,x_0/2).
\eeqn
In particular, the function $\theta_\a$ in (\ref{4.29}) rewrites as 
$\theta_{\a}(x):=(4+\a)x(x-x_\a)$, so it is
positive in $[2x_0,x_\a)$ and negative in $(x_\a,0)$. 
Observe that, if $x_0<-M_{\a,\b}$, then it holds
$x_\a<x_{\a,\b;-}$ or $x_{\a,\b;-}\le x_\a$ according to the case 
$x_0\in(-E_{\a,\b},-M_{\a,\b})$ or $x_0\le-E_{\a,\b}$, respectively.
Therefore, since in Corollary \ref{cor4.10} 
we are interested in the least value of the function $h_{\a,\b}$ 
in the interval $[2x_0,x_\a]$, 
from Lemma \ref{lem4.8} we see that the
only ambiguous case is $x_0<-D_{\a,\b}$, $\a\in(0,1)$.
(see also Figures 2 and 3). In fact, due to what observed before,
if $x_0<-D_{\a,\b}$, $\a\in(0,1)$, then $x_0<-E_{\a,\b}$ 
and $x_\a$ belongs to the interval $(x_{\a,\b;-},x_0)$, 
in which $h_{\a,\b}$ decreases. Thus, to determine the least value 
of $h_{\a,\b}$ in $[2x_0,x_\a]$ we have to compare the values 
$h_{\a,\b}(2x_0)$ and $h_{\a,\b}(x_\a)$. But, if $\a\in(0,1)$, then
\beqn\label{4.47}
\hskip -1truecm
h_{\a,\b}(x_\a)
&\!\!\!=\!\!\!&
\big\{(d_\a-1)^2|x_0|^2
+\b^{-2}\big[(\a+1)|\b|d_\a(2-d_\a)\big]^{2\a/(\a+1)}|x_0|^{4\a/(\a+1)}\big\}^{1/2}\no
\\[2mm]
\hskip -1truecm
&\!\!\!=\!\!\!&
\big\{(d_\a-1)^2|x_0|^{-2/\g_\a}
+\b^{-2}\big[(\a+1)|\b|d_\a(2-d_\a)\big]^{2\a/(\a+1)}\big\}^{1/2}|x_0|^{2\a/(\a+1)}.
\eeqn
where $\g_\a=(\a+1)/(\a-1)$. Comparing (\ref{4.47}) with
$h_{\a,\b}(2x_0)=|x_0|$, an easy computation shows that
\beqn\label{4.48}
&&\hskip -1,5truecm
\min_{\b<0<\a<1}\{h_{\a,\b}(2x_0),h_{\a,\b}(x_\a)\}:=
\left\{\!
\begin{array}{lll}
h_{\a,\b}(2x_0),\q{\rm if}\ x_0\in(-L_{\a,\b},-D_{\a,\b}),
\\[2mm]
h_{\a,\b}(x_\a)=h_{\a,\b}(2x_0),\q{\rm if}\ x_0=-L_{\a,\b},
\\[2mm]
h_{\a,\b}(x_\a),\q{\rm if}\ x_0<-L_{\a,\b}.
\end{array}
\right.
\eeqn
\begin{corollary}\label{cor4.10}
Let $\Theta_{\a,\b}$, $\a>0>\b$, be the function defined in $(\ref{4.29})$
and let $x_{\a,\b;\pm}$, $x_\a$ and $\wtil x_\a$ 
be the points in $(\ref{4.36})$ and $(\ref{4.46})$.
{\it i)} Let $\a=1$. Then
\beqn\label{4.49}
&&\hskip -2truecm
C_6(\b,x_0)\le \Theta_{1,\b}(x)\le C_7(\b,x_0),\q\forall\,x\in[2x_0,0],
\eeqn
where
\beqn\label{4.50}
\hskip -1truecm
0>C_6(\b,x_0)
&\!\!\!\!:=\!\!\!&
\left\{\!
\begin{array}{lll}
[h_{1,\b}(x_0)]^{-1}\theta_1(\wtil x_1),\q{\rm if}\ \b\le-1,
\\[2mm]
[h_{1,\b}(0)]^{-1}\theta_1(\wtil x_1),\q{\rm if}\ \b\in(-1,0),
\end{array}
\right.
\\[2mm]
\label{4.51}
\hskip -1truecm
0<C_7(\b,x_0)
&\!\!\!\!:=\!\!\!&
\left\{\!
\begin{array}{lll}
[h_{1,\b}(x_1)]^{-1}\theta_1(2x_0),
\q{\rm if}\ \b\le-1,
\\[2mm]
\Theta_{1,\b}(2x_0),\q{\rm if}\ \b\in(-1,0),
\end{array}
\right.
\eeqn
{\it ii)} Let $\a>1$. Then
\beqn\label{4.52}
&&\hskip -1truecm
C_8(\a,\b,x_0)\le \Theta_{\a,\b}(x)\le C_9(\a,\b,x_0),\q\forall\,x\in[2x_0,0],
\eeqn
where
\beqn\label{4.53}
\hskip -1truecm
0>C_8(\a,\b,x_0)
&\!\!\!\!:=\!\!\!&
\left\{\!
\begin{array}{lll}
[h_{\a,\b}(x_0)]^{-1}\theta_\a(\wtil x_\a),\q{\rm if}\ x_0\in[-M_{\a,\b},0),
\\[2mm]
[h_{\a,\b}(x_{\a,\b;+})]^{-1}\theta_\a(\wtil x_\a),\q{\rm if}\ x_0<-M_{\a,\b},
\end{array}
\right.
\\[2mm]
\label{4.54}
\hskip -1truecm
0<C_9(\a,\b,x_0)
&\!\!\!\!:=\!\!\!&
\left\{\!
\begin{array}{lll}
[h_{\a,\b}(x_\a)]^{-1}\theta_\a(2x_0),\q{\rm if}\ x_0\in(-E_{\a,\b},0),
\\[2mm]
[h_{\a,\b}(x_{\a,\b;-})]^{-1}\theta_\a(2x_0),\q{\rm if}\ x_0\le-E_{\a,\b}.
\end{array}
\right.
\eeqn
{\it iii)} Let $\a\in(0,1)$. Then
\beqn\label{4.55}
&&\hskip -1truecm
C_{10}(\a,\b,x_0)\le \Theta_{\a,\b}(x)\le C_{11}(\a,\b,x_0),\q\forall\,x\in[2x_0,0],
\eeqn
where
\beqn\label{4.56}
\hskip -0,5truecm
0>C_{10}(\a,\b,x_0)
&\!\!\!\!:=\!\!\!&
\left\{\!
\begin{array}{lll}
[h_{\a,\b}(0)]^{-1}\theta_\a(\wtil x_\a),\q{\rm if}\ x_0\in[-D_{\a,\b},0),
\\[2mm]
[h_{\a,\b}(x_0)]^{-1}\theta_\a(\wtil x_\a),\q{\rm if}\ x_0<-D_{\a,\b},
\end{array}
\right.
\\[2mm]
\label{4.57}
\hskip -0,5truecm
0<C_{11}(\a,\b,x_0)
&\!\!\!\!:=\!\!\!&
\left\{\!\!
\begin{array}{lll}
\Theta_{\a,\b}(2x_0),\q{\rm if}\ x_0\in[-L_{\a,\b},0),
\\[2mm]
[h_{\a,\b}(x_\a)]^{-1}\theta_\a(2x_0),\q{\rm if}\ x_0<-L_{\a,\b}.
\end{array}
\right.
\eeqn
\end{corollary}
\begin{remark}\label{rem4.11}
\emph{For the reader's convenience and to exhibit 
how they depend on $x_0$, we report here the values
of the constant $C_j(\b,x_0)$, $j=6,7$, and $C_j(\a,\b,x_0)$, $j=8,\ldots,11$,
in Corollary \ref{cor4.10}.
First, observe that $\theta_\a(2x_0)=6|x_0|^2$ for every $\a>0$
and recall that $h_{\a,\b}(2x_0)=h_{\a,\b}(0)=|x_0|$ for every $\a>0>\b$.
Now, since $d_1=7/5$, by evaluating 
$\theta_1(\wtil x_1)$ and $h_{1,\b}(x_1)$ and using
$h_{1,\b}(x_0)=|\b|^{-1/2}|x_0|$, we find
\beqn\no
&&\hskip -1truecm
C_6(\b,x_0)=\left\{\!
\begin{array}{lll}
-\ds\frac{49|\b|^{1/2}|x_0|}{20},\q {\rm if}\ \b\le-1,
\\[3mm]
-\ds\frac{49|x_0|}{20},\q {\rm if}\ \b\in(-1,0),
\end{array}
\right.
\q
C_7(\b,x_0)=\left\{\!
\begin{array}{lll}
\ds\frac{30|\b|^{1/2}|x_0|}{[21+4|\b|]^{1/2}},\q {\rm if}\ \b\le-1,
\\[3mm]
6|x_0|,\q {\rm if}\ \b\in(-1,0).\no
\end{array}
\right.
\eeqn
Instead, when $\a>1$ and $C_4(\a,\b,x_0)$ is defined by (\ref{4.37}), we have
\beqn\no
&&\hskip -1truecm
C_8(\a,\b,x_0)=\left\{\!
\begin{array}{lll}
-c_8(\a,\b)|x_0|^{2/(\a+1)},\q{\rm if}\ x_0\in[-M_{\a,\b},0),
\\[2mm]
\ds-\frac{(4+\a)d_\a^2|x_0|^2}{4C_4(\a,\b,x_0)},\q{\rm if}\ x_0<-M_{\a,\b},
\end{array}
\right.
\\[2mm]
&&\hskip -1truecm
C_9(\a,\b,x_0)=
\left\{\!\!
\begin{array}{lll}
c_9(\a,\b,x_0)|x_0|,\q {\rm if}\ x_0\in(-E_{\a,\b},0),
\\[1mm]
\ds\frac{6|x_0|^2}{C_4(\a,\b,x_0)},\q {\rm if}\ x_0\le-E_{\a,\b}. 
\end{array}
\right.\no
\eeqn
Here
\beqn\no
\hskip -0,5truecm
c_8(\a,\b)
&\!\!\!\!:=\!\!\!&
(4+\a)|\b|\big[(\a+1)|\b|/2\big]^{-\a/(\a+1)}(d_\a/2)^2,
\\[2mm]
c_9(\a,\b,x_0)
&\!\!\!\!:=\!\!\!&
6\big\{(d_\a-1)^2+\b^{-2}\big[(\a+1)|\b|d_\a(2-d_\a)/2\big]^{2\a/(\a+1)}
|x_0|^{2/\g_\a}\big\}^{-1/2},\no
\eeqn
where $\g_\a=(\a+1)/(\a-1)$. 
Finally, if $\a\in(0,1)$, then using (\ref{4.47}) we obtain
\beqn\no
&&\hskip -1truecm
C_{10}(\a,\b,x_0)=
\left\{\!\!
\begin{array}{lll}
-(4+\a)(d_\a/2)^2|x_0|,\q {\rm if}\ x_0\in[-D_{\a,\b},0),
\\[1mm]
-c_{10}(\a,\b)|x_0|^{2/(\a+1)},\q {\rm if}\ x_0<-D_{\a,\b},
\end{array}
\right.\no
\\[2mm]
&&\hskip -1truecm
C_{11}(\a,\b,x_0)=
\left\{\!\!
\begin{array}{lll}
6|x_0|,\q {\rm if}\ x_0\in[-L_{\a,\b},0),
\\[1mm]
c_{11}(\a,\b,x_0)|x_0|^{2/(\a+1)},\q {\rm if}\ x_0<-L_{\a,\b}, 
\end{array}
\right.\no
\eeqn
where $c_{10}(\a,\b)=c_8(\a,\b)$ and
\beqn\no
&&\hskip -1truecm
c_{11}(\a,\b,x_0):=
6\big\{(d_\a-1)^2|x_0|^{-2/\g_\a}
+\b^{-2}\big[(\a+1)|\b|d_\a(2-d_\a)\big]^{2\a/(\a+1)}\big\}^{-1/2}.
\eeqn}
\end{remark}
\begin{remark}\label{rem4.12}
\emph{Notice that, when $x_0<-M_{\a,\b}$, one has $x_{\a,\b;+}=\wtil x_\a$
for $x_0=-\wtil E_{\a,\b}$, where 
$\wtil E_{\a,\b}:=\{d_\a[1-(d_\a/4)]\}^{-1/2}M_{\a,\b}>M_{\a,\b}$.
In particular, if $d_\a\in(4/3,2)$, then $\wtil E_{\a,\b}\in(M_{\a,\b},E_{\a,\b})$.
Therefore, since $d_\a\in(4/3,2)$ for every $\a>1/2$,
when $\a>1$ and $x_0=-\wtil E_{\a,\b}$ the estimate from
below in (\ref{4.52}) becomes $\Theta_{\a,\b}(x_{\a,\b;+})\le \Theta_{\a,\b}(x)$,
which is sharp. On the contrary, when $x_0<-E_{\a,\b}$ is large enough, 
$x_{\a,\b;+}$ and $x_{\a,\b;-}$ approach $0$ and $2x_0$, respectively. 
This implies that, for $x_0<-E_{\a,\b}$ large enough, 
while the lower bound in (\ref{4.52}) approaches 
$[h_{\a,\b}(0)]^{-1}\theta_\a(\wtil x_\a)\le\Theta_{\a,\b}(x)$ 
and becomes less precise, the upper bound approaches 
$\Theta_{\a,\b}(x)\le\Theta_{\a,\b}(2x_0)$ and becomes more accurate.}
\end{remark}
We now specify some other points which appear in Lemma \ref{lem4.13} below, 
where, under the assumption $\a\ge 1/2$,
we find upper and lower bounds for the function $\Psi_{\a,\b}$ defined in (\ref{4.31}). For every $\a\ge1/2$ and $\b<0$ we set
\beqn
\label{4.58}
&&\hskip -1truecm
\ov x_\a:=\frac{(4\a+1)x_0}{3\a},\qq
\ov x_{\a;\pm}:=x_0\pm\Big(\frac{\a+1}{\a+4}\Big)^{1/2}|x_0|,
\\[2mm]
\label{4.59}
&&\hskip -1truecm
x_{\b;\pm}:=\frac{3x_0}{2}\pm\frac{(8+|\b|^2x_0^2)^{1/2}}{2|\b|},\qq
\wtil x_{\b;\pm}:=x_0\pm\frac{(6+3|\b|^2x_0^2)^{1/2}}{3|\b|},
\\[2mm]
\label{4.60}
&&\hskip -1truecm
x_\b=\frac{2|\b|x_0}{3+2|\b|},\qq
\widehat x_{\b;\pm}:=\frac{(9+8|\b|)x_0\pm[8(3+2|\b|)^2+9]^{1/2}|x_0|}{4(3+2|\b|)}.
\eeqn
Observe that $x_{\b;-}<2x_0$, $x_{\b;+}>x_0$ and 
$\wtil x_{\b;-}<x_0<\wtil x_{\b;+}$ for every $\b<0$,
but the points $x_{\b;+}$, $\wtil x_{\b;+}$ and $\wtil x_{\b;-}$
belong to $(x_0,0)$ and $(2x_0,x_0)$, respectively, only if $x_0<-|\b|^{-1}$.
Otherwise, if $x_0\in[-|\b|^{-1},0)$, then $x_{\b;+}\ge 0$, $\wtil x_{\b;+}\ge 0$
and $\wtil x_{\b;-}\le 2x_0$. 
Instead, for every $\a\ge1/2$ and $\b<0$, it is easy to see that the points 
$\ov x_\a$, $\ov x_{\a;\pm}$, $x_\b$ and $\widehat x_{\b;\pm}$
satisfy $\ov x_\a\in[2x_0,x_0)$, $\ov x_{\a;-},\widehat x_{\b;-}\in(2x_0,x_0)$
and  $\ov x_{\a;+}, x_\b, \widehat x_{\b;+}\in(x_0,0)$, with $x_\b<\widehat x_{\b;+}$.

Finally, for every $\b<0$ we define the positive number $R_\b$ by
\beqn\label{4.61}
&&\hskip -1truecm
R_\b:=\frac{(3+2|\b|)|\b|^{1/2}}{4[3(3+|\b|)]^{1/2}}.
\eeqn
Since the number $R_\b$ appears in the case $\a=2$ of 
Lemma \ref{lem4.13}, we observe here that $R_\b$ 
is greater than $M_{2,\b}=|\b|/(2\sqrt 3)$ for every $\b<0$. 
Also, using $c_{2,\b}=|\b|^2/12$, an easy computation shows 
that the point $x_{2,\b;+}$ defined by (\ref{4.36}) 
is related to the point $x_\b$ in (\ref{4.60})
by $x_{2,\b;+}<x_\b$ or $x_\b\le x_{2,\b;+}$ according to the case
$x_0\in(-R_\b,-M_{2,\b})$ or $x_0\le -R_\b$.
\begin{lemma}\label{lem4.13}
Let $\Psi_{\a,\b}$, $\a\ge1/2$, $\b<0$, be the function defined in $(\ref{4.31})$
and let $x_{\a,\b;\pm}$, $\ov x_\a$, $\ov x_{\a;\pm}$, $x_{\b;+}$, 
$\wtil x_{\b;\pm}$, $x_\b$ and $\widehat x_{\b;\pm}$ be the points in 
$(\ref{4.36})$ and $(\ref{4.58})-(\ref{4.60})$.
i) Let $\a=1/2$. {\it i-a)} If $x_0\in[-|\b|^{-1},0)$, then
\beqn\label{4.62}
&&\hskip -1truecm
C_{12}(\b,x_0)\le \Psi_{1/2,\b}(x)\le C_{13}(\b,x_0):=0,\q\forall\,x\in[2x_0,0].
\eeqn
Here, $D_{1/2,\b}$ being equal to $(3/4)|\b|^{-2}$, we have set
\beqn\label{4.63}
&&\hskip -1truecm
0>C_{12}(\b,x_0):=\left\{
\begin{array}{lll}
\Psi_{1/2,\b}(2x_0),\q{\rm if}\ \b\in[-3/4,0),
\\[2mm]
[h_{1/2,\b}(x_0)]^{-1}\psi_{1/2,\b}(2x_0),\q {\rm if}\ \b<-3/4,
\ x_0\in[-|\b|^{-1},-D_{1/2,\b}),
\\[2mm]
\Psi_{1/2,\b}(2x_0),\q{\rm if}\ \b<-3/4,\ x_0\in[-D_{1/2,\b},0).
\end{array}
\right.
\eeqn
{\it i-b)} Let $x_0<-|\b|^{-1}$. Then
\beqn\label{4.64}
&&\hskip -1truecm
C_{14}(\b,x_0)\le \Psi_{1/2,\b}(x)\le C_{15}(\b,x_0),\q\forall\,x\in[2x_0,0],
\eeqn
where, $c_{15}(\b,x_0)$ standing for $\min\{h_{1/2,\b}(x_{\b;+}),h_{1/2,\b}(0)\}$,
we have set
\beqn\label{4.65}
\hskip -0,5truecm
0>C_{14}(\b,x_0)
&\!\!\!\!:=\!\!\!&
\left\{\!
\begin{array}{lll}
[h_{1/2,\b}(2x_0)]^{-1}\psi_{1/2,\b}(\wtil x_{\b;-}),
\q{\rm if}\ \b\in(-3/4,0),\ x_0\in[-D_{1/2,\b},-|\b|^{-1}),
\\[2mm]
[h_{1/2,\b}(x_0)]^{-1}\psi_{1/2,\b}(\wtil x_{\b;-}),
\q{\rm if}\ \b\in(-3/4,0),\ x_0<-D_{1/2,\b},
\\[2mm]
[h_{1/2,\b}(x_0)]^{-1}\psi_{1/2,\b}(\wtil x_{\b;-}),
\q{\rm if}\ \b\le-3/4,
\end{array}
\right.
\eeqn
\beqn\label{4.66}
\hskip -0,5truecm
0<C_{15}(\b,x_0)
&\!\!\!\!:=\!\!\!&
\left\{\!
\begin{array}{lll}
[h_{1,\b}(0)]^{-1}\psi_{1/2,\b}(\wtil x_{\b;+}),
\q{\rm if}\ \b\in(-3/4,0),\ x_0\in[-D_{1/2,\b},-|\b|^{-1}),
\\[2mm]
[c_{15}(\b,x_0)]^{-1}\psi_{1/2,\b}(\wtil x_{\b;+}),
\q{\rm if}\ \b\in(-3/4,0),\ x_0<-D_{1/2,\b},
\\[2mm]
[c_{15}(\b,x_0)]^{-1}\psi_{1/2,\b}(\wtil x_{\b;+}),
\q{\rm if}\ \b\le-3/4.
\end{array}
\right.
\eeqn
{\it ii)} Let $\a\in(1/2,1)$. Then
\beqn\label{4.67}
&&\hskip -1truecm
C_{16}(\a,\b,x_0)\le \Psi_{\a,\b}(x)\le C_{17}(\a,\b,x_0)\q\forall\,x\in[2x_0,0],
\eeqn
where
\beqn\label{4.68}
\hskip -1truecm
0>C_{16}(\a,\b,x_0)
&\!\!\!\!:=\!\!\!&
\left\{\!
\begin{array}{lll}
[h_{\a,\b}(2x_0)]^{-1}[\varphi_{\a,\b}(\ov x_\a)+\phi_{\a,\b}(\ov x_{\a;-})],\q
{\rm if}\ x_0\in[-D_{\a,\b},0),
\\[2mm]
[h_{\a,\b}(x_0)]^{-1}[\varphi_{\a,\b}(\ov x_\a)+\phi_{\a,\b}(\ov x_{\a;-})],\q
{\rm if}\ x_0<-D_{\a,\b},
\end{array}
\right.
\\[2mm]
\label{4.69}
0<C_{17}(\a,\b,x_0)
&\!\!\!\!:=\!\!\!&
\left\{\!
\begin{array}{lll}
[h_{\a,\b}(0)]^{-1}\phi_{\a,\b}(\ov x_{\a;+}),\q {\rm if}\ x_0\in[-D_{\a,\b},0),
\\[2mm]
[h_{\a,\b}(x_0)]^{-1}\phi_{\a,\b}(\ov x_{\a;+}),\q {\rm if}\ x_0<-D_{\a,\b}.
\end{array}
\right.
\eeqn
{\it iii)} Let $\a=1$. Then
\beqn\label{4.70}
&&\hskip -1truecm
C_{18}(\b,x_0)\le \Psi_{1,\b}(x)\le C_{19}(\b,x_0)\q\forall\,x\in[2x_0,0],
\eeqn
where
\beqn\label{4.71}
\hskip -1truecm
0>C_{18}(\b,x_0)
&\!\!\!\!:=\!\!\!&
\left\{\!
\begin{array}{lll}
[h_{1,\b}(2x_0)]^{-1}[\varphi_{1,\b}(\ov x_1)+\phi_{1,\b}(\ov x_{1;-})],
\q{\rm if}\ \b\in(-1,0),
\\[2mm]
[h_{1,\b}(x_0)]^{-1}[\varphi_{1,\b}(\ov x_1)+\phi_{1,\b}(\ov x_{1;-})],
\q{\rm if}\ \b\le-1,
\end{array}
\right.
\\[2mm]
\label{4.72}
0<C_{19}(\b,x_0)
&\!\!\!\!:=\!\!\!&
\left\{\!
\begin{array}{lll}
[h_{1,\b}(0)]^{-1}\phi_{1,\b}(\ov x_{1;+}),\q {\rm if}\ \b\in(-1,0),
\\[2mm]
[h_{1,\b}(x_0)]^{-1}\phi_{1,\b}(\ov x_{1;+}),\q {\rm if}\ \b\le-1.
\end{array}
\right.
\eeqn
{\it iv)} Let $\a\in(1,2)\cup(2,+\infty)$. Then
\beqn\label{4.73}
&&\hskip -0,5truecm
C_{20}(\a,\b,x_0)\le \Psi_{\a,\b}(x)\le C_{21}(\a,\b,x_0)\q\forall\,x\in[2x_0,0],
\eeqn
where
\beqn\label{4.74}
\hskip -1truecm
0>C_{20}(\a,\b,x_0)
&\!\!\!\!:=\!\!\!&
\left\{\!
\begin{array}{lll}
[h_{\a,\b}(x_0)]^{-1}[\varphi_{\a,\b}(\ov x_\a)+\phi_{\a,\b}(\ov x_{\a;-})],\q
{\rm if}\ x_0\in[-M_{\a,\b},0),
\\[2mm]
[h_{\a,\b}(x_{\a,\b;-})]^{-1}[\varphi_{\a,\b}(\ov x_\a)+\phi_{\a,\b}(\ov x_{\a;-})],\q
{\rm if}\ x_0<-M_{\a,\b},
\end{array}
\right.
\\[2mm]
\label{4.75}
0<C_{21}(\a,\b,x_0)
&\!\!\!\!:=\!\!\!&
\left\{\!
\begin{array}{lll}
[h_{\a,\b}(x_0)]^{-1}\phi_{\a,\b}(\ov x_{\a;+}),\q {\rm if}\ x_0\in[-M_{\a,\b},0),
\\[2mm]
[h_{\a,\b}(x_{\a,\b;+})]^{-1}\phi_{\a,\b}(\ov x_{\a;+}),\q {\rm if}\ x_0<-M_{\a,\b}.
\end{array}
\right.
\eeqn
{\it v)} Let $\a=2$. Then
\beqn\label{4.76}
&&\hskip -1truecm
C_{22}(\b,x_0)\le \Psi_{2,\b}(x)\le C_{23}(\b,x_0)\q\forall\,x\in[2x_0,0],
\eeqn
where
\beqn\label{4.77}
\hskip -1truecm
0>C_{22}(\b,x_0)
&\!\!\!\!:=\!\!\!&
\left\{\!
\begin{array}{lll}
[h_{2,\b}(x_0)]^{-1}\psi_{2,\b}(\widehat x_{\b;-}),\q{\rm if}\ x_0\in[-M_{2,\b},0),
\\[2mm]
[h_{2,\b}(x_{2,\b;-})]^{-1}\psi_{2,\b}(\widehat x_{\b;-}),\q{\rm if}\ x_0<-M_{2,\b},
\end{array}
\right.
\\[2mm]
\hskip -1truecm
\label{4.78}
0<C_{23}(\b,x_0)
&\!\!\!\!:=\!\!\!&
\left\{\!
\begin{array}{lll}
[h_{2,\b}(x_\b)]^{-1}\psi_{2,\b}(\widehat x_{\b;+}),\q{\rm if}\ x_0\in(-R_\b,0),
\\[2mm]
[h_{2,\b}(x_{2,\b;+})]^{-1}\psi_{2,\b}(\widehat x_{\b;+}),\q{\rm if}\ x_0\le-R_\b.
\end{array}
\right.
\eeqn
\end{lemma}
\begin{remark}\label{rem4.14}
\emph{From definitions (\ref{4.36}) and 
(\ref{4.58}) we find that the points
$x_{\a,\b;+}$ and $\ov x_{\a;+}$ coincide for 
$x_0=-c(\a)M_{\a,\b}<-M_{\a,\b}$, where $c(\a)=[(\a+4)/3]^{1/2}$, 
Therefore, if $\a\in(1,2)\cup(2,+\infty)$ and $x_0=-c(\a)M_{\a,\b}$,
then the upper bound in (\ref{4.73}) becomes 
$\Psi_{\a,\b}(x)\le\Psi_{\a,\b}(\ov x_{\a;+})$, $x\in[2x_0,0]$,
which is sharp. Instead, concerning the lower bound in (\ref{4.73}), 
we first observe that the points $\ov x_\a$ and $\ov x_{\a;-}$ 
coincide only for $\a=\wtil \a$, where $\wtil\a=(5+3\sqrt 17)/16$. 
An easy computation then shows that the equality 
$x_{\wtil\a,\b;-}=\ov x_{\wtil\a;-}$, or, equivalently, 
$x_{\wtil\a,\b;-}=\ov x_{\wtil\a}$, is satisfied for 
$x_0=-c({\wtil \a})M_{\wtil\a,\b}<-M_{\wtil\a,\b}$,
where $c=c(\a)$ is as above.
As a consequence, if $\a=\wtil\a\in(1,2)$ and $x_0=-c(\wtil\a)M_{\wtil\a,\b}$, 
then (\ref{4.73}) reduces to
$\Psi_{\wtil\a,\b}(\ov x_{\wtil\a;-})\le\Psi_{\wtil\a,\b}(x)
\le\Psi_{\wtil\a,\b}(\ov x_{\wtil\a;+})$, $x\in[2x_0,0]$,
which is optimal. On the contrary, if $\a\neq\wtil\a$ and/or 
$x_0\neq -c(\a)M_{\a,\b}$, estimate (\ref{4.73}) becomes less precise. 
For instance, when $|x_0|>M_{\a,\b}$ is large enough the points
$x_{\a,\b;-}$ and $x_{\a,\b;+}$ approach $2x_0$ and $0$, respectively,
and (\ref{4.73}) approaches 
$[h_{\a,\b}(2x_0)]^{-1}[\varphi_{\a,\b}(\ov x_\a)+\phi_{\a,\b}(\ov x_{\a;-})]
\le\Psi_{\a,\b}(x)\le[h_{\a,\b}(0)]^{-1}\psi_{\a,\b}(\ov x_{\a;+})$.
Despite of this fact, we want to stress that to find the greatest 
and least values of $\Psi_{2,\b}$ using the standard tools of calculus 
seems a very hard task, due to the difficulties in locating its stationary points
(see Remark \ref{rem5.3}).}
\end{remark}
\begin{remark}\label{rem4.15}
\emph{Similar observations to those in Remark \ref{rem4.14}
hold for $\a=2$ and estimate (\ref{4.76}). In fact, from definitions
(\ref{4.36}) and (\ref{4.60}) we find that the equalities
$x_{2,\b;-}=\widehat x_{\b;-}$ and $x_{2,\b;+}=\widehat x_{\b;+}$ 
are satisfied for $x_0=x_{0;-}:=-c_+(\b)M_{2,\b}$ and 
$x_0=x_{0;+}:=-c_-(\b)M_{2,\b}$, respectively, where
\beqn\no
&&\hskip -1truecm
c_\pm(\b):=\frac{4(3+2|\b|)}
{\big\{\big([8(3+2|\b|)^2+9]^{1/2}\pm 3\big)^2-36\big\}^{1/2}},\q \b<0.
\eeqn
We get $1<c_+(\b)<c_-(\b)$, 
so that $x_{0;+}<x_{0;-}<-M_{2,\b}$ for every $\b<0$.
In addition, using $M_{2,\b}=|\b|/(2\sqrt 3)$ and definition (\ref{4.61}),
easy computations yield $x_{0;+}<-R_\b$ for every $\b<0$, 
whereas $x_{0;-}<-R_\b$ only for $\b<\wtil\b$, where $\wtil\b=3(1-\sqrt 3)/2$.
From (\ref{4.76})--(\ref{4.78}) we thus deduce that, if $x_0=x_{0;-}$
(respectively, $x_0=x_{0;+}$), then the estimate from below
(respectively, above) in (\ref{4.76}) becomes 
$\Psi_{2,\b}(\widehat x_{\b;-})\le\Psi_{2,\b}(x)$
(respectively, $\Psi_{2,\b}(x)\le \Psi_{2,\b}(\widehat x_{\b;+}))$,
which is sharp.
Observe also that in the case $(\a,\b)=(2,-3/2)$, i.e. the case of
the normal Tricomi curve, we have $R_{\b}=1/2$ and, according
to $\b<\wtil\b$, both the point $x_{0;\pm}$ are less than $-1/2$.
Indeed, letting $\b=-3/2$ in $c_\pm(\b)$ and using $M_{2,-3/2}=\sqrt 3/4$, 
we obtain $x_{0;\pm}=-[(15\pm\sqrt{33})/32]^{1/2}$.} 
\end{remark}
\begin{corollary}\label{cor4.16}
Let $\a\ge1/2$ and $\b<0$ and let $C_{2m}(\b,x_0)$, $m=6,7,9,11$, 
and $C_{2n}(\a,\b,x_0)$, $n=8,10$, be the negative constants defined in 
$(\ref{4.63})$, $(\ref{4.65})$, $(\ref{4.68})$, $(\ref{4.71})$, $(\ref{4.74})$ 
and $(\ref{4.77})$. Then
\beqn\label{4.79}
&&\hskip -1truecm
|\Psi_{\a,\b}(x)|\le C_{24}(\a,\b,x_0),\q\forall\,x\in[2x_0,0],
\eeqn
where
\beqn\label{4.80}
&&\hskip -1truecm
C_{24}(\a,\b,x_0):=\left\{\!
\begin{array}{lll}
|C_{12}(\b,x_0)|,\q{\rm if}\ \a=1/2,\ x_0\in[-|\b|^{-1},0),
\\[2mm]
|C_{14}(\b,x_0)|,\q{\rm if}\ \a=1/2,\ x_0<-|\b|^{-1},
\\[2mm]
|C_{16}(\a,\b,x_0)|,\q{\rm if}\ \a\in(1/2,1),
\\[2mm]
|C_{18}(\b,x_0)|,\q{\rm if}\ \a=1,
\\[2mm]
|C_{20}(\a,\b,x_0)|,\q{\rm if}\ \a\in(1,2)\cup(2,+\infty),
\\[2mm]
|C_{22}(\b,x_0)|,\q{\rm if}\ \a=2.
\end{array}
\right.
\eeqn
\end{corollary}
\begin{remark}\label{rem.4.17}
\emph{As we have done in Remark \ref{rem4.11} for the constants
$C_j(\b,x_0)$, $j=6,7$, and $C_j(\a,\b,x_0)$, $j=8,\ldots,11$,
here we report the values of the constants $|C_{2m}(\b,x_0)|$, $m=6,7,9,11$,
and $|C_{2n}(\a,\b,x_0)|$, $n=8,10$. Using the expression of $\psi_{1/2,\b}$
(see the proof of Lemma \ref{lem4.13} and Corollary \ref{cor4.16} 
in Section \ref{Sec5}) we have
\beqn\no
\hskip -1truecm
|C_{12}(\b,x_0)|
&\!\!\!\!:=\!\!\!&
\left\{
\begin{array}{lll}
12|\b|^{-1},\q{\rm if}\ \b\in[-3/4,0),
\\[2mm]
12(3|\b|/4)^{-1/3}|x_0|^{1/3},\q {\rm if}\ \b<-3/4,\ x_0\in[-|\b|^{-1},-D_{1/2,\b}),
\\[2mm]
12|\b|^{-1},\q{\rm if}\ \b<-3/4,\ x_0\in[-D_{1/2,\b},0),
\end{array}
\right.
\\[2mm]
\hskip -1truecm
|C_{14}(\b,x_0)|
&\!\!\!\!:=\!\!\!&
\left\{\!
\begin{array}{lll}
c_{14}(\b,x_0)|x_0|^{-1},
\q{\rm if}\ \b\in(-3/4,0),\ x_0\in[-D_{1/2,\b},-|\b|^{-1}),
\\[2mm]
c_{14}(\b,x_0)(3|\b|/4)^{-1/3}|x_0|^{-2/3},\q{\rm if}\ \b\in(-3/4,0),\ x_0<-D_{1/2,\b},
\\[2mm]
c_{14}(\b,x_0)(3|\b|/4)^{-1/3}|x_0|^{-2/3},\q{\rm if}\ \b\le-3/4,
\end{array}
\right.\no
\eeqn
where $c_{14}(\b,x_0)=-\psi_{1/2,\b}(\wtil x_{\b;-})$ is given by the first 
expression in (\ref{5.24}). Instead, by evaluating $\varphi_{\a,\b}(\ov x_\a)$
and $\phi_{\a,\b}(\ov x_{\a;-})$, $\a>1/2$, $\b<0$, we get
\beqn\no
\hskip -1truecm
|C_{16}(\a,\b,x_0)|
&\!\!\!\!:=\!\!\!&
\left\{\!
\begin{array}{lll}
c_1(\a,\b)|x_0|^{(2\a-1)/(\a+1)}+c_2(\a,\b)|x_0|^{3/(\a+1)},
\q{\rm if}\ x_0\in[-D_{\a,\b},0),
\\[2mm]
c_3(\a,\b)|x_0|^{\a/(\a+1)}+c_4(\a,\b)|x_0|^{(4-\a)/(\a+1)},
\q{\rm if}\ x_0<-D_{\a,\b},
\end{array}
\right.
\\[2mm]
|C_{18}(\b,x_0)|
&\!\!\!\!:=\!\!\!&
\left\{\!
\begin{array}{lll}
c_1(1,\b)|x_0|^{1/2}+c_2(1,\b)|x_0|^{3/2}
,\q {\rm if}\ \b\in(-1,0),
\\[2mm]
c_3(1,\b)|x_0|^{1/2}+c_4(\a,\b)|x_0|^{3/2}
,\q {\rm if}\ \b\le-1,
\end{array}
\right.\no
\\[2mm]
|C_{20}(\a,\b,x_0)|
&\!\!\!\!:=\!\!\!&
\left\{\!
\begin{array}{lll}
c_3(\a,\b)|x_0|^{\a/(\a+1)}+c_4(\a,\b)|x_0|^{(4-\a)/(\a+1)},
\q{\rm if}\ x_0\in[-M_{\a,\b},0),
\\[2mm]
c_5(\a,\b,x_0)|x_0|^{3\a/(\a+1)}+c_6(\a,\b,x_0)|x_0|^{(\a+4)/(\a+1)},
\q{\rm if}\ x_0<-M_{\a,\b}.
\end{array}
\right.\no
\eeqn
Here, for every $\a>1/2$ and $\b<0$ we have set
\beqn\no
&&\hskip -1truecm
c_1(\a,\b):=
2\a^{-1}(4\a+1)^{(4\a+1)/(2\a+2)}
\big[(2\a^2+\a-1)/(18\a^2)\big]^{(2\a-1)/(2\a+2)}|\b|^{-3/(2\a+2)},\no
\\[1mm]
&&\hskip -1truecm
c_2(\a,\b):=
4\big[(\a+1)/(\a+4)\big]^{(\a+4)/(2\a+2)}(3|\b|/2)^{3/(2\a+2)},\no
\\[1mm]
&&\hskip -1truecm
c_j(\a,\b):=
|\b|\big[(\a+1)|\b|/2\big]^{-\a/(\a+1)}c_{j-2}(\a,\b),\q j=3,4,\no
\\[2mm]
&&\hskip -1truecm
c_j(\a,\b,x_0):=
[C_4(\a,\b,x_0)]^{-1}c_{j-4}(\a,\b),\q j=5,6,\no
\eeqn
where $C_4(\a,\b,x_0)$ is defined by (\ref{4.37}). Finally, if $\a=2$, then
using the first expression in the following (\ref{5.27}) and evaluating
$[g_{2,\b}(\widehat x_{\b;-})]^{3/2}$ we obtain
\beqn\no
\hskip -1truecm
|C_{22}(\b,x_0)|
&\!\!\!\!:=\!\!\!&
\left\{\!
\begin{array}{lll}
(4|\b|/9)^{1/3}c_7(\b)|x_0|^{2/3},
\q{\rm if}\ x_0\in[-M_{2,\b},0),
\\[2mm]
[C_4(2,\b,x_0)]^{-1}c_7(\b)|x_0|^2
\q{\rm if}\ x_0<-M_{2,\b},
\end{array}
\right.
\eeqn
where for every $\b<0$ we have set
\beqn\no
\hskip 0truecm
c_7(\b)
&\!\!\!\!:=\!\!\!&
(3|\b|^{-1}/8)^{1/2}[4(3+2|\b|)]^{-1}
\big\{[8(3+2|\b|)^2+9]^{1/2}+9\big\}\no
\\[1mm]
\hskip 0truecm
&&
\hskip 1truecm
\times
\big\{32|\b|^2+96|\b|+54+6[8(3+2|\b|)^2+9]^{1/2}\big\}^{1/2}.\no
\eeqn
For instance, if we let $\b=-3/2$ in the previous formula
of $C_{22}(\b,x_0)$, then we obtain the constant $C(x_0):=|C_{22}(-3/2,x_0)|$ 
for the case when $\s_{\a,\b}$ is the normal Tricomi curve.
An easy computation shows that this constant is precisely
\beqn\no
C(x_0)
&\!\!\!\!:=\!\!\!&
\left\{\!
\begin{array}{lll}
2^{-19/6}3^{2/3}(3+\sqrt{33})(15+\sqrt{33})^{1/2}|x_0|^{2/3},
\q{\rm if}\ x_0\in[-\sqrt 3/4,0),
\\[2mm]
2^{-7/2}3[x_0^2-(3/64)]^{-1/2}(3+\sqrt{33})(15+\sqrt{33})^{1/2}|x_0|^2,
\q{\rm if}\ x_0<-\sqrt 3/4.
\end{array}
\right.
\eeqn}
\end{remark}
We can now finally proceed to estimate the line integral $\int_\s\om_1\d s$
in the case when $\s=\s_{\a,\b}$ for some $\a\ge 1/2$ and $\b<0$. 
Due to Corollaries \ref{cor4.10} and \ref{cor4.16}
our estimate will deeply depend on the values of
the parameter pair $(\a,\b)$ and the length of the parabolic diameter 
$|AB|=2|x_0|$ of $\Om_{\a,\b}$.
\begin{lemma}\label{lem4.18}
Let $\Om=\Om_{\a,\b}$ for some $\a\ge1/2$ and $\b<0$ and
let $u$ be a real valued solution of problem $(\ref{3.1})$ which is 
Fr\' echet differentiable at each of the points of $\s=\s_{\a,\b}$
and such that $|y|^{1/2}u_x$, $u_y\in L^2(\s)$.
Let $\om_1$ be defined by $(\ref{3.9})$.
Then, for every $\ve>0$, the following estimate holds:
\beqn\label{4.81}
&&\hskip -1,5truecm
0\le\int_\s\om_1\d s
\le C_{25}(\a,\b,x_0,\ve)\||y|^{1/2}u_x\|_{L^2(\s)}^2
+C_{26}(\a,\b,x_0,\ve)\|u_y\|_{L^2(\s)}^2.
\eeqn
Here, for $j=25,26$ we have set
\beqn\label{4.82}
&&\hskip -1truecm
C_j(\a,\b,x_0,\ve):=
\left\{\!\!
\begin{array}{lll}
(-1)^{j+1}C_{k_{1,j}}(1/2,\b,x_0)+(\ve/2)|C_{12}(\b,x_0)|,
\q{\rm if}\ \a=1/2,\ x_0\in[-|\b|^{-1},0),
\\[2mm]
(-1)^{j+1}C_{k_{1,j}}(1/2,\b,x_0)+(\ve/2)|C_{14}(\b,x_0)|,
\q{\rm if}\ \a=1/2,\ x_0<-|\b|^{-1},
\\[2mm]
(-1)^{j+1}C_{k_{1,j}}(\a,\b,x_0)+(\ve/2)|C_{16}(\a,\b,x_0)|,
\q{\rm if}\ \a\in(1/2,1),
\\[2mm]
(-1)^{j+1}C_{k_{2,j}}(\b,x_0)+(\ve/2)|C_{18}(\b,x_0)|,
\q{\rm if}\ \a=1,
\\[2mm]
(-1)^{j+1}C_{k_{3,j}}(\a,\b,x_0)+(\ve/2)|C_{20}(\a,\b,x_0)|,
\q{\rm if}\ \a\in(1,2)\cup(2,+\infty),
\\[2mm]
(-1)^{j+1}C_{k_{3,j}}(2,\b,x_0)+(\ve/2)|C_{22}(\b,x_0)|,
\q{\rm if}\ \a=2,
\end{array}
\right.
\eeqn
where
\beqn\label{4.83}
&&\hskip -1truecm
k_{1,j}=\left\{\!
\begin{array}{lll}
11,\q{\rm if}\ j=25,
\\[1mm]
10,\q{\rm if}\ j=26,
\end{array}
\right.
\q
k_{2,j}=\left\{\!
\begin{array}{lll}
7,\q{\rm if}\ j=25,
\\[1mm]
6,\q{\rm if}\ j=26,
\end{array}
\right.
\q
k_{3,j}=\left\{\!
\begin{array}{lll}
9,\q{\rm if}\ j=25,
\\[1mm]
8,\q{\rm if}\ j=26,
\end{array}
\right.
\eeqn
and the constants $C_m(\b,x_0)$, $m=6,7,12,14,18,22$, and
$C_j(\a,\b,x_0)$, $j=8,9,10,11,16,20$, are defined by $(\ref{4.50})$,
$(\ref{4.51})$, $(\ref{4.53})$, $(\ref{4.54})$, $(\ref{4.56})$, $(\ref{4.57})$,
$(\ref{4.63})$, $(\ref{4.65})$, $(\ref{4.68})$, $(\ref{4.71})$, $(\ref{4.74})$ 
and $(\ref{4.77})$.
\end{lemma}
\begin{proof}
First, since $u_{|\s_{\a,\b}}=0$, the assumption that $u$ 
is Fr\' echet differentiable at each of the points of $\s_{\a,\b}$ 
implies that the directional derivative of $u$ is zero along $\s_{\a,\b}$. 
Therefore, since the assumption $\a\ge1/2$ implies that
$\partial\Om_{\a,\b}$ is $D$-starlike with respect to 
$D=-3x\partial_x-2y\partial_y$ (see Lemma \ref{lem4.2}), 
we are in position to apply the argument in \cite[p. 416]{LP5}, 
to which we refer the reader for the details, to derive
the lower bound $0\le\int_\s\om_1\d s$.
Now, recall that (see formula (\ref{4.32}))
\beqn\label{4.84}
\widehat{\om_1}=
\Theta_{\a,\b}(x)y\what{u_x}^2+\Psi_{\a,\b}(x)y^{1/2}\what{u_x}\,\what{u_y}
-\Theta_{\a,\b}(x)\what{u_y}^2,
\eeqn
where $\Theta_{\a,\b}$ and $\Psi_{\a,\b}$ are the functions defined in
(\ref{4.29}) and (\ref{4.31}). Then, using the inequality
$2|a||b|\le \ve a^2+\ve^{-1}b^2$, $a,b\in\rsp$, $\ve>0$, 
from (\ref{4.84}) we obtain:
\beqn\label{4.85}
\hskip -1truecm
\what{\om_1}
&\!\!\!\le\!\!\!&
\Theta_{\a,\b}(x)y\what{u_x}^2+|\Psi_{\a,\b}(x)||y^{1/2}\what{u_x}||\what{u_y}|
-\Theta_{\a,\b}(x)\what{u_y}^2\no
\\[2mm]
&\!\!\!\le\!\!\!&
\{\Theta_{\a,\b}(x)+(\ve/2)|\Psi_{\a,\b}(x)|\}y\what{u_x}^2
+\{(2\ve)^{-1}|\Psi_{\a,\b}(x)|-\Theta_{\a,\b}(x)\}\what{u_y}^2.
\eeqn
It thus now suffices to apply Corollaries \ref{cor4.10} and \ref{cor4.16}
to conclude the proof.
Indeed, due to (\ref{4.49})--(\ref{4.57}), (\ref{4.79}) and (\ref{4.80})
inequality (\ref{4.85}) leads us to
\beqn\no
&&\hskip -2truecm
\what{\om_1}\le 
C_{25}(\a,\b,x_0,\ve)y\what{u_x}^2+C_{26}(\a,\b,x_0,\ve)\what{u_y}^2,
\q\forall\,\ve>0,
\eeqn
where the constants $C_j(\a,\b,x_0,\ve)$, $j=25,26$, are defined 
by (\ref{4.82}), (\ref{4.83}). Hence
\beqn\no
&&\hskip -1truecm
\int_\s\om_1\d s
\le C_{25}(\a,\b,x_0,\ve)\||y|^{1/2}u_x\|_{L^2(\s)}^2
+C_{26}(\a,\b,x_0,\ve)\|u_y\|_{L^2(\s)}^2,\q\forall\,\ve>0,
\eeqn
and the proof of (\ref{4.81}) is complete.
\end{proof}
\begin{remark}\label{rem4.19}
\emph{Notice that the assumption that $u$ is Fr\' echet differentiable 
at each of the points of $\s$ is necessary only in order to show the estimate 
from below $0\le\int_\s \om_1\d s$, but does not contribute in any way for
showing the estimate from above of $\int_\s \om_1\d s$. Moreover,
not even the $D$-starlikeness of $\Om_{\a,\b}$, which derives from
the assumption $\a\ge 1/2$, plays any role in the proof of the
upper bound of $\int_\s \om_1\d s$. Indeed, what really counts in the proof
of such an upper bound is only the assumption $\a\ge1/2$
which allows us to apply Corollary $\ref{4.16}$.} 
\end{remark}
We can now prove our main result. 
For brevity, in Theorem \ref{thm4.20} below,
the symbols $L^2(\Om)$, $C^1(\ov\Om)$ and $C^2(\Om)$
are used without exception for both real and complex valued functions.
Needless to say, if we have to deal with complex valued functions, then
$L^2(\Om)$ is understood endowed with the usual complex inner product
$\langle\cdot,\cdot\rangle_{2,\sim}$ defined in Remark \ref{rem3.5},
whereas the spaces $C^1(\ov\Om)$ and $C^2(\Om)$ 
are meant for $C^1(\ov\Om;\csp)$ and $C^2(\Om;\csp)$, respectively.
\begin{theorem}\label{thm4.20}
Let $\Om=\Om_{\a,\b}$ for some $\a\ge1/2$ and $\b<0$
and let $\l_0>0$ be the principal eigenvalue of $\wtil T_{AC\cup\s}$ 
defined in {\rm Theorem \ref{thm2.4}}, where $\s=\s_{\a,\b}$.
Let $u=\Re u+i\Im u$, where $\Re u, \Im u\in \wtil W_{AC\cup\s}^1(\Om)$, 
be a non almost everywhere vanishing solution to problem 
\beqn\label{4.86}
&&\hskip -0,5truecm
\left\{\!\!
\begin{array}{lll}
Tu=\l u\q {\rm in}\ \Om,
\\[2mm]
u=0\q{\rm on}\ AC\cup\s,
\end{array}
\right.
\q\l\in[\l_0,+\infty),
\eeqn
and assume that $u_y$, $xu_x$, $yu_x\in C^1(\ov\Om)$, $xu\in C^2(\Om)$,
$u\in L^2(BC)$ and $|y|^{1/2}u_x$, $u_y\in L^2(BC)\cap L^2(\s)$.
Then, for every $\ve_j>0$, $j=1,2$, the following estimate holds:
\beqn\label{4.87}
\hskip 0truecm
0&\!\!\!<\!\!\!&
2\l^{1/2}\|u\|_{L^2(\Om)}\no
\\[2mm]
\hskip 0truecm
&\!\!\!\le\!\!\!&
\Big\{C_1(x_0,\ve_1)\||y|^{1/2}u_x\|_{L^2(BC)}^2
+C_2(x_0,\ve_1)\|u_y\|_{L^2(BC)}^2\no
\\[2mm]
\hskip -0truecm
&&
\ \, +\,C_3(x_0)\big[\|\Re u\|_{L^2(BC)}+\|\Im u\|_{L^2(BC)}\big]
\big[\||y|^{1/2}u_x\|_{L^2(BC)}+\|u_y\|_{L^2(BC)}\big]\no
\\[2mm]
\hskip -0truecm
&&
\ \, +\,C_{25}(\a,\b,x_0,\ve_2)\||y|^{1/2}u_x\|_{L^2(\s)}^2
+C_{26}(\a,\b,x_0,\ve_2)\|u_y\|_{L^2(\s)}^2\Big\}^{1/2},
\eeqn
where $C_j(x_0,\ve_1)$, $j=1,2$, are defined by $(\ref{4.20})$ with $\ve=\ve_1$,
$C_3(x_0)$ is defined by $(\ref{4.25})$, and $C_j(\a,\b,x_0,\ve_2)$, $j=25,26$,
are defined by $(\ref{4.82})$, $(\ref{4.83})$ with $\ve=\ve_2$.
\end{theorem}
\begin{proof}
First (see (\ref{3.11})), 
since $\l\in[\l_0,+\infty)$ and $u$ does not vanish almost everywhere,
we have that the real valued functions 
$v_1=\Re u$ and $v_2=\Im u$ also solve (\ref{4.86}),
and that at least one of the two is not the zero element of $L^2(\Om)$.
Moreover, our assumptions on $u$ imply that 
$(v_j)_y$, $x(v_j)_x$, $y(v_j)_x\in C^1(\ov\Om)$, 
$xv_j\in C^2(\Om)$, $v_j\in L^2(BC)$ and 
$|y|^{1/2}(v_j)_x$, $(v_j)_y\in L^2(BC)\cap L^2(\s)$, $j=1,2$.
Thus, setting $f(t)=\l t$, $t\in\rsp$, we may as well suppose from 
the outset that $u\in \wtil W^1_{AC\cup\s}(\Om)$,
$u\neq 0$, is a real valued solution to problem (\ref{3.1}) 
satisfying the assumptions of Theorem \ref{thm3.3} 
and Lemmas \ref{lem4.4}, \ref{lem4.5} and \ref{lem4.18}. 
Observe that, according to Remark \ref{rem4.19}, here we do not
require that $u$ is Fr\' echet differentiable at each of the points of $\s$,
since we do not need the lower bound $0\le\int_\s \om_1\ds$. 
In particular, $u$ satisfies the identity (\ref{3.8}) with $F(t)=(\l/2)t^2$, $t\in\rsp$.
Therefore, since $10F(u)-uf(u)=4\l u^2$, we have
\beqn\label{4.88}
&&\hskip -2truecm
0<4\l\|u\|_{L^2(\Om)}^2
=\int_{BC}(\om_1+\om_2)\d s+\int_{\s}\om_1\d s,
\eeqn
where $\om_1$ and $\om_2$ are defined by (\ref{3.9}) and (\ref{3.10}).
Hence, taking $\ve=\ve_1>0$ in Lemma \ref{lem4.4}
and $\ve=\ve_2>0$ in Lemma \ref{lem4.18} and
applying estimate (\ref{4.19}), (\ref{4.24}) and (\ref{4.81}),
from (\ref{4.88}) we deduce
\beqn\label{4.89}
\hskip -0truecm
0<4\l\|u\|_{L^2(\Om)}^2
&\!\!\!\le\!\!\!&
C_1(x_0,\ve_1)\||y|^{1/2}u_x\|_{L^2(BC)}^2
+C_2(x_0,\ve_1)\|u_y\|_{L^2(BC)}^2\no
\\[2mm]
\hskip -0truecm
&&
+\,C_3(x_0)\|u\|_{L^2(BC)}
\big[\||y|^{1/2}u_x\|_{L^2(BC)}+\|u_y\|_{L^2(BC)}\big]\no
\\[2mm]
\hskip -0truecm
&&
+\,C_{25}(\a,\b,x_0,\ve_2)\||y|^{1/2}u_x\|_{L^2(\s)}^2
+C_{26}(\a,\b,x_0,\ve_2)\|u_y\|_{L^2(\s)}^2.
\eeqn
This proves (\ref{4.87}) in the case that $u$ is real valued.
To complete the proof in the general case it suffices to replace
$u$ in (\ref{4.89}) with $\Re u$ and $\Im u$, respectively,
and then summing up the so obtained estimate, taking into account
the identities $|y|^{1/2}({\mathfrak F} u)_x={\mathfrak F}(|y|^{1/2}u_x)$ and
$({\mathfrak F} u)_y={\mathfrak F}(u_y)$, ${\mathfrak F}=\Re,\Im$,
and the inequalities
$\||y|^{1/2}(\mathfrak F u)_x\|_{L^2(BC)}\le \||y|^{1/2}u_x\|_{L^2(BC)}$
and $\|(\mathfrak F u)_y\|_{L^2(BC)}\le \|u_y\|_{L^2(BC)}$, 
${\mathfrak F}=\Re,\Im$.
\end{proof}
\begin{remark}\label{rem4.21}
\emph{Notice that, if $u\in\wtil W_{AC\cup\s}^1$
is an eigenfunction corresponding to an eigenvalue $\l\in[\l_0,+\infty)$
satisfying the assumption of Theorem \ref{thm4.20}, 
then (\ref{4.88}) improves the inequality
$0\le \int_{BC}(\om_1+\om_2)\d s+\int_{\s}\om_1\d s$
which is shown in the proof of \cite[Theorem 4.2]{LP5} (take there $k=0$) 
under the assumption $u\in C^2(\ov\Om)$.}
\end{remark}
Of course, when $\Om=\Om_{\a,\b}$ for some $\a\ge1/2$ and $\b<0$,
one can apply estimate (\ref{4.87}), with the quadruplet $(\l,u,\Re u,\Im u)$
being replaced by $(\l_0,u_0,u_0,0)$ and $(\wtil\l_0,\wtil u_0,\wtil u_0,0)$, 
respectively, to the eigenfunctions $u_0\in\wtil W_{AC\cup\s}(\Om)$ 
and $\wtil u_0\in\wtil W_{AC\cup\s}(\Om)\cap C(\ov\Om)$
of Theorems \ref{thm2.4} and \ref{thm2.5}, provided one can show
that they satisfy the additional regularity requirements of Theorem \ref{thm4.20}.
\section{Proof of Lemmas \ref{lem4.8} and \ref{lem4.13} 
and Corollaries \ref{cor4.10} and \ref{cor4.16}}
\label{Sec5}
\setcounter{equation}{0}
{\it Proof of Lemma \ref{lem4.8}.}
First, from definitions (\ref{4.3}) and (\ref{4.12}) we immediately 
derive $h_{\a,\b}(x)=h_{\a,\b}(2x_0-x)$, so $h_{\a,\b}$ is an even function 
with respect to the line $\{x=x_0\}$ and it suffices to prove the lemma assuming 
$x\in[x_0,0]$. With such a convention, we change the variable from $x$ to 
$X=x-x_0\in[0,-x_0]$ and we consider the function 
\beqn\label{5.1}
&&\hskip -1truecm
H_{\a,\b}(X):=h_{\a,\b}(X+x_0)=\big\{X^2+\b^{-2}[G_{\a,\b}(X)]^{2\a}\big\}^{1/2},
\q X\in[0,-x_0],
\eeqn
where 
\beqn\label{5.2}
&&\hskip -1truecm
G_{\a,\b}(X):=g_{\a,\b}(X+x_0)=
\bigg[\frac{(\a+1)|\b|(x_0^2-X^2)}{2}\bigg]^{1/(\a+1)},
\q X\in[0,-x_0].
\eeqn
Differentiating (\ref{5.1}) with respect to $X$ and using 
$G_{\a,\b}'(X)=\b X[G_{\a,\b}(X)]^{-\a}$, $X\in[0,-x_0)$,  we get
\beqn\label{5.3}
&&\hskip -1truecm
H_{\a,\b}'(X)=
[H_{\a,\b}(X)]^{-1}X\big\{1-\a|\b|^{-1}[G_{\a,\b}(X)]^{\a-1}\big\},\q X\in[0,-x_0).
\eeqn
Hence, the three cases $\a=1$, $\a>1$ and $\a\in(0,1)$ have to be considered.

\noindent
{\it i) Case $\a=1$}. 
In this case $H_{1,\b}'(X)\ge 0$ if and only if $1-|\b|^{-1}\ge 0$, i.e. if  $\b\le -1$. 
In particular, if $\b=-1$, then $H_{1,\b}'(X)=0$ for every $X\in[0,-x_0)$ and
(see the observation below formula (\ref{4.12})) $H_{1,\b}(X)$ is costant
equal to $|x_0|$. Therefore $H_{1,\b}$ is non decreasing
(respectively, decreasing) for $\b\le -1$ (respectively, $\b\in(-1,0)$)
and (\ref{4.40}) follows from
\beqn\no
\left\{\!
\begin{array}{lll}
h_{1,\b}(x_0)=H_{1,\b}(0)\le H_{1,\b}(X)\le H_{1,\b}(-x_0)=h_{1,\b}(0),\q \b\le -1,
\\[2mm]
h_{1,\b}(0)=H_{1,\b}(-x_0)\le H_{1,\b}(X)\le H_{1,\b}(0)=h_{1,\b}(x_0),\q\b\in(-1,0).
\end{array}
\right.
\eeqn
To complete the proof of the case $\a=1$, we rewrite (\ref{5.3}) as
$H_{1,\b}'(X)=(1-|\b|^{-1})\wtil H_{1,\b}(X)$, where
$\wtil H_{1,\b}(X)$ is the nonnegative function 
$[H_{1,\b}(X)]^{-1}X$, $X\in[0,-x_0)$.
Differentiating $\wtil H_{1,\b}(X)$ with respect to $X$ and
using (\ref{5.1})--(\ref{5.3}) with $\a=1$, we get that $\wtil H_{1,\b}$
is an increasing function, since
\beqn\no
\hskip -1truecm
\wtil H_{1,\b}'(X)
&\!\!\!=\!\!\!&
[H_{1,\b}(X)]^{-3}\big\{[H_{1,\b}(X)]^2-(1-|\b|^{-1})X^2\big\}\no
\\[2mm]
\hskip -1truecm
&\!\!\!=\!\!\!&
[H_{1,\b}(X)]^{-3}|\b|^{-1}x_0^2> 0,\q\forall\,X\in[0,-x_0).\no
\eeqn
As a consequence, if $\b<-1$, then $1-|\b|^{-1}>0$, and
$H_{1,\b}'(X)=(1-|\b|^{-1})\wtil H_{1,\b}(X)$ is increasing too, proving
the convexity of $h_{1,\b}(x)$, $x\in[2x_0,0]$. On the contrary, if $\b\in(-1,0)$,
then $1-|\b|^{-1}<0$ and $H_{1,\b}'(X)=(1-|\b|^{-1})\wtil H_{1,\b}(X)$
turns out to be a decreasing function, proving the concavity of
$h_{1,\b}(x)$, $x\in[2x_0,0]$.

\noindent{\it ii) Case $\a>1$}.
In this case from (\ref{5.3}) we find $H_{\a,\b}'(X)\ge 0$ if and only if
$1-\a|\b|^{-1}[G_{\a,\b}(X)]^{\a-1}\ge 0$, i.e. if and only if
$[G_{\a,\b}(X)]^{\a-1}\le\a^{-1}|\b|$. Raising both sides of this latter inequality
to the power $\g_\a=(\a+1)/(\a-1)>1$ and using (\ref{5.2}) we thus find
$H_{\a,\b}'(X)\ge 0$ if and only if $X^2\ge x_0^2-c_{\a,\b}$, where
$c_{\a,\b}$ is defined by (\ref{4.33}).
Therefore, if $x_0\in[-M_{\a,\b},0)$, $M_{\a,\b}=(c_{\a,\b})^{1/2}$,
then $H_{\a,\b}'(X)\ge 0$ for every $X\in[0,-x_0)$ and $H_{\a,\b}$ 
is a non decreasing function in $[0,-x_0]$. Hence, (\ref{4.41}) follows from
$h_{\a,\b}(x_0)=H_{\a,\b}(0)\le H_{\a,\b}(X)\le H_{\a,\b}(-x_0)=h_{\a,\b}(0)$.
On the contrary, if $x_0<-M_{\a,\b}$, we have 
$H_{\a,\b}'(X)\le 0$ in $[0,X_{\a,\b;+}]$ 
and $H_{\a,\b}'(X)\ge 0$ in $[X_{\a,\b;+},-x_0)$, where 
$X_{\a,\b;+}={[x_0^2-c_{\a,\b}]}^{1/2}$.
Hence $H_{\a,\b}$ is non increasing in $[0,X_{\a,\b;+}]$ and 
non decreasing in $[X_{\a,\b;+},-x_0)$, and (\ref{4.42}) follows from
$h_{\a,\b}(x_{\a,\b;+})=H_{\a,\b}(X_{\a,\b;+})\le H_{\a,\b}(X)\le
\max\{H_{\a,\b}(0),H_{\a,\b}(-x_0)\}=\max\{h_{\a,\b}(x_0),h_{\a,\b}(0)\}$.
To complete the proof of the case $\a>1$, 
let us assume first $x_0\in[-M_{\a,\b},0)$.
As we noted above, in such a case the function 
$\wtil G_{\a,\b}(X):=1-\a|\b|^{-1}[G_{\a,\b}(X)]^{\a-1}$, $X\in[0,-x_0)$, 
is nonnegative and, moreover, is non decreasing 
due to the non increasing character of $G_{\a,\b}$. 
An easy computations, taking into account formulae (\ref{5.1}) and (\ref{5.3}),
shows that the nonnegative function $\wtil H_{\a,\b}(X):=[H_{\a,\b}(X)]^{-1}X$, 
$X\in[0,-x_0)$, is non decreasing too, since for every $X\in[0,-x_0)$ we have
\beqn\label{5.4}
\hskip -0,5truecm  
\wtil H_{\a,\b}'(X)
&\!\!\!=\!\!\!&
[H_{\a,\b}(X)]^{-3}
\big\{[H_{\a,\b}(X)]^2-\big(1-\a|\b|^{-1}[G_{\a,\b}(X)]^{\a-1}\big)X^2\big\}\no
\\[2mm]
\hskip -0,5truecm
&\!\!\!=\!\!\!&
[H_{\a,\b}(X)]^{-3}
\big\{\b^{-2}[G_{\a,\b}(X)]^{\a+1}+\a|\b|^{-1}X^2\big\}
[G_{\a,\b}(X)]^{\a-1}\ge 0.
\eeqn
Then, if we take $0\le X_1\le X_2<-x_0$, 
from $0\le\wtil G_{\a,\b}(X_1)\le\wtil G_{\a,\b}(X_2)$,
$0\le \wtil H_{\a,\b}(X_1)\le\wtil H_{\a,\b}(X_2)$ and 
$H_{\a,\b}'(X)=\wtil G_{\a,\b}(X)\wtil H_{\a,\b}(X)$ we derive
$0\le H_{\a,\b}'(X_1)\le H_{\a,\b}'(X_2)$. 
Thus $H_{\a,\b}'$ is a non decreasing function
or, equivalently, $H_{\a,\b}$ is a convex function 
completing the proof of {\it ii-a)}.
Let us now assume $x_0<-M_{\a,\b}$. Since in this case we have
$\wtil G_{\a,\b}(X)\ge 0$ if and only if $X\in[X_{\a,\b;+},-x_0]$, 
the previous argument can be used only to show that 
$H_{\a,\b}$ is still a convex function in the interval $[X_{\a,\b;+},-x_0]$.
To see what happens in the interval $[0,X_{\a,\b;+}]$ 
we analyze the second derivative of $H_{\a,\b}$ with respect to $X$. 
To this purpose it is more convenient to differentiate the expression
$H_{\a,\b}'(X)=\wtil G_{\a,\b}(X)\wtil H_{\a,\b}(X)$ where 
$\wtil G_{\a,\b}$ and $\wtil H_{\a,\b}$ are as before.
Using formula (\ref{5.4}) for $\wtil H_{\a,\b}'(X)$ and 
\beqn\label{5.5}
\hskip -0truecm
\wtil G_{\a,\b}'(X)
=-\a(\a-1)|\b|^{-1}[G_{\a,\b}(X)]^{\a-2}G_{\a,\b}'(X)
=\a(\a-1)[G_{\a,\b}(X)]^{-2}X,
\eeqn
we thus find
\beqn\label{5.6}
\hskip -1truecm
H_{\a,\b}''(X)
&\!\!\!=\!\!\!&
\wtil G_{\a,\b}'(X)\wtil H_{\a,\b}(X)+\wtil G_{\a,\b}(X)\wtil H_{\a,\b}'(X)\no
\\[2mm]
\hskip -1truecm
&\!\!\!=\!\!\!&
\wtil G_{\a,\b}'(X)[H_{\a,\b}(X)]^{-1}X++\wtil G_{\a,\b}(X)\wtil H_{\a,\b}'(X)\no
\\[2mm]
\hskip -1truecm
&\!\!\!=\!\!\!&
[G_{\a,\b}(X)]^{-2}[H_{\a\,\b}(X)]^{-3}N_{\a,\b}(X),
\q X\in[0,X_{\a,\b;+}].
\eeqn
Here
\beqn\label{5.7}
\hskip -1truecm
&&N_{\a,\b}(X)\no
\\[2mm]
\hskip -1truecm
&\!\!\!=\!\!\!&
\a(\a-1)[H_{\a,\b}(X)]^2X^2+\wtil G_{\a,\b}(X)
\big\{\b^{-2}[G_{\a,\b}(X)]^{\a+1}+\a|\b|^{-1}X^2\big\}
[G_{\a,\b}(X)]^{\a+1}\no
\\[2mm]
\hskip -1truecm
&\!\!\!=\!\!\!&
\a(\a-1)[H_{\a,\b}(X)]^2X^2+
\Big(\frac{\a+1}{4}\Big)\big[\a(x_0^4-X^4)+(x_0^2-X^2)^2\big]
\wtil G_{\a,\b}(X),
\eeqn
where in the latter equality we have used
$[G_{\a,\b}(X)]^{\a+1}=2^{-1}(\a+1)|\b|(x_0^2-X^2)$.
Now, since for $x_0<-M_{\a,\b}$ we have
$\wtil G_{\a,\b}(0)=1-\a|\b|^{-1}[G_{\a,\b}(0)]^{\a-1}<0$
and $\wtil G_{\a,\b}(X_{\a,\b;+})=0$, from (\ref{5.7}) it follows 
\beqn\label{5.8}
\left\{\!
\begin{array}{lll}
N_{\a,\b}(0)=[(\a+1)x_0^2/2]^2\wtil G_{\a,\b}(0)<0,
\\[2mm]
N_{\a,\b}(X_{\a,\b;+})=\a(\a-1)[H_{\a,\b}(X_{\a,\b;+})]^2[X_{\a,\b;+}]^2>0.
\end{array}
\right.
\eeqn
Consequently, from (\ref{5.6}) we deduce 
$H_{\a,\b}''(0)<0<H_{\a,\b}''(X_{\a,\b;+})$. 
To complete the proof of {\it ii-b)} it then suffices to show that 
$N_{\a,\b}(X)$ is an increasing function in $[0,X_{\a,\b;+}]$. 
Indeed, $N_{\a,\b}(X)$ being continuous, by virtue of the Mean Value Theorem
this will imply that there exists a unique $\wtil X_{\a,\b}\in(0,X_{\a,\b;+})$ 
such that $N_{\a,\b}(X)<0$ for $X\in[0,\wtil X_{\a,\b})$, 
$N_{\a,\b}(\wtil X_{\a,\b})=0$ 
and $N_{\a,\b}(X)>0$ for $X\in(\wtil X_{\a,\b},X_{\a,\b;+}]$.
Thus, from (\ref{5.6}) we shall derive 
$H_{\a,\b}''(X)<0$ for $X\in[0,\wtil X_{\a,\b})$, $H_{\a,\b}(\wtil X_{\a,\b})=0$ 
and $H_{\a,\b}''(X)>0$ for $X\in(\wtil X_{\a,\b},X_{\a,\b;+}]$, 
and {\it ii-b)} will be proved with 
$\wtil x_{\a,\b}=x_0+\wtil X_{\a,\b}\in(x_0,x_{\a,\b;+})$.
Now, differentiating (\ref{5.7}) with respect to $X$ and using (see (\ref{5.3}))
$H_{\a,\b}'(X)H_{\a,\b}(X)=X\wtil G_{\a,\b}(X)$ and (\ref{5.5}), we find
\beqn\label{5.9}
\hskip -1truecm
N_{\a,\b}'(X)
&\!\!\!=\!\!\!&
2\a(\a-1)\big\{X^2\wtil G_{\a,\b}(X)+[H_{\a,\b}(X)]^2\big\}X
-(\a+1)\big[x_0^2+(\a-1)X^2\big]\wtil G_{\a,\b}(X)X\no
\\[2mm]
\hskip -1truecm
&&
+\Big[\frac{\a(\a^2-1)}{4}\Big]\big[\a(x_0^4-X^4)+(x_0^2-X^2)^2\big]
[G_{\a,\b}(X)]^{-2}X\no
\\[2mm]
\hskip -1truecm
&\!\!\!=\!\!\!&
[G_{\a,\b}(X)]^{-2}XJ_{\a,\b}(X),\q \forall\,X\in(0,X_{\a,\b;+}),
\eeqn
where
\beqn\label{5.10}
\hskip -0,5truecm
J_{\a,\b}(X)
&\!\!\!=\!\!\!&
\big[(\a-1)^2X^2-(\a+1)x_0^2\big]\wtil G_{\a,\b}(X)[G_{\a,\b}(X)]^{2}
+2\a(\a-1)[H_{\a,\b}(X)]^2[G_{\a,\b}(X)]^2\no
\\[2mm]
\hskip -0,5truecm
&&+\Big[\frac{\a(\a^2-1)}{4}\Big]\big[\a(x_0^4-X^4)+(x_0^2-X^2)^2\big],
\q X\in (0,X_{\a,\b;+}).
\eeqn
If we can show $J_{\a,\b}(X)>0$ in $(0,X_{\a,\b;+})$, then from (\ref{5.9}) 
we get $N_{\a,\b}'(X)>0$ and $N_{\a,\b}(X)$ is an increasing function 
in $[0,X_{\a,\b;+}]$ completing our proof. 
Hence, it is enough to show that $J_{\a,\b}$ 
is a decreasing function and that $J_{\a,\b}(X_{\a,\b;+})>0$ for $x_0<-M_{\a,\b}$.
Before going on, and in order to justify the forthcoming computations,
we want to stress that the positivity of the function $J_{\a,\b}$
in the interval $(0,X_{\a,\b;+})$ can not be deduced immediately from
formula (\ref{5.10}). To see this we observe that
the last two terms in (\ref{5.10}) are clearly positive for $\a>1$,
but the first term may be negative. Indeed, since $\wtil G_{\a,\b}(X)<0$ in
$(0,X_{\a,\b;+})$, for the first term to be nonnegative we need the inequality
$(\a-1)^2X^2-(\a+1)x_0^2\le0$, which, when $\a>1$, is satisfied
only for $X\in(0,s_{\a,x_0}]$, where $s_{\a,x_0}=-(\a-1)^{-1}(\a+1)^{1/2}x_0$.
But, if $\a>3$, then $(0,s_{\a,x_0}]\subsetneq (0,-x_0)$,
and we cannot ensure that the first term is positive in $(0,X_{\a,\b;+})$.
In fact, since for $\a\to+\infty$ we have $s_{\a,x_0}\to 0^+$
and $X_{\a,\b,+}\to(-x_0)^-$, for $\a>3$ large enough we have
$(0,s_{\a,x_0}]\subsetneq (0,X_{\a,\b;+})$ and the first term
becomes a negative one in $(s_{\a,x_0},X_{\a,\b;+})$.
To prove that $J_{\a,\b}$ is decreasing we study the sign 
of its first derivative $J_{\a,\b}'(X)$, but first it is convenient 
to rewrite $J_{\a,\b}$ in a easier way. In fact, we have
\beqn\label{5.11}
\hskip -1truecm
\wtil G_{\a,\b}(X)[G_{\a,\b}(X)]^2
&\!\!\!=\!\!\!&
[G_{\a,\b}(X)]^2-\a|\b|^{-1}[G_{\a,\b}(X)]^{\a+1}\no
\\[2mm]
\hskip -1truecm
&\!\!\!=\!\!\!&
[G_{\a,\b}(X)]^2-\Big[\frac{\a(\a+1)}{2}\Big](x_0^2-X^2),
\eeqn
and
\beqn\label{5.12}
\hskip -1truecm
[H_{\a,\b}(X)]^2[G_{\a,\b}(X)]^2
&\!\!\!=\!\!\!&
X^2[G_{\a,\b}(X)]^2+\b^{-2}[G_{\a,\b}(X)]^{2(\a+1)}\no
\\[2mm]
\hskip-1truecm
&\!\!\!=\!\!\!&
X^2[G_{\a,\b}(X)]^2+\Big[\frac{(\a+1)^2}{4}\Big](x_0^2-X^2)^2.
\eeqn
Therefore, replacing (\ref{5.11}) and (\ref{5.12}) in (\ref{5.10}), we get
\beqn\label{5.13}
&&\hskip -1truecm
J_{\a,\b}(X)=P_\a(X;x_0)[G_{\a,\b}(X)]^2+Q_\a(X;x_0),
\eeqn
where
\beqn\label{5.14}
&&\hskip -1truecm
P_\a(X;x_0)=(\a-1)^2X^2-(\a+1)x_0^2+2\a(a-1)X^2
=(3\a^2-4\a+1)X^2-(\a+1)x_0^2,
\eeqn
and
\beqn
\hskip -1truecm
Q_\a(X;x_0)
&\!\!\!=\!\!\!&
-\Big[\frac{\a(\a+1)}{2}\Big]\big[(\a-1)^2X^2-(\a+1)x_0^2\big](x_0^2-X^2)\no
\\[2mm]
\hskip -1truecm
&&
+\Big[\frac{2\a(\a-1)(\a+1)^2}{4}\Big](x_0^2-X^2)^2\no
\\
[2mm]
\hskip -1truecm
&&
+\Big[\frac{\a(\a^2-1)}{4}\Big]\big[\a(x_0^4-X^4)+(x_0^2-X^2)^2\big].\no
\eeqn
Now, a lengthy but easy computation yields
\beqn\no
\hskip -1truecm
Q_\a(X;x_0)
&\!\!\!=\!\!\!&
\Big[\frac{\a(\a+1)}{4}\Big]
\big[(3\a^2-2\a-1)X^4-2(3\a^2-1)x_0^2X^2+(3\a^2+2\a-1)x_0^4\big],
\eeqn
so from (\ref{5.13}) and (\ref{5.14}) we obtain
\beqn\label{5.15}
\hskip -1truecm
J_{\a,\b}(X)
&\!\!\!=\!\!\!&
\big[(3\a^2-4\a+1)X^2-(\a+1)x_0^2\big][G_{\a,\b}(X)]^2\no
\\[1mm]
\hskip -1truecm
&&
+\Big[\frac{\a(\a+1)}{4}\Big]
\big[(3\a^2-2\a-1)X^4-2(3\a^2-1)x_0^2X^2+(3\a^2+2\a-1)x_0^4\big].
\eeqn
Now, using $G_{\a,\b}'(X)=\b X[G_{\a,\b}(X)]^{-\a}$, from (\ref{5.15}) 
it follows
\beqn\label{5.8.6}
\hskip -0,5truecm
J_{\a,\b}'(X)
&\!\!\!=\!\!\!&
2(3\a^2-4\a+1)X[G_{\a,\b}(X)]^2+
2\b\big[(3\a^2-4\a+1)X^2-(\a+1)x_0^2\big]X[G_{\a,\b}(X)]^{1-\a}\no
\\[2mm]
\hskip -0,5truecm
&&
+\a(\a+1)\big[(3\a^2-2\a-1)X^2-(3\a^2-1)x_0^2\big]X\no
\\[2mm]
\hskip -0,5truecm
&\!\!\!=\!\!\!&
X[G_{\a,\b}(X)]^{1-\a}\wtil J_{\a,\b}(X),\no
\eeqn
where
\beqn\label{5.16}
\hskip -0,5truecm
\wtil J_{\a,\b}(X)
&\!\!\!=\!\!\!&
2(3\a^2-4\a+1)[G_{\a,\b}(X)]^{\a+1}
+2\b\big[(3\a^2-4\a+1)X^2-(\a+1)x_0^2\big]\no
\\[2mm]
\hskip -0,5truecm
&&
+\a(\a+1)\big[(3\a^2-2\a-1)X^2-(3\a^2-1)x_0^2\big][G_{\a,\b}(X)]^{\a-1}\no
\\[2mm]
\hskip -0,5truecm
&\!\!\!=\!\!\!&
(3\a^2-4\a+1)(\a+1)|\b|(x_0^2-X^2)
-2|\b|\big[(3\a^2-4\a+1)X^2-(\a+1)x_0^2\big]\no
\\[2mm]
\hskip -0,5truecm
&&
-\a(\a+1)\big[(3\a^2-1)x_0^2-(3\a^2-2\a-1)X^2\big][G_{\a,\b}(X)]^{\a-1}\no
\\[2mm]
&\!\!\!=\!\!\!&
|\b|\big[(\a+1)(3\a^2-4\a+3)x_0^2-(\a+3)(3\a^2-4\a+1)X^2\big]\no
\\[2mm]
\hskip -0,5truecm
&&
-\a(\a+1)\big[(3\a^2-1)x_0^2-(3\a^2-2\a-1)X^2\big][G_{\a,\b}(X)]^{\a-1}.
\eeqn
Hence, $J_{\a,\b}'(X)< 0$, $X\in(0,X_{\a,\b;+})$, i.e. $J_{\a,\b}$ is a 
decreasing function, if and only if $\wtil J_{\a,\b}(X)<0$, $X\in(0,X_{\a,\b;+})$. 
To prove that $\wtil J_{\a,\b}(X)<0$, $X\in(0,X_{\a,\b;+})$, we show
that $\wtil J_{\a,\b}(X)\ge0$, $X\in(0,X_{\a,\b;+})$, leads to a contradiction.
First we observe that, if $\a>1$, then $0<3\a^2-2\a-1<3\a^2-1$,
so $(3\a^2-1)x_0^2-(3\a^2-2\a-1)X^2>0$
for every $X\in(0,r_{\a,x_0})$, where 
$r_{\a,x_0}=-[(3\a^2-1)/(3\a^2-2\a-1)]^{1/2}x_0>-x_0$. Hence, a fortiori,
$(3\a^2-1)x_0^2-(3\a^2-2\a-1)X^2>0$ for $X\in(0,X_{\a,\b;+})\subsetneq(0,-x_0)$.
It thus follows from (\ref{5.16}) that $\wtil J_{\a,\b}(X)\ge 0$, $X\in(0,X_{\a,\b;+})$,
is equivalent to
\beqn\label{5.17}
&&\hskip -1truecm
[G_{\a,\b}(X)]^{\a-1}\le R_{\a,\b}(X),\q X\in (0,X_{\a,\b;+}),
\eeqn
where
\beqn
&&\hskip -1truecm
R_{\a,\b}(X):=
\frac{|\b|\big[(\a+1)(3\a^2-4\a+3)x_0^2-(\a+3)(3\a^2-4\a+1)X^2\big]}
{\a(\a+1)\big[(3\a^2-1)x_0^2-(3\a^2-2\a-1)X^2\big]}.\no
\eeqn
But
\beqn
&&\hskip -1truecm
R_{\a,\b}'(X)=\frac{-16|\b|(3\a+2)(\a-1)^{2}x_0^2X}
{(\a+1)\big[(3\a^2-1)x_0^2-(3\a^2-2\a-1)X^2\big]^2},\no
\eeqn
so $R_{\a,\b}$ is a decreasing function in $(0,X_{\a,\b;+})$ for
every $\a>1$ and $\b<0$. Therefore, 
$R_{\a,\b}(X)<R_{\a,\b}(0)=\a^{-1}|\b|(3\a^2-4\a+3)(3\a^2-1)^{-1}<\a^{-1}|\b|$
for every $X\in(0,X_{\a,\b;+})$, $\a>1$, $\b<0$.  On the other side, 
the function $G_{\a,\b}$ being decreasing, for $\a>1$ we have 
$\a^{-1}|\b|=[G_{\a,\b}(X_{\a,\b;+})]^{\a-1}<[G_{\a,\b}(X)]^{\a-1}$,
so that $R_{\a,\b}(X)<\a^{-1}|\b|<[G_{\a,\b}(X)]^{\a-1}$ for every
$X\in(0,X_{\a,\b;+})$ which is incompatible with (\ref{5.17}).
This proves that $\wtil J_{\a,\b}(X)<0$, $X\in(0,X_{\a,\b;+})$,
and hence that $J_{\a,\b}$ is a decreasing function in $(0,X_{\a,\b;+})$.
It remains only to show that $J_{\a,\b}(X_{\a,\b;+})$ is positive 
for $x_0<-M_{\a,\b}$. But, since $\wtil G_{\a,\b}(X_{\a,\b;+})=0$, 
this trivially follows from definition (\ref{5.10}) 
completing the proof of the case {\it ii-b)}.

\noindent{\it iii) Case $\a\in(0,1)$}. 
Since in this case the power $\g_\a=(\a+1)/(\a-1)$ is less than $-1$, 
from (\ref{5.3}) it follows that $H_{\a,\b}'(X)> 0$ if and only if 
$[G_{\a,\b}(X)]^{(\a-1)\g_\a}>\a^{-\g_\a}|\b|^{\g_\a}$,
i.e. if an only if $X^2<x_0^2-c_{\a,\b}$. Therefore, if $x_0\in[-M_{\a,\b},0)$,
then we have $H_{\a,\b}'(X)\le 0$ for every $X\in[0,-x_0)$ and $H_{\a,\b}$
is a non increasing function in $[0,-x_0)$. Hence, (\ref{4.43})
follows from $h_{\a,\b}(0)=H_{\a,\b}(-x_0)\le H_{\a,\b}(X)
\le H_{\a,\b}(0)=h_{\a,\b}(x_0)$. On the contrary, if $x_0<-M_{\a,\b}$,
then we have $H_{\a,\b}'(X)\ge 0$ in $[0,X_{\a,\b;+}]$ and $H_{\a,\b}'(X)\le 0$
in $[X_{\a,\b;+},-x_0)$, so that $H_{\a,\b}$ is non decreasing in $[0,X_{\a,\b;+}]$
and non increasing in $[X_{\a,\b;+},-x_0)$. Thus, (\ref{4.44}) follows
from $\min\{h_{\a,\b}(x_0),h_{\a,\b}(0)\}=\min\{H_{\a,\b}(0),H_{\a,\b}(-x_0)\}
\le H_{\a,\b}(X)\le H_{\a,\b}(X_{\a,\b;+})=h_{\a,\b}(x_{\a,\b;+})$.
To complete the proof, let us first assume $x_0\in[-M_{\a,\b},0)$.
Since $\a\in(0,1)$, this time the function
$\wtil G_{\a,\b}(X)=1-\a|\b|^{-1}[G_{\a,\b}(X)]^{\a-1}$, $X\in[0,-x_0)$, 
is non positive and non increasing (see (\ref{5.5})). 
Consequently,  the function 
$\wtil H_{\a,\b}(X)=[H_{\a,\b}(X)]^{-1}X$, $X\in[0,-x_0)$, being
nonnegative and non decreasing (see (\ref{5.4})), if we take
$0\le X_1\le X_2<-x_0$, from $\wtil G_{\a,\b}(X_2)\le\wtil G_{\a,\b}(X_1)\le 0$ 
and $0\le \wtil H_{\a,\b}(X_1)\le\wtil H_{\a,\b}(X_2)$  we derive
$H_{\a,\b}'(X_2)=\wtil G_{\a,\b}(X_2)\wtil H_{\a,\b}(X_2)
\le \wtil G_{\a,\b}(X_1)\wtil H_{\a,\b}(X_1)=H_{\a,\b}'(X_1)\le 0$. 
Thus $H_{\a,\b}'$ is a non increasing function
or, equivalently, $H_{\a,\b}$ is a concave function 
completing the proof of {\it iii-a)}. Let now $x_0<-M_{\a,\b}$.
Then, the non increasing function $\wtil G_{\a,\b}$ is 
positive in $[0,X_{\a,\b;+})$, vanishes at $X=X_{\a,\b;+}$ and 
is negative in $(X_{\a,\b;+},-x_0)$, so the previous argument may be 
employed only to show that $H_{\a,\b}$ is concave in $[X_{\a,\b,+},-x_0)$. 
As far as the interval $[0,X_{\a,\b,+}]$ is concerned, this time 
the function $N_{\a,\b}$ defined by (\ref{5.7}) satisfies (\ref{5.8}) 
with the reversed inequalities, i.e. $N_{\a,\b}(0)>0$ and 
$N_{\a,\b}(X_{\a,\b;+})<0$. Therefore
$H_{\a,\b}''(0)>0>H_{\a,\b}''(X_{\a,\b;+})$ and to complete the proof
it suffices to show that $N_{\a,\b}$ is a decreasing function in $[0,X_{\a,\b;+}]$.
Due to (\ref{5.9}), we only need to prove that the function $J_{\a,\b}$
defined by (\ref{5.10}) is negative in $(0,X_{\a,\b;+})$.
This is true, since all the three terms in (\ref{5.10}) are negative 
in $(0,X_{\a,\b;+})$. Indeed, the last two terms are clearly negative 
for $\a\in(0,1)$, but the first term is negative, too. For, when $\a\in(0,1)$, 
$\wtil G_{\a,\b}(X)>0$, $X\in(0,X_{\a,\b;+})$,
and $(\a-1)^2X^2-(\a+1)x_0^2<0$, $X\in(0,s_{\a,x_0})$,
where $s_{\a,x_0}=-(1-\a)^{-1}(\a+1)^{1/2}x_0>-x_0$.
This completes the proof.
\begin{remark}\label{rem5.1}
\emph{We stress that in the previous proof 
the standard procedure of calculus for locating the
inflection point of $H_{\a,\b}$, $\a>0>\b$, $\a\neq 1$,
is not profitable, due to the difficulty in studying the sign of $H_{\a,\b}''$
(see formulae (\ref{5.6}) and (\ref{5.7})).}
\end{remark}
As usual, for any function $f:I\subset\rsp\to\rsp$, $I$ an interval, we denote
by $f^+=\max\{f,0\}$ and $f^-=\max\{-f,0\}$ its positive and negative parts, respectively, such that $f=f^+-f^-$, $|f|=f^++f^-$ and $-f^-\le f\le f^+$.
\vskip 0,3truecm
\noindent{\it Proof of Corollary \ref{cor4.10}.}
Let $x_\a$ and $\wtil x_\a$, $\a>0$, be the points defined by (\ref{4.46}).
Of course, the function $\theta_\a$ in (\ref{4.29})
decreases for $x\in[2x_0,\wtil x_\a]$ and increases for $x\in[\wtil x_\a,0]$.
Moreover, it is positive in $[2x_0,x_\a)$, 
negative in $(x_\a,0)$, vanishes at $x=x_\a$ and $x=0$, 
and satisfies $\theta_\a(2x_0)=2(4+\a)(2-d_\a)x_0^2=6x_0^2$ for every $\a>0$.
Recall also that $x_\a\in(2x_0,x_0)$. 
Then, $h_{\a,\b}$ being a positive function, we have
\beqn\no
&&\hskip -1truecm
\Theta_{\a,\b}^+(x)=\left\{\!\!
\begin{array}{lll}
\Theta_{\a,\b}(x),\ x\in [2x_0,x_\a],
\\[2mm]
0,\ x\in[x_\a,0],
\end{array}
\right.
\q
\Theta_{\a,\b}^-(x)=\left\{\!\!
\begin{array}{lll}
0,\ x\in [2x_0,x_\a],
\\[2mm]
-\Theta_{\a,\b}(x),\ x\in[x_\a,0].
\end{array}
\right.
\eeqn
{\it i) Case $\a=1$}. 
If $\b\le-1$, then Lemma \ref{lem4.8}{\it i)} implies that $h_{1,\b}$ 
does not increase in $[2x_0,x_0]$ and hence $h_{1,\b}(x_1)\le h_{1,\b}(x)$ 
for every $x\in[2x_0,x_1]$, $d_1=7/5$. Therefore
\beqn\no
&&\hskip -1truecm
\Theta_{1,\b}(x)\le \Theta_{1,\b}^+(x)
\le[h_{1,\b}(x_1)]^{-1}\theta_1(2x_0),\q\forall\,x\in[2x_0,0].
\eeqn
Instead, the estimate from below follows by combining 
the inequalities $\theta_1(\wtil x_1)\le \theta_1(x)<0$ and 
$0<h_{1,\b}(x_0)\le h_{1,\b}(x)$ for $x\in(x_1,0]$, which yield 
\beqn\no
&&\hskip -1truecm
[h_{1,\b}(x_0)]^{-1}\theta_1(\wtil x_1)\le -\Theta_{1,\b}^-(x)\le \Theta_{1,\b}(x),
\q\forall\,x\in[2x_0,0].
\eeqn
If $\b\in(-1,0)$, then $h_{1,\b}$ is concave and attains 
its minimum value $|x_0|$ at the point $x=2x_0$ and $x=0$. 
We thus find $h_{1,\b}(2x_0)\le h_{1,\b}(x)$
for every $x\in[2x_0,x_1]$ and $h_{1,\b}(0)\le h_{1,\b}(x)$ for
every $x\in[x_1,0]$. These yield
\beqn\no
&&\hskip -1truecm
[h_{1,\b}(0)]^{-1}\theta_1(\wtil x_1)\le
-\Theta_{1,\b}^-(x)\le \Theta_{1,\b}(x)\le\Theta_{1,\b}^+(x)
\le\Theta_{1,\b}(2x_0),\q\forall\,x\in[2x_0,0],
\eeqn
completing the proof of (\ref{4.49})--(\ref{4.51}).

\noindent
{\it ii) Case $\a>1$}.
Assume first $x_0\in[-M_{\a,\b},0)$, $\b<0$. 
In this case the proof of (\ref{4.52}) is the same as that for the
case $\a=1$, $\b\le-1$. In fact, due to Lemma \ref{lem4.8}{\it ii-a)} 
we have that $h_{\a,\b}$ does not increase in $[2x_0,x_0]$ and hence
$h_{\a,\b}(x_\a)\le h_{\a,\b}(x)$ for $x\in[2x_0,x_\a]$. 
Therefore
\beqn\no
&&\hskip -1truecm
\Theta_{\a,\b}(x)\le \Theta_{\a,\b}^+(x)
\le[h_{\a,\b}(x_\a)]^{-1}\theta_\a(2x_0),\q\forall\,x\in[2x_0,0].
\eeqn
The estimate from below follows by combining 
$\theta_\a(\wtil x_\a)\le \theta_\a(x)<0$ and 
$0<h_{\a,\b}(x_0)\le h_{\a,\b}(x)$ for $x\in(x_\a,0)$, which yield 
\beqn\no
&&\hskip -1truecm
[h_{\a,\b}(x_0)]^{-1}\theta_\a(\wtil x_\a)\le -\Theta_{\a,\b}^-(x)\le
\Theta_{\a,\b}(x),\q\forall\,x\in[2x_0,0],
\eeqn
i.e. (\ref{4.52}) with $C_j(\a,\b,x_0)$, $j=8,9$, 
being defined by the first expressions in (\ref{4.53}) and (\ref{4.54}).
Let us now take $x_0<-M_{\a,\b}$. In this case, according to 
Lemma \ref{lem4.8}{\it ii-b)}, the function $h_{\a,\b}$ attains 
its least value at both the points $x_{\a,\b;\pm}$. Also, being convex in 
$[2x_0,2x_0-\wtil x_{\a,\b}]\cup[\wtil x_{\a,\b},0]$, $\wtil x_{\a,\b}\in(x_0,x_{\a,\b;+})$,
and concave in $[2x_0-\wtil x_{\a,\b},\wtil x_{\a,\b}]$ with a local maximum 
at $x=x_0$, it decreases in $[2x_0,x_{\a,\b;-}]\cup[x_0,x_{\a,\b;+}]$.
We then distinguish the two sub-cases $x_0\in(-E_{\a,\b},-M_{\a,\b})$ 
and $x_0\le-E_{\a,\b}$, corresponding to $x_\a< x_{\a,\b;-}$ 
and $x_{\a,\b;-}\le x_\a$, respectively.
Let first be $x_0\in(-E_{\a,\b},-M_{\a,\b})$. 
Since $x_\a<x_{\a,\b;-}$ and $h_{\a,\b}$ decreases in $[2x_0,x_{\a,\b;-}]$,
the same reasonings as above for the case $x_0\in[-M_{\a,\b},0)$ lead to 
\beqn\no
&&\hskip -1truecm
[h_{\a,\b}(x_{\a,\b;+})]^{-1}\theta_\a(\wtil x_\a)\le \Theta_{\a,\b}(x)
\le  [h_{\a,\b}(x_\a)]^{-1}\theta_\a(2x_0),\q\forall\,x\in[2x_0,0],
\eeqn
i.e. (\ref{4.52}) with $C_{8}(\a,\b,x_0)$ and $C_{9}(\a,\b,x_0)$ 
being defined, respectively, by the second expression in (\ref{4.53}) 
and the first expression in (\ref{4.54}). 
Finally, if $x_0\le-E_{\a,\b}$, since $x_{\a,\b;-}\le x_\a$, 
estimate (\ref{4.52}), with $C_j(\a,\b,x_0)$, $j=8,9$, being defined 
by the second expressions in (\ref{4.53}) and (\ref{4.54}), 
follows from $\Theta_{\a,\b}^+(x)\le [h_{\a,\b}(x_{\a,\b;-})]^{-1}\theta_\a(2x_0)$ 
for $x\in[2x_0,x_\a]$ and 
$[h_{\a,\b}(x_{\a,\b;+})]^{-1}\theta_\a(\wtil x_\a)\le -\Theta_{\a,\b}^-(x)$ 
for $x\in[x_\a,0]$. 

\noindent
{\it iii) Case $\a\in(0,1)$}. 
As before, let first $x_0\in[-M_{\a,\b},0)$. Due to Lemma \ref{lem4.8}{\it iii-a)} 
we have that $h_{\a,\b}$ assumes its least value at both the points 
$x=2x_0$ and $x=0$. Consequently
\beqn
&&\hskip -1truecm
\Theta_{\a,\b}^+(x)\le\Theta_{\a,\b}(2x_0),
\q\forall\,x\in[2x_0,x_\a],\no
\\[2mm]
&&\hskip -1truecm
[h_{\a,\b}(0)]^{-1}\theta_\a(\wtil x_\a)\le-\Theta_{\a,\b}^-(x),
\q\forall x\in[x_\a,0],\no
\eeqn
completing the proof of (\ref{4.55}) with $C_j(\a,\b,x_0)$, $j=10,11$, 
being defined by the first expressions in (\ref{4.56}) and (\ref{4.57}).
Let now $x_0<-M_{\a,\b}$. Due to Lemma \ref{lem4.8}{\it iii-b)}
and the characterization (\ref{4.39}) of the constant $C_5(\a,\b,x_0)$,
we find that, if $x_0\in[-D_{\a,\b},-M_{\a,\b})$, then
$h_{\a,\b}(2x_0)\le h_{\a,\b}(x)$ for every $x\in[2x_0,x_\a]$
and $h_{\a,\b}(0)\le h_{\a,\b}(x)$ for every $x\in[x_\a,0]$.
It thus follows that for every $x_0\in[-D_{\a,\b},-M_{\a,\b})$
\beqn
&&\hskip -1truecm
\Theta_{\a,\b}^+(x)\le\Theta_{\a,\b}(2x_0),
\q\forall\,x\in[2x_0,x_\a],\no
\\[2mm]
&&\hskip -1truecm
[h_{\a,\b}(0)]^{-1}\theta_\a(\wtil x_\a)\le-\Theta_{\a,\b}^-(x),
\q\forall x\in[x_\a,0],\no
\eeqn
which prove (\ref{4.55}) with $C_j(\a,\b,x_0)$, $j=10,11$,
being defined by the first expressions in (\ref{4.56}) and (\ref{4.57}).
Finally, if $x_0<-D_{\a,\b}$, then 
$h_{\a,\b}$ increases in $[2x_0,x_{\a,\b;-}]\cup[x_0,x_{\a,\b;+}]$, 
decreases in $[x_{\a,\b;-},x_0]\cup[x_{\a,\b;+},0]$, and attains its
least value at the point $x=x_0$. Therefore, for every $x_0<-D_{\a,\b}$ 
we deduce
\beqn
&&\hskip -1truecm
\Theta_{\a,\b}^+(x)
\le[\min\{h_{\a,\b}(2x_0),h_{\a,\b}(x_\a)\}]^{-1}\theta_\a(2x_0),
\q\forall\,x\in[2x_0,x_\a],\no
\\[2mm]
&&\hskip -1truecm
[h_{\a,\b}(x_0)]^{-1}\theta_\a(\wtil x_\a)\le-\Theta_{\a,\b}^-(x),
\q\forall x\in[x_\a,0].\no
\eeqn
Due to (\ref{4.48}), the latter inequalities 
complete the proof of (\ref{4.55}), with $C_{10}(\a,\b,x_0)$
being defined by the second expression in (\ref{4.56}) and 
$C_{11}(\a,\b,x_0)$ being defined by the first or the second expression
in (\ref{4.57}) according that $x_0\in[-L_{\a,\b},-D_{\a,\b})$ or $x_0<-L_{\a,\b}$.
\begin{remark}\label{rem5.2}
\emph{Notice that, since from (\ref{5.3}) we obtain
$\Theta_{\a,\b}'(x)=[h_{\a,\b}(x)]^{-3}S_{\a,\b}(x)$ where
\beqn\no
&&\hskip -1truecm
S_{\a,\b}(x)=(4+\a)(2x-x_\a)[h_{\a,\b}(x)]^2
-\theta_\a(x)\{1-\a|\b|^{-1}[g_{\a,\b}(x)]^{\a-1}\}(x-x_0),
\eeqn
to find the greatest and least values of $\Theta_{\a,\b}$ 
by studying its first derivative $\Theta_{\a,\b}'$ it is not computationally amenable.
In some sense, the study of $\Theta_{\a,\b}'$ with the consequent 
location of the stationary points of $\Theta_{\a,\b}$ yields, more or less, 
to the same computational difficulties that we have 
highlighted in Remark \ref{rem5.1}, as regards to the study 
of $H_{\a,\b}''$ and the location of the inflection point of $H_{\a,\b}$.}
\end{remark}
\noindent{\it Proof of Lemma \ref{lem4.13}.}
{\it i) Case $\a=1/2$}. 
In this case, since $\varphi_{1/2,\b}=6|\b|^{-1}x$ and 
$g_{1/2,\b}=[3|\b|x(2x_0-x)/4]^{2/3}$, the function $\psi_{\a,\b}$ reduces 
to $\psi_{1/2,\b}=-3|\b|^{-1}x\psi_\b(x)$, $x\in[2x_0,0]$, where
\beqn\label{5.18}
&&\hskip -1truecm
\psi_\b(x)=|\b|^2x^2-3|\b|^2x_0x+2|\b|^2x_0^2-2.
\eeqn
Therefore, $\psi_{1/2,\b}(x)\ge 0$, $x\in[2x_0,0]$, if and only if 
$\psi_\b(x)\ge 0$, $x\in[2x_0,0]$. 
But, $\psi_\b(x)\ge 0$ if and only if $x\le x_{\b;-}$ and $x\ge x_{\b;+}$, where
the points $x_{\b;\pm}$ are defined in (\ref{4.59}).
Then, since $x_{\b;-}<2x_0$ for every $\b<0$, 
the function $\psi_{1/2,\b}$ turns out to be non positive 
in the entire interval $[2x_0,0]$ in the case $x_{\b;+}\ge 0$, 
corresponding to the choice $x_0\in[-|\b|^{-1},0)$. 
We thus split the proof of the case $\a=1/2$ 
in the two sub-cases $x_0\in[-|\b|^{-1},0)$ and $x_0<-|\b|^{-1}$.
Let first $x_0\in[-|\b|^{-1},0)$. 
Differentiating $\psi_{1/2,\b}$ with respect to $x$ we find
\beqn\no
&&\hskip -1truecm
\psi_{1/2,\b}'(x)=
-3|\b|^{-1}\big[3|\b|^2x^2-6|\b|^2x_0x+2|\b|^2x_0^2-2\big],\q x\in(2x_0,0),
\eeqn
so that $\psi_{1/2,\b}$ is non decreasing for $x\in[\wtil x_{\b;-},\wtil x_{\b;+}]$
and non increasing for $x\in[2x_0,\wtil x_{\b;-}]\cup[\wtil x_{\b;+},0]$, where
the points $\wtil x_{\b;\pm}$ are defined in (\ref{4.59}).
Now, since $x_0\in[-|\b|^{-1},0)$, we have $\wtil x_{\b;-}\le 2x_0$ and
$\wtil x_{\b;+}\ge 0$ and $\psi_{1/2,\b}$ is non decreasing in the
entire interval $[2x_0,0]$. Thus, 
$12|\b|^{-1}x_0=\psi_{1/2,\b}(2x_0)\le\psi_{1/2,\b}(x)<\psi_{1/2,\b}(0)=0$
for every $x\in[2x_0,0)$. Combining these estimates
with (\ref{4.43}) and (\ref{4.44}) and using
$D_{1/2,\b}=(3/4)|\b|^{-2}$ we obtain (\ref{4.62}) and (\ref{4.63}).
Let $x_0<-|\b|^{-1}$. In this case the points $x_{\b;+}$ and $\wtil x_{\b;\pm}$
are interior to the interval $[2x_0,0]$ and the function $\psi_{1/2,\b}$
is negative in $[2x_0,x_\b)$, positive in $(x_\b,0)$ and vanishes at $x=x_\b$
and $x=0$. In particular, $\wtil x_{\b;-}<x_{\b;+}<\wtil x_{\b;+}$,
so $\psi_{1/2,\b}(\wtil x_{\b;-})<0<\psi_{1/2,\b}(\wtil x_{\b;+})$. 
The inequality $\wtil x_{\b;-}<x_{\b;+}$ is obvious, 
since $\wtil x_{\b;-}<x_0<x_{\b;+}$. 
Also, since $\psi_{1/2,\b}(x_{\b;+})=\psi_{1/2,\b}(0)=0$ and $\psi_{1/2,\b}>0$ 
in $(x_{\b;+},0)$, as a consequence of Rolle's theorem the 
differentiable function $\psi_{1/2,\b}$ attains a positive maximum 
in some point of $(x_{\b;+},0)$, which is necessary a stationary point. 
This proves $x_{\b;+}<\wtil x_{\b;+}$.
Therefore, the estimate from above in (\ref{4.64}) simply follows by combining 
$0<\psi_{1/2,\b}(x)\le\psi_{1/2,\b}(\wtil x_{\b;+})$, $x\in(x_{\b;+},0)$ 
with $h_{1/2,\b}(0)\le h_{1/2,\b}(x)$, $x\in[x_{\b;+},0]$, $x_0\in[-D_{1/2,\b},0)$, 
and $\min\{h_{1/2,\b}(x_{\b;+}),h_{1/2,\b}(0)\}\le h_{1/2,\b}(x)$, $x\in[x_{\b;+},0]$,
$x_0<-D_{1/2,\b}$. On the contrary, the estimate from below in (\ref{4.64})
follows by combining $\psi_{1/2,\b}(\wtil x_{\b;-})\le\psi_{1/2,\b}(x)<0$,
$x\in[2x_0,x_{\b;+})$, with $h_{1/2,\b}(2x_0)\le h_{1/2,\b}(x)$, 
$x\in[2x_0,x_{\b;+})$, $x_0\in[-D_{1/2,\b},0)$, 
and $h_{1/2,\b}(x_0)\le h_{1/2,\b}(x)$, $x\in[2x_0,x_{\b;+})$, $x_0<-D_{1/2,\b}$. 

\noindent
{\it ii) Case $\a\in(1/2,1)$}. 
Due to (\ref{4.30}), the function $\varphi_{\a,\b}$ is negative in $(2x_0,0)$ 
and vanishes at $x=2x_0$ and $x=0$, whereas the function $\phi_{\a,\b}$
is negative in $(2x_0,x_0)$, positive in $(x_0,0)$ and vanishes at $x=2x_0$,
$x=x_0$ and $x=0$. Moreover, differentiating $\varphi_{\a,\b}$ and
$\phi_{\a,\b}$ with respect to $x$, using 
$g_{\a,\b}'(x)=\b(x-x_0)[g_{\a,\b}(x)]^{-\a}$, $x\in(2x_0,0)$, 
and $[g_{\a,\b}(x)]^{\a+1}=(\a+1)|\b|x(2x_0-x)/2$,
we obtain for every $x\in(2x_0,0)$
\beqn\label{5.19}
\hskip -0,5truecm
\varphi_{\a,\b}'(x)
&\!\!\!=\!\!\!&
6|\b|^{-1}\big\{[g_{\a,\b}(x)]^{(2\a-1)/2}
+[(2\a-1)/2]x[g_{\a,\b}(x)]^{(2\a-3)/2}g_{\a,\b}'(x)\big\}\no
\\[2mm]
\hskip -0,5truecm
&\!\!\!=\!\!\!&
6|\b|^{-1}\big\{[g_{\a,\b}(x)]^{\a+1}
+[(2\a-1)/2]\b x(x-x_0)\big\}[g_{\a,\b}(x)]^{-3/2}\no
\\[2mm]
\hskip -0,5truecm
&\!\!\!=\!\!\!&
-3x[3\a x-(4\a+1)x_0][g_{\a,\b}(x)]^{-3/2},
\\[2mm]
\label{5.20}
\hskip -0,5truecm
\phi_{\a,\b}'(x)
&\!\!\!=\!\!\!&
4\big\{[g_{\a,\b}(x)]^{3/2}+(3/2)(x-x_0)[g_{\a,\b}(x)]^{1/2}g_{\a,\b}'(x)\big\}\no
\\[2mm]
\hskip -0,5truecm
&\!\!\!=\!\!\!&
4\big\{[g_{\a,\b}(x)]^{\a+1}+(3/2)\b(x-x_0)^2\big\}[g_{\a,\b}(x)]^{(1-2\a)/2}\no
\\[2mm]
\hskip -0,5truecm
&\!\!\!=\!\!\!&
-2|\b|[(4+\a)x^2-2(4+\a)x_0x+3x_0^2][g_{\a,\b}(x)]^{(1-2\a)/2}.
\eeqn
From (\ref{5.19}) we find that $\varphi_{\a,\b}$  is non decreasing in $[\ov x_\a,0]$,
non increasing in $[2x_0,\ov x_\a]$, and attains its least negative value
at $x=\ov x_\a$, where $\ov x_\a\in[2x_0,x_0)$ is defined in (\ref{4.58}). 
Instead, from (\ref{5.20}) we deduce that $\phi_{\a,\b}$ is non decreasing
in $[\ov x_{\a;-},\ov x_{\a;+}]$, non increasing in $[2x_0,\ov x_{\a;-}]\cup
[\ov x_{\a;+},0]$, and attains its least negative value and its greatest positive 
value at the points $\ov x_{\a;-}$ and $\ov x_{\a;+}$, respectively,
where the points $\ov x_{\a;\pm}$ are defined in (\ref{4.58}).
Consequently, $\varphi_{\a,\b}$ being non positive, 
from $\psi_{\a,\b}=\varphi_{\a,\b}+\phi_{\a,\b}$ we get
\beqn\label{5.21}
&&\hskip -1truecm
\varphi_{\a,\b}(\ov x_\a)+\phi_{\a,\b}(\ov x_{\a;-})
\le\psi_{\a,\b}(x)\le\phi_{\a,\b}(\ov x_{\a;+}),\q\forall\,x\in[2x_0,0].
\eeqn
We stress that the previous arguments, and, in particular, 
formulae (\ref{5.19})--(\ref{5.21}) holds for every 
$\a\ge 1/2$ and not only for $\a\in(1/2,1)$.
Now, combining (\ref{5.21}) with Lemma \ref{lem4.8}{\it iii)}, i.e. with 
$h_{\a,\b}(2x_0)=h_{\a,\b}(0)\le h_{\a,\b}(x)$, $x\in[2x_0,0]$, 
$x_0\in[-D_{\a,\b},0)$, and $h_{\a,\b}(x_0)\le h_{\a,\b}(x)$, $x\in[2x_0,0]$,
$x_0<-D_{\a,\b}$, we obtain (\ref{4.67})--(\ref{4.69}).

\noindent
{\it iii) Case $\a=1$}. 
Letting $\a=1$ in (\ref{5.19})--(\ref{5.21})  we deduce
\beqn\label{5.22}
&&\hskip -1truecm
\varphi_{1,\b}(\ov x_1)+\phi_{\a,\b}(\ov x_{1;-})
\le\psi_{1,\b}(x)\le\phi_{1,\b}(\ov x_{1;+}),\q\forall\,x\in[2x_0,0].
\eeqn
Therefore, (\ref{4.70})--(\ref{4.72}) simply follow by combining (\ref{5.22})
with Lemma \ref{lem4.8}{\it i)}.

\noindent
{\it iv) Case $\a\in(1,2)\cup(2,+\infty)$}.
To obtain (\ref{4.73})--(\ref{4.75}) it suffices to combine (\ref{5.21}) 
with Lemma \ref{lem4.8}{\it ii)}, i.e.
with $h_{\a,\b}(x_0)\le h_{\a,\b}(x)$, $x\in[2x_0,0]$, $x_0\in[-M_{\a,\b},0)$,
and $h_{\a,\b}(x_{\a,\b;-})=h_{\a,\b}(x_{\a,\b;+})\le h_{\a,\b}(x)$,
$x\in[2x_0,0]$, $x_0<-M_{\a,\b}$.

\noindent
{\it v) Case $\a=2$}. 
Of course, taking $\a=2$ in (\ref{5.21}), we can still proceed as in the case {\it iv)}. 
On the other side, since when $\a=2$ we have $(2\a-1)/2=3/2$, 
the function $\psi_{\a,\b}=\varphi_{\a,\b}+\phi_{\a,\b}$ reduces to 
$\psi_{2,\b}(x)=2|\b|^{-1}[(3+2|\b|)x-2|\b|x_0][g_{2,\b}(x)]^{3/2}$
and we can obtain a more precise result.
Indeed, differentiating $\psi_{2,\b}$ with respect to $x$ and using 
$g_{2,\b}'(x)=\b(x-x_0)[g_{2,\b}(x)]^{-2}$ and $[g_{2,\b}(x)]^3=3|\b|x(2x_0-x)/2$, 
or, equivalently, summing up formulae (\ref{5.19}) and (\ref{5.20}) with
$\a=2$, we obtain
\beqn\label{5.23}
&&\hskip -1truecm
\psi_{2,\b}'(x)=-3[2(3+2|\b|)x^2-(9+8|\b|)x_0x+2|\b|x_0^2][g_{2,\b}(x)]^{-3/2},
\q x\in(2x_0,0).
\eeqn
Hence, from (\ref{5.23}) it follows that $\psi_{2,\b}$ is 
non decreasing in $[\widehat x_{\b;-},\widehat x_{\b;+}]$ and 
non increasing in $[2x_0,\widehat x_{\b;-}]\cup[\widehat x_{\b;+},0]$, 
where the points $\widehat x_{\b;+}$ are defined in (\ref{4.60}).
In addition, $x_\b\in(x_0,\widehat x_{\b;+})$ being defined in (\ref{4.60}),
$\psi_{2,\b}$ is negative in $(2x_0,x_\b)$, positive in $(x_\b,0)$
and vanishes at $x=2x_0$, $x=x_\b$ and $x=0$.
We thus have
\beqn\no
&&\hskip -1truecm
\Psi_{2,\b}^-(x)=
\left\{\!\!
\begin{array}{lll}
-\Psi_{2,\b}(x),\ x\in [2x_0,x_\b],
\\[2mm]
0,\ x\in[x_\b,0],
\end{array}
\right.
\q
\Psi_{2,\b}^+(x)=
\left\{\!\!
\begin{array}{lll}
0,\ x\in[2x_0,x_\b],
\\[2mm]
\Psi_{2,\b}(x),\ x\in [x_\b,0].
\end{array}
\right.
\eeqn
Let $x_0\in[-M_{2,\b},0)$ where $M_{2,\b}=|\b|/(2\sqrt 3)$. 
Due to Lemma \ref{lem4.8}{\it ii-a)}, we have
$h_{2,\b}(x_0)\le h_{2,\b}(x)$ for every $x\in[2x_0,x_\b]$ and
$h_{2,\b}(x_\b)\le h_{2,\b}(x)$ for every $x\in[x_\b,0]$. 
Therefore, from $\psi_{2,\b}(\widehat x_{\b;-})\le\psi_{2,\b}(x)<0$,
$x\in(2x_0,x_\b)$, and $0<\psi_{2,\b}(x)\le \psi_{2,\b}(\widehat x_{\b;+})$,
$x\in(x_\b,0)$, we find for every $x\in[2x_0,0]$
\beqn\no
&&\hskip -0,5truecm
[h_{2,\b}(x_0)]^{-1}\psi_{2,\b}(\widehat x_{\b;-})
\le-\Psi_{2,\b}^-(x)\le \Psi_{2,\b}(x)\le \Psi_{2,\b}^+(x)
\le [h_{2,\b}(x_\b)]^{-1}\psi_{2,\b}(\widehat x_{\b;+}).
\eeqn
This proves (\ref{4.76}) with $C_j(\b,x_0)$, $j=22,23$,
being defined by the firs expressions in (\ref{4.77}) and (\ref{4.78}).
Let us now assume $x_0<-M_{2,\b}$. 
According to what observed after the definition (\ref{4.61}) 
of the positive number $R_\b$, we distinguish the two sub-cases 
$x_0\in(-R_\b,-M_{2,\b})$ and $x_0\le -R_\b$, 
corresponding to $x_{2,\b;+}<x_\b$ and $x_\b\le x_{2,\b;+}$, respectively. 
Let first $x_0\in(-R_\b,-M_{2,\b})$. 
Since $x_{2,\b;+}<x_\b$ and $h_{2,\b}$ is increasing in $[x_{2,\b;+},0]$ 
by virtue of Lemma \ref{lem4.8}{\it ii-b)}, 
we have $h_{2,\b}(x_\b)\le h_{2,\b}(x)$ for $x\in[x_\b,0]$, 
whereas $h_{2,\b}(x_{2,\b;-})=h_{2,\b}(x_{2,\b;+})\le h_{2,\b}(x)$ 
for $x\in[2x_0,x_\b]$. Then, estimate (\ref{4.76}), 
with $C_{22}(\b,x_0)$ and $C_{23}(\b,x_0)$ being defined, respectively,
by the second expression in (\ref{4.77}) and the first expression in (\ref{4.78}),
follows from 
\beqn\no
&&\hskip -1truecm
[h_{2,\b}(x_{2,\b;-})]^{-1}\psi_{2,\b}(\widehat x_{\b;-})\le-\Psi_{2,\b}^-(x),\q 
\forall\,x\in[2x_0,x_\b],
\\[2mm]
&&\hskip -1truecm 
\Psi_{2,\b}^+(x)\le[h_{2,\b}(x_\b)]^{-1}\psi_{2,\b}(\widehat x_{\b;+}),\q
\forall\,x\in[x_\b,0].\no
\eeqn 
Instead, if $x_0\le -R_\b$, since $x_\b\le x_{2,\b;+}$,
from $h_{2,\b}(x_{2,\b;+})\le h_{2,\b}(x)$, $x\in[x_\b,0]$, 
and $h_{2,\b}(x_{2,\b;-})\le h_{2,\b}(x)$, $x\in[2x_0,x_\b]$, 
we deduce for every $x\in[2x_0,0]$
\beqn\no
&&\hskip -0,5truecm
[h_{2,\b}(x_{2,\b;-})]^{-1}\psi_{2,\b}(\widehat x_{\b;-})
\le -\Psi_{2,\b}^-(x)\le \Psi_{2,\b}(x)\le 
\Psi_{2,\b}^+(x)\le[h_{2,\b}(x_{2,\b;+})]^{-1}\psi_{2,\b}(\widehat x_{\b;+}).
\eeqn
This completes the proof of (\ref{4.76})--(\ref{4.78}).
\begin{remark}\label{rem5.3}
\emph{We observe that also in the case of Lemma \ref{lem4.13}
to find the greatest and least values of $\Psi_{2,\b}$
by locating its stationary points is not profitable. 
For, $\Psi_{2,\b}'(x)=[h_{\a,\b}(x)]^{-3}\wtil S_{\a,\b}(x)$, where
\beqn\no
&&\hskip -1truecm
\wtil S_{\a,\b}(x)=[\varphi_{\a,\b}'(x)+\phi_{\a,\b}'(x)][h_{\a,\b}(x)]^2
-\psi_{\a,\b}(x)\{1-\a|\b|^{-1}[g_{\a,\b}(x)]^{\a-1}\}(x-x_0),
\eeqn
$\varphi_{\a,\b}'$ and $\phi_{\a,\b}'$ being given by (\ref{5.19}) and (\ref{5.20}).}
\end{remark}
\noindent{\it Proof of Corollary \ref{cor4.16}.}
{\it i) Case $\a=1/2$}. 
If $x_0\in[-|\b|^{-1},0)$, then the definition $C_{24}(1/2,\b,x_0):=|C_{12}(\b,x_0)|$ 
is obvious, since $C_{12}(\b,x_0)\le\psi_{1/2,\b}(x)\le 0$, $x\in[2x_0,0]$. 
If $x_0<-|\b|^{-1}$, then, due to (\ref{4.65}) and (\ref{4.66}), 
we have first to compare the positive values $-\psi_{1/2,\b}(\wtil x_{\b;-})$
and $\psi_{1/2,\b}(\wtil x_{\b;+})$, the points $\wtil x_{\b;\pm}$ being
defined in (\ref{4.59}). Here, for brevity, we set
$\wtil x_{\b;\pm}=x_0\pm t_{\b,x_0}$, where 
$t_{\b,x_0}=(6+3|\b|^2x_0^2)^{1/2}/(3|\b|)$.
Since $\psi_{1/2,\b}(x)=-3|\b|^{-1}x\psi_\b(x)$
and $\psi_\b(\wtil x_{\b;\pm})=|\b|^2t_{\b,x_0}^2\mp|\b|^2x_0t_{\b,x_0}-2$,
where $\psi_\b$ is defined by (\ref{5.18}), we obtain
\beqn\label{5.24}
&&\hskip -1truecm
\left\{\!
\begin{array}{rll}
-\psi_{1/2,\b}(\wtil x_{\b;-})
&\!\!\!=\!\!\!&
3|\b|^{-1}(x_0-t_{\b,x_0})(|\b|^2t_{\b,x_0}^2+|\b|^2x_0t_{\b,x_0}-2),
\\[2mm]
\psi_{1/2,\b}(\wtil x_{\b;+})
&\!\!\!=\!\!\!&
-3|\b|^{-1}(x_0+t_{\b,x_0})(|\b|^2t_{\b,x_0}^2-|\b|^2x_0t_{\b,x_0}-2).
\end{array}
\right.
\eeqn
Therefore
\beqn\no
-\psi_{1/2,\b}(\wtil x_{\b;-})-\psi_{1/2,\b}(\wtil x_{\b;+})=-12|\b|^{-1}x_0>0,
\eeqn
i.e. $-\psi_{1/2,\b}(\wtil x_{\b;-})>\psi_{1/2,\b}(\wtil x_{\b;+})$.
Now, since $h_{1/2,\b}(2x_0)=h_{1/2,\b}(0)$, the inequality 
$|C_{14}(\b,x_0)|>C_{15}(\b,x_0)$ simply follows by observing that, 
if $x_0<D_{1/2,\b}$, then from Lemma \ref{lem4.8}{\it iii-b)} and (\ref{4.39}) 
we have $h_{1/2,\b}(x_0)<h_{1/2,\b}(x)$, $x\in[2x_0,0]\backslash\{x_0\}$,
and, in particular, $h_{1/2,\b}(x_0)<\min\{h_{1/2,\b}(x_\b),h_{1/2,\b}(0)\}$,
$x_\b\in(x_0,0)$ being defined in (\ref{4.60}). 

\noindent
{\it ii) Case $\a\in(1/2,1)$}. 
Observe that $\phi_{\a,\b}$, $\a>0>\b$,
is an odd function with respect to the line $x=x_0$. 
Therefore, the points $\ov x_{\a;\pm}$ in (\ref{4.58}) 
being symmetric with respect to $x=x_0$, 
we have $\phi_{\a,\b}(\ov x_{\a;-})=-\phi_{\a,\b}(\ov x_{\a;+})$.
Then, since the function $\varphi_{\a,\b}$, $\a>1/2$, $\b<0$, 
is negative in $(2x_0,0)$ we obtain
\beqn\label{5.25}
-[\varphi_{\a,\b}(\ov x_\a)+\phi_{\a,\b}(\ov x_{\a;-})]
>-\phi_{\a,\b}(\ov x_{\a;-})=\phi_{\a,\b}(\ov x_{\a;+})>0,\q \a>1/2,\ \b<0.
\eeqn
Using $h_{\a,\b}(2x_0)=h_{\a,\b}(0)$ and (\ref{5.25}),
from (\ref{4.68}) and (\ref{4.69}) it thus follows 
$|C_{16}(\a,\b,x_0)|>C_{17}(\a,\b,x_0)$,
which proves $C_{24}(\a,\b,x_0)=|C_{16}(\a,\b,x_0)|$.

\noindent
{\it iii) Case $\a=1$}. 
Due to definitions (\ref{4.71}) and (\ref{4.72}), the proof of 
$C_{24}(1,\b,x_0)=|C_{18}(\b,x_0)|>C_{19}(\b,x_0)$ 
is the same as that of the case $\a\in(1/2,1)$.

\noindent
{\it iv) Case $\a\in(1,2)\cup(2,+\infty)$}.
Due to definitions (\ref{4.74}) and (\ref{4.75}), to prove
$C_{24}(\a,\b,x_0)=|C_{20}(\a,\b,x_0)|>C_{21}(\a,\b,x_0)$ 
it suffices to use (\ref{5.25}) and
$h_{\a,\b}(x_{\a,\b;-})=h_{\a,\b}(x_{\a,\b;+})$.

\noindent
{\it v) Case $\a=2$}. 
As in the proof of Lemma \ref{lem4.13} here we have 
\beqn\label{5.26}
&&\hskip -1truecm
\psi_{2,\b}(x)=2|\b|^{-1}[(3+2|\b|)x-2|\b|x_0][g_{2,\b}(x)]^{3/2},\q x\in[2x_0,0].
\eeqn
Therefore, due to definitions (\ref{4.77}) and (\ref{4.78}) of 
$C_{22}(\b,x_0)$ and $C_{23}(\b,x_0)$ and since
for every $x\in[2x_0,0]$ it holds
$h_{2,\b}(x_0)\le h_{2,\b}(x)$ if $x_0\in[-M_{2,\b},0)$, 
and $h_{2,\b}(x_{2,\b;-})=h_{2,\b}(x_{2,\b;+})\le h_{2,\b}(x)$ if $x_0<-M_{2,\b}$, 
to prove $C_{24}(2,\b,x_0)=|C_{22}(\b,x_0)|$ it suffices 
to show $-\psi_{2,\b}(\widehat x_{\b;-})>\psi_{2,\b}(\widehat x_{\b;+})$.
Now, using definition (\ref{4.60}) of the points $\widehat x_{\b;\pm}$,
from (\ref{5.26}) we get
\beqn\label{5.27}
&&\hskip -1truecm
\left\{\!
\begin{array}{rll}
-\psi_{2,\b}(\widehat x_{\b;-})
&\!\!\!=\!\!\!&
(2|\b|)^{-1}\big\{[8(3+2|\b|)^2+9]^{1/2}+9\big\}
|x_0|[g_{2,\b}(\widehat x_{\b;-})]^{3/2}>0,
\\[2mm]
\hskip -1truecm
\psi_{2,\b}(\widehat x_{\b;+})
&\!\!\!=\!\!\!&
(2|\b|)^{-1}\big\{[8(3+2|\b|)^2+9]^{1/2}-9\big\}
|x_0|[g_{2,\b}(\widehat x_{\b;+})]^{3/2}>0.
\end{array}
\right.
\eeqn
Therefore, to show 
$-\psi_{2,\b}(\widehat x_{\b;-})>\psi_{2,\b}(\widehat x_{\b;+})$
it suffices to prove that $g_{2,\b}(\widehat x_{\b;-})>g_{2,\b}(\widehat x_{\b;+})$.
To this purpose, recalling that $\widehat x_{\b;-}\in(2x_0,x_0)$ and
that $g_{\a,\b}$, $\a>0>\b$, is an even function with respect to the line $\{x=x_0\}$,
we rewrite $g_{2,\b}(\widehat x_{\b;-})$ as $g_{2,\b}(\widehat x_{\b;-}^{\;*})$ 
where $\widehat x_{\b;-}^{\;*}=2x_0-\widehat x_{\b;-}$. 
Now, the points $\widehat x_{\b;-}^{\;*}$ and $\widehat x_{\b;+}$
belong to $(x_0,0)$ and their distance from $x=x_0$ is given, respectively, by
\beqn\no
&&\hskip -1truecm
\widehat x_{\b;-}^{\;*}-x_0
=\frac{\big\{[8(3+2|\b|)^2+9]^{1/2}-3\big\}|x_0|}{4(3+2|\b|)},
\qq
\widehat x_{\b;+}-x_0
=\frac{\big\{[8(3+2|\b|)^2+9]^{1/2}+3\big\}|x_0|}{4(3+2|\b|)}.\no
\eeqn
Hence, $x_0<\widehat x_{\b;-}^{\;*}<\widehat x_{\b;+}$ 
and since $g_{\a,\b}$,  $\a>0>\b$, is decreasing in $(x_0,0)$ 
we deduce $g_{2,\b}(\widehat x_{\b;-})=
g_{2,\b}(\widehat x_{\b;-}^{\;*})>g_{2,\b}(\widehat x_{\b;+})$. 
This completes the proof.
\section{Final remarks}\label{Sec6}
\setcounter{equation}{0}
A natural question arising from the assumptions of Theorem \ref{thm4.20}
is if there is some global regularity, expressed in terms of the weighted Sobolev 
space $\wtil W_{AC\cup\s}^1(\Om)$, which implies the needed boundary 
regularity for $u$. Unfortunately, and as we shall explain briefly in a while, 
there are very few expectations that $\Re u, \Im u\in\wtil W_{AC\cup\s}^1(\Om)$ 
is sufficient to ensure the boundary regularity required for $u=\Re u+i\Im u$, 
and one might hope that such boundary regularity is instead a straightforward 
consequence of $u$ being an eigenfunction of problem (\ref{4.86}).
Indeed, as observed in Section \ref{Sec3}, at the moment
for the space $\wtil W_{AC\cup\s}^1(\Om)$
the only known embedding result is the following: 
if $\Om$ is a $D$-star-shaped normal Tricomi domain, 
then $\wtil W_{AC\cup\s}^1(\Om)$ is continuously and compactly embedded in
$L^p(\Om)$ for every $p\in[1,p^*)$, where $p^*=2N(N-2)^{-1}$, $N=5/2$. 
Of course, such an embedding is not sufficient to guarantee the boundary
regularity required in the statement of Theorem \ref{thm4.20}, but
is probably the best result that one can obtain for the 
weighted Sobolev space $W_{AC\cup\s}^1(\Om)$. 
In fact, for weighted Sobolev spaces whose weights are functions $\s(x,y)$
more general than $|y|$, the continuous and compact embedding mentioned
before is usually not satisfied. 
In order to avoid the introduction of a complicate notation 
which would yield us out of the aims of this paper, 
we refer for instance to \cite[Sections 3.8.2 and 3.8.3]{Tri} 
where it is shown that the spaces 
$W^{m,p}(\Om;\rho^\mu;\rho^{\mu+mp})$, $\mu\in\rsp$, 
are not embedded in $L^p(\Om,\rho^\nu)$ if $\nu>\mu+mp$, 
whereas the embedding exists continuous, but not compact, if $\nu=\mu+mp$.
Here $\rho$ is a positive given function such that 
$1/\rho$ is a mollification of the distance function from the boundary
and $\Om$ is a $C^\infty$ bounded domain in $\rsp^n$, $n\in\nsp$. 
Clearly, the situation considered in \cite{Tri} is far to be our case, but it
is however significant, in the sense that for functions
belonging to weighted Sobolev spaces not too much regularity
should be expected.
 
Another question related to Theorem \ref{thm4.20} is
that if the upper bound (\ref{4.87}) is in any way connected
with some asymptotic behaviour of the distribution of the eigenvalues 
of the Tricomi operator. At the moment we do not know whether there is an
affirmative answer to this question, but there is some hope for it.
Indeed, the proof of many results which provide eigenvalue asymptotics
for elliptic and degenerate elliptic operators are based upon upper bound
estimates for $\l\|u\|_{L^2(\Om)}^2$ in terms of the square $L^2(\Om)$-norm
of the gradient of $u$. For instance, in \cite{KS} such an
argument, but with $L^2(\Om)$ being replaced by the weighted space
$L_\om^2(\Om)$ where $\om$ is a Muckenhoupt weight, is employed
for proving asymptotics eigenvalue of a degenerate elliptic operator
of second order with positive potential. The main difference
here is that on the right-hand side of (\ref{4.87}) we do not have
volume integrals of the first partial derivatives of $u$, but only their
surface integrals on the subset $BC\cup\s$ of $\partial\Om$. Of course,
here we have the problem that the Tricomi operator $T$ is hyperbolic
in the half-plane $y<0$, so the quadratic form associated with
$T$ is indefinite for $y<0$. This prevents us to find the lower bounds
for the eigenvalues which together with the quoted upper bounds
yield asymptotic results.

Observe also that $T$ restricted to $y\ge 0$
is a degenerate elliptic operator of second order with the coefficient
of the second partial derivative with respect to $x$ which is precisely
the distance of a point $(x,y)$, $y\ge 0$, from the degeneracy line $y=0$.
However, even with this restriction, $T$ does not belong to the classes
of Tricomi differential degenerate operators of first and second type considered
in \cite[Chapter 7]{Tri} and for which asymptotics eigenvalue are
established in \cite[Sections 7.8.2 and 7.8.3]{Tri}. 
Moreover, since the degeneracy line $y=0$ is a smooth manifold 
of codimension one in $\rsp^2$, we are prevented to apply to $T$, 
considered as a degenerate elliptic operator in $y\ge 0$, 
the result in \cite[Theorem 6]{Go} with $n=2$. 
Indeed, in \cite{Go} asymptotics of eigenvalues are proved for degenerate
elliptic operators whose coefficients depend on the distance of the points 
of $\Om\subset\rsp^n$ from a $m$-dimensional smooth manifold 
$\wtil\G_m\subset\partial\Om$ having codimension $n-m\ge 2$, 
so that, when $n=2$, $\wtil\G_m$ reduces to a discrete set of points.

We stress that, due to Lemma \ref{lem4.3}, 
the normal Tricomi domains $\Om_{\a,\b}$ are $D$-star-shaped for $\a\ge1/2$.
Then, $\wtil W_{AC\cup\s}^1(\Om_{\a,\b})$ is continuously and compactly
embedded in $L^p(\Om_{\a,\b})$, $p\in[1,p^*]$. 
Since the proofs in \cite{CE} for a non-linear perturbations 
of the Schr\"odinger operator rely on the compactness 
of the embedding of a certain weighted Sobolev space in an $L^p$ space, 
there is the possibility that, if some asymptotics eigenvalue 
for the Tricomi operator $T$ were known, then the compact embedding 
$\wtil W_{AC\cup\s}(\Om_{\a,\b})\hookto L^p(\Om_{\a,\b})$, 
$\a\ge1/2$, $p\in[1,p^*]$, could be applied in some way 
to prove asymptotics eigenvalue of non-linear perturbations of $T$. 
This suggests that the open problem of finding, if any, 
asymptotics eigenvalue of the Tricomi operator is a fundamental one 
which needs to be solved before perturbed Tricomi operators
are taken into account.
In this regards, we hope that our upper bounds (\ref{4.86}) could be 
a first step towards a complete comprehension of 
at least the real eigenvalues of $T$. We conclude by noticing
that the problem of establishing lower bounds for the eigenvalues of $T$
seems, at moment, a really hard one. In fact, if any,
neither an isoperimetric inequality nor a Cheeger constant 
are available for the Tricomi operator. This lack prevents one
to rely on many of the proofs which establish lower bounds
for the eigenvalue of the Laplacian with Dirichlet, Neumann or Robin boundary conditions (see, for instance, \cite{CO} and \cite{Ke}).
\vskip 0,4truecm
\centerline\acks
\vskip 0,2truemm
The author wish to thank Professor Mauro Garavello  of
the Universit\`a del Piemonte Orientale `A. Avogadro" for
having checked the computations in the proof of
Lemma \ref{lem4.8} and having him pointed out some
misprints in the manuscript.

\end{document}